\documentclass[11pt]{article}
	\usepackage{amsmath}
	\usepackage{amssymb}
	\usepackage{calligra}
	\usepackage{amsthm}
	\usepackage{graphicx}
	
	\usepackage[hidelinks]{hyperref}
	\usepackage{bm}
	\usepackage{url}
	\usepackage{float}
	\usepackage{cite}
	\usepackage{mathtools}
	\mathtoolsset{showonlyrefs=true}
	\usepackage{subcaption}
	\usepackage{mathrsfs}
	\usepackage[usenames]{color}
	\usepackage[table,xcdraw]{xcolor}
	\usepackage{comment}
	\theoremstyle{definition}
	\newtheorem{dfn}{Definition}[section]
	\newtheorem{thm}[dfn]{Theorem}
	\newtheorem{prop}[dfn]{Proposition}
	\newtheorem{lem}[dfn]{Lemma}
	\newtheorem{cor}[dfn]{Corollary}
	\newtheorem{rem}[dfn]{Remark}

	\newtheorem{hyp}{Hypothesis}
	\newcommand{\R}{\mathbb{R}}
	
	\newcommand{\Z}{\mathbb{Z}}
	\newcommand{\N}{\mathbb{N}}
	\newcommand{\C}{\mathbb{C}}
		\newcommand{\overbar}[1]{\mkern 1.5mu\overline{\mkern-1.5mu#1\mkern-1.5mu}\mkern 1.5mu}
	\newcommand{\bx}{\bar{x}}
	\newcommand{\bz}{\bar{z}}
	\newcommand{\by}{\bar{y}}

	\newcommand{\fft}{\mathrm{F}\mathrm{F}\mathrm{T}}

	\newcommand{\bydef}{\,\stackrel{\mbox{\tiny\textnormal{\raisebox{0ex}[0ex][0ex]{def}}}}{=}\,}

	\newcommand{\fix}{\mathrm{F}\mathrm{i}\mathrm{x}\hspace{.16667em plus .08333em}}

	\newcommand{\mi}{\mathrm{i}}
	\newcommand{\mx}{\mathrm{x}}

	\numberwithin{equation}{section}
	\usepackage[top=1in, bottom=1in, left=1in, right=1in]{geometry}

	\DeclareMathOperator{\sign}{sign}
	\DeclareMathOperator{\Real}{Re}
	\DeclareMathOperator{\Imag}{Im}
	\renewcommand{\Re}{\text{Re}}
	\renewcommand{\Im}{\text{Im}}
	\definecolor{orange-red}{rgb}{1.0, 0.27, 0.0}

\usepackage{calligra}

\DeclareMathAlphabet{\mathpzc}{OT1}{pzc}{m}{it}
	\graphicspath{{./Figures/}}

\begin{document}

	\title{
	From heteroclinic loops to homoclinic snaking in reversible systems: \\ rigorous forcing through computer-assisted proofs
	}
	\author{
		Jan Bouwe van den Berg \thanks{VU Amsterdam, Department of Mathematics, De Boelelaan 1081, 1081 HV Amsterdam, The Netherlands (\texttt{janbouwe@few.vu.nl})
		}
		\and 
		Gabriel William Duchesne \thanks{Department of Mathematics and Statistics, McGill University, 805 Sherbrooke West, Montreal, QC H3A 0B9, Canada (\texttt{gabriel.duchesne@mail.mcgill.ca})
		}
		\and
		Jean-Philippe~Lessard \thanks{Department of Mathematics and Statistics, McGill University, 805 Sherbrooke West, Montreal, QC H3A 0B9, Canada (\texttt{jp.lessard@mcgill.ca})
		}
	}
	\date{}
	\maketitle
	
	\begin{abstract}

Homoclinic snaking is a widespread phenomenon observed in many pattern-forming systems. Demonstrating its occurrence in non-perturbative regimes has proven difficult, although a forcing theory has been developed based on the identification of patterned front solutions. These heteroclinic solutions are themselves challenging to analyze due to the nonlinear nature of the problem. In this paper, we use computer-assisted proofs to find parameterized loops of heteroclinic connections between equilibria and periodic orbits in time reversible systems. This leads to a proof of homoclinic snaking in both the Swift–Hohenberg and Gray–Scott problems. Our results demonstrate that computer-assisted proofs of continuous families of connecting orbits in nonlinear dynamical systems are a powerful tool for understanding global dynamics and their dependence on parameters.
	\end{abstract}
	
	\begin{center}
	{\bf \small Keywords} \\ \vspace{.05cm}
	{\small Homoclinic snaking $\cdot$ Heteroclinic connections $\cdot$ Time-reversible systems $\cdot$ Computer-assisted proofs $\cdot$ Swift-Hohenberg equation $\cdot$ Gray-Scott equation}
	\end{center}
	
	\begin{center}
	{\bf \small Mathematics Subject Classification (2020)} \\ \vspace{.05cm}
	{\small 34C37 $\cdot$ 37M21 $\cdot$ 35B36 $\cdot$ 65G40 $\cdot$  65T40 $\cdot$ 42A10 $\cdot$ 37C79}
	\end{center}
	
	
\section{Introduction} \label{sec:intro}

Spatially localized patterns whose bifurcation diagrams exhibit an oscillatory structure known as \textit{snaking}, have been observed in a wide range of physical systems and model equations, see \cite{MR2665448,Burke2006,Coullet2000, Pomeau1986, Woods1999}. Although this phenomenon is more widespread, to simplify the discussion and to make our computations effective, we focus on the case of {\em time reversible} systems in $\R^4$:
\begin{equation} \label{eq:general_system}
	U' = f(U , \alpha) \in \R^4  ,
\end{equation} 
where $\alpha \in \R$ is a parameter. 
In that setting, the snaking patterns correspond to orbits which are homoclinic to an equilibrium, say the origin, and which visit a small neighborhood of a periodic orbit $\Gamma$. This implies that these homoclinic orbits have a long central segment where they oscillate. We follow~\cite{Aougab2017} and denote the approximate length of this oscillatory segment by $2L$. When varying the parameter~$\alpha$ in the system and following the branch of these homoclinic orbits in the bifurcation diagram, one observes oscillations in the parameter value, whereas $L$ grows to infinity: the bifurcation curve is said to \emph{snake}. This is depicted in Figure \ref{fig:pattern_and_snaking}. For very large $L$, the authors of \cite{ Beck2009, Aougab2017 } interpret these localized patterns as concatenations of front and back interfaces connecting a periodic structure $\Gamma$ to the origin. Under appropriate assumptions on the geometry and transversality of the stable manifold of the equilibrium and the unstable manifold of the periodic orbit, they prove the existence of infinitely many localized patterns within a bounded parameter interval for systems on $\R^4$ that are either reversible or Hamiltonian (or both). 
Essentially they prove a forcing theorem, namely, they  
identify explicit conditions on a family of \emph{heteroclinic} orbits between equilibria and periodic orbits (Figure \ref{fig:pattern_hetero}) that \emph{force} the occurrence of either an infinite snake or an infinite stack of isolas of \emph{homoclinic} orbits (Figure \ref{fig:pattern_and_snaking}). 


\begin{figure}
	\center
	\includegraphics[scale = 0.35]{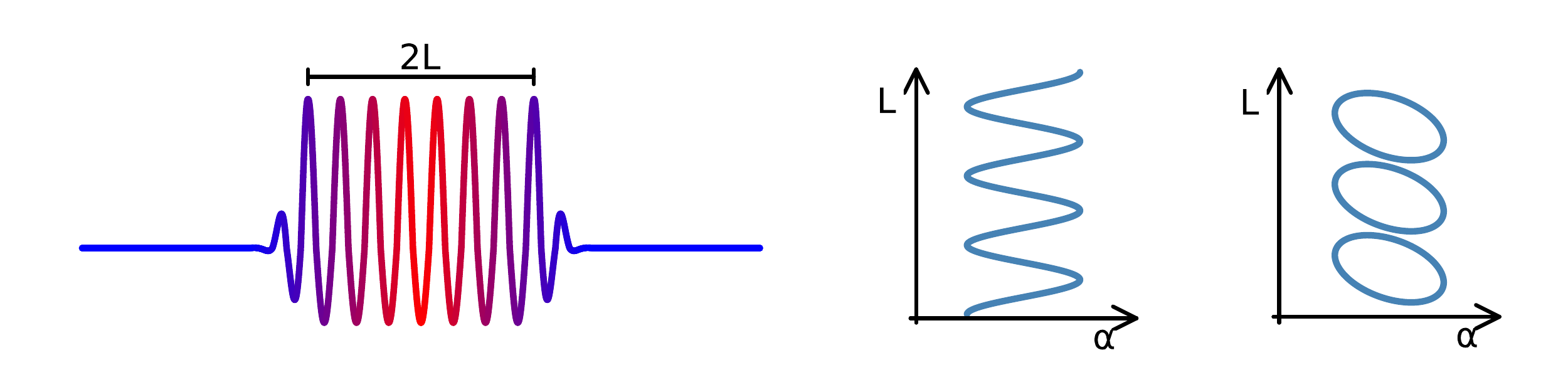} 
	\vspace{-.5cm}
	\caption{{\small Sketch of spatially localized patterns homoclinic to zero which spend time $2L$ in the neighborhood of a periodic orbit of \eqref{eq:general_system} (left panel), along with the corresponding bifurcation diagrams exhibiting snaking behavior (center panel) and disconnected infinite stacked isolas (right panel). This figure mimics~\cite{Aougab2017}.}} \label{fig:pattern_and_snaking}
\end{figure}
\begin{figure}
	\center
	\includegraphics[scale = 0.20]{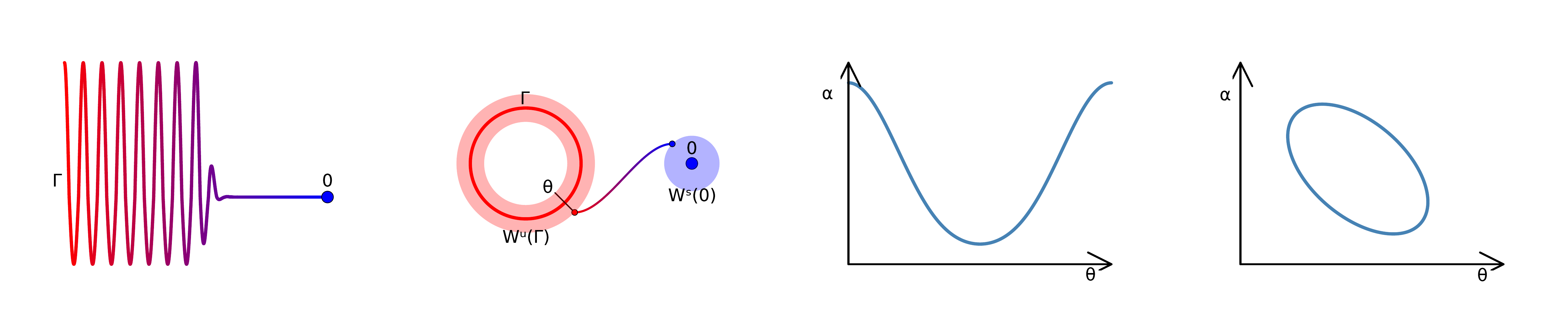} 
	\vspace{-1cm}	
	\caption{{\small Shown are graphs of the heteroclinic orbit connecting the equilibrium at the origin to the periodic orbit $\Gamma$ of \eqref{eq:general_system} (first panel), along with a representation of this connection via a finite orbit segment bridging the unstable and stable manifolds (second panel). The last two panels depict the continuation diagram of the heteroclinic orbit, visualized through the variation of the parameter $\alpha$ versus the angle $\theta$ that characterizes the connection between the unstable manifold $ W^u (\Gamma)$ and the stable manifold $W^s(0)$. These final panels illustrate the cases of snaking and isolas, respectively.}} \label{fig:pattern_hetero}
\end{figure}

While \cite{ Beck2009, Aougab2017 } thus provides a very strong forcing result, verifying the hypotheses on the heteroclinic orbits for a given nonlinear systems is challenging, as these involve specific geometric and dynamical properties of the manifolds involved.
In this paper, we present a general computer-assisted method for validating the necessary conditions on the heteroclinic orbits. We focus specifically on the family of second-order systems given by
\begin{equation}\label{Second Order System} 
	\left\{
	\begin{aligned}
		u''(t) &= {f_1}(u(t),v(t), \alpha) \\ 
		v''(t) &= {f_2}(u(t),v(t), \alpha),
	\end{aligned}
	\right. 
\end{equation}
where ${f_1,f_2}$ are polynomial functions in $u$, $v$ and $\alpha$. This class of systems includes several well-known models that exhibit snaking behavior, such as the Swift–Hohenberg equation \cite{Burke2006} and the Gray–Scott equation \cite{ Gandhi2018, AlSaadi2021}. 
Systems of the form \eqref{Second Order System} are reversible. Depending on the nonlinearities $f_1$ and $f_2$ the system can also be Hamiltonian, but we do not use that in our analysis in any way. The restriction to polynomial nonlinearities is purely to limit technicalities and we refer to~\cite{bergbredenlessardmireles} for the details of how to extend the construction to analytic nonlinearities.

The method developed in this work, along with the accompanying code, is designed to be general and applicable to any system of the form \eqref{Second Order System}. Our approach is to reformulate the existence of a heteroclinic orbit as a zero-finding problem on a Banach space, and to apply a Newton–Kantorovich method to rigorously validate the solution near a numerical approximation. The heteroclinic orbit is computed by locally parametrizing the stable manifold of the equilibrium $W_{\rm loc}^s(0)$ \cite{Cabre2003,Cabre2003a,Cabre2005} and the unstable manifold of the periodic orbit $W_{\rm loc}^u(\Gamma)$ \cite{Castelli2015,Castelli2017}, and subsequently connecting these two manifolds \cite{Berg2018}. Moreover, we extend this validation using a rigorous pseudo-arclength continuation method \cite{Berg2010,Breden2023} to demonstrate that the bifurcation diagram of the heteroclinic connections is periodic (in the sense of either of the two diagrams on the right in Figure~\ref{fig:pattern_hetero}). The zero-finding problem is constructed so that all the necessary conditions for the existence of snaking, as formulated in \cite{ Beck2009, Aougab2017}, are satisfied for systems of the form \eqref{Second Order System}.
We present results proving both snaking and isola structures in the Swift–Hohenberg equation (which is a special case of~\eqref{Second Order System} with $f_1=v$), as well as snaking behavior in the Gray–Scott model. The details of these applications, which  demonstrate the robustness and generality of our method for systems of the form \eqref{Second Order System}, are provided in Section~\ref{sec:application}. 

Before diving into technicalities, we now outline the main construction from~\cite{Beck2009, Aougab2017}, restricted to the context of the system~\eqref{Second Order System}. This allows us to set the stage for the branch of heteroclinic connections from the periodic orbit $\Gamma$ to the equilibrium $0$, which serves as the main ingredient for the snaking (or isolas) forcing theorem, and for which we develop a computer-assisted proof in the remainder of the paper. 


We start by rewriting~\eqref{Second Order System} as the first-order system
\begin{equation}\label{First Order System}
u'(t) = 
\begin{pmatrix} u'_1(t) \\ u'_2(t) \\ u'_3(t) \\ u'_4(t) \end{pmatrix}
=  
\begin{pmatrix} u_2(t) \\ f_1(u_1(t),u_3(t), \alpha ) \\ u_4(t) \\ f_2(u_1(t),u_3(t), \alpha) \end{pmatrix}
\bydef 
	\begin{pmatrix}
		\Psi_1(u) \\
		\Psi_2(u, \alpha) \\
		\Psi_3(u) \\
		\Psi_4(u, \alpha) 
	\end{pmatrix} = \Psi(u , \alpha) ,
\end{equation}
with  $u = ( u_1,u_2,u_3,u_4 )$. This system is reversible: the linear map $$\mathcal{R}:( u_1,u_2,u_3,u_4 ) \mapsto ( u_1,-u_2,u_3,-u_4 )$$ satisfies 
 $\mathcal{R}^2 = 1$ and $ \dim \fix \mathcal{R} = 2$, and it interacts with the vector field through $\Psi\left( \mathcal{R}u, \alpha \right) = - \mathcal{R} \Psi(u,\alpha)$ for all $(u,\alpha)$.
\begin{dfn} \label{def:symmetrix}
A solution $u(t)$ of \eqref{First Order System} is said to be {\em symmetric} if $u(0) \in \fix \mathcal{R}$. 
\end{dfn}

For simplicity, we will assume that, possibly after a change of variable,  $f_1(0,0,\alpha)=f_2(0,0,\alpha)=0$. This implies that the origin is an equilibrium of the vector field $\Psi$. In order for homoclinic snaking to be possible, we are naturally led to assume that this equilibrium is hyperbolic.
\begin{hyp}\label{hyp:hyperbolic}
	The origin $u = 0$ is a hyperbolic equilibrium: $\Psi(0,\alpha) = 0$ for all $\alpha$ and $ D_u\Psi (0,\alpha)$ has two eigenvalues with strictly negative real part and (by reversibility) two eigenvalues with strictly positive real part.
\end{hyp}
Many computational methods have been developed to rigorously compute eigenvalues and eigenvector of matrices (see \cite{Yamamoto1980,Yamamoto1982,Rump2001,Hladik2010,Castelli2011}). In our setup, the computation of these eigenvalues and associated eigenvectors will be part of the zero-finding formulation used to compute the local parameterization of the stable manifold $W_{\rm loc}^s(0)$. Details are provided in Section~\ref{sec:maps}.

To define the (nondegenerate) heteroclinic connections between the symmetric periodic orbits and the origin, which form the basis for the forcing theorem, we introduce $S^1 \bydef \R / (2\pi\Z)$ and the product set $\mathcal{X} \bydef \R \times \R^+ \times C^1(S^1,\R^4) \times C^1(\R,\R^4)$, see also~\cite{sandstede_notes_2025}.
\begin{dfn}\label{dfn:patternedfront}
	An element $x^* = ( \alpha^* , \omega^* , \gamma^*, u_{\text{het}}^* ) \in \mathcal{X}$ is called a \textit{patterned front} if 
	\begin{enumerate}
		\item 
		the function $u_{\text{per}}^*(t) \bydef \gamma^* (\omega^* t)$ is a \emph{symmetric}  periodic orbit of \eqref{First Order System} for $\alpha=\alpha^*$ with two  Floquet multipliers at either $e^{\pm 2\pi \lambda^*}$ or $-e^{\pm 2\pi \lambda^*}$ with $\lambda^* >0$. In particular, it has period $2\pi/\omega^*$ and $u_{\text{per}}^*(0) \in \fix \mathcal{R}$ as well as $u_{\text{per}}^*(\pi/\omega^*) \in \fix \mathcal{R}$. 
		\item
		the function $u_{\text{het}}^*(t)$ is a solution of \eqref{First Order System} for $\alpha = \alpha^* $ such that $\lim_{t\to\infty}u_{\text{het}}^*(t) =0$ and 
there is a phase $\theta^* \in S^1$ such that $| u_{\text{het}}^*(t) - \gamma^*( \omega^*t + \theta^*) | \rightarrow 0$ as $t \rightarrow - \infty$.
	\end{enumerate}
\end{dfn}
\begin{rem}
	The sign of the Floquet multipliers determines the orientability of the unstable bundle of the periodic orbit: positive multipliers imply orientability, while negative multipliers correspond to a non-orientable bundle. 
	\end{rem}
\begin{rem}
While the time shift invariance has been lifted for the periodic orbit by requiring that $\gamma^*(0) \in \fix \mathcal{R}$, for the heteroclinic orbit the phase condition has not been made explicit. From now on we assume that a phase condition has been chosen. One choice is $|u_{\text{het}}^*(0)|=\delta$ for some small value $\delta>0$, although later on we settle on a different somewhat more natural choice.
In any case, this then fixes the value of $\theta^*$ up to multiples of $2\pi$.
\end{rem}


To compute patterned front solutions, we will formulate a zero-finding problem that simultaneously encodes for the periodic solution $\gamma^*$, the frequency $\omega^*$, the angle $\theta^*$ and the Floquet exponent $\lambda^*$. We expand the components of $\gamma^* = (\gamma^*_1,\gamma^*_2,\gamma^*_3,\gamma^*_4)$ using a cosine Fourier series for $\gamma^*_1$ and $\gamma^*_3$ and a sine Fourier series for $\gamma^*_2$ and $\gamma^*_4$, ensuring that $ \gamma^*(0) \in \fix \mathcal{R} $ by construction. More details on solving for $\theta^*$ and $\lambda^*$ are given in Section \ref{sec:maps} when defining the connecting segment of $u_{\text{het}}^*$ and the unstable manifold of $\gamma^*$.

What we need for the forcing theorem is not just a single patterned front, but a one-parameter family, a loop, of patterned fronts. Moreover, this family is required to be ``in general position'' for the forcing theorem to be applicable. We will make this precise below. Fortunately, these conditions are also natural ones for our computer-assisted proof.

First, the equilibrium at the origin is assumed to be hyperbolic, see Hypothesis~\ref{hyp:hyperbolic}. The non-degeneracy condition on the Floquet multipliers in Definition~\ref{dfn:patternedfront} says that $u^*_{\text{per}}$ is a symmetric \emph{hyperbolic} periodic orbit of \eqref{First Order System} for $\alpha = \alpha^*$. 
This implies that, for fixed parameter values, it is locally part of a smooth one-parameter family $u_{\text{per}}(\kappa)$ of symmetric periodic orbits parametrized by
$|\kappa|<1$, and $u^*_{\text{per}}=u_{\text{per}}(0)$. Since we are essentially following orbits in the extended phase space of the pair $(u,\alpha) \in \R^4 \times \R \cong \R^5$, we consider this one-parameter family as part of a smooth two-parameter family $u_{\text{per}}(\kappa,\alpha)$ for $\alpha$ near $\alpha^*$. Following~\cite{sandstede_notes_2025}, we define the union of unstable manifolds, which can also be interpreted as the center-unstable manifold in the extended phase space,
\begin{equation*}
	\widetilde{W}_{\text{loc}}^{cu}(u_{\text{per}}^*; \alpha^*) \bydef \bigcup_{|\alpha-\alpha^*|<\epsilon, |\kappa|<1}
	W^{u}(u_{\text{per}}(\kappa;  \alpha))  \subset \R^4 \times \R ,	
\end{equation*}
for some small $\epsilon>0$.
Similarly, near the origin we define
\begin{equation*}
	\widetilde{W}_{\text{loc}}^{cs}(0; \alpha^*) \bydef \bigcup_{|\alpha-\alpha^*|<\epsilon}
	W^{s}(0;  \alpha)  \subset \R^4 \times \R .
\end{equation*}
Counting dimensions, we infer that the sets $\widetilde{W}_{\text{loc}}^{cu}(u_{\text{per}}^*; \alpha^*) $ and $\widetilde{W}_{\text{loc}}^{cs}(0; \alpha^*) $ are smooth 4-dimensional and 3-dimensional manifolds, respectively, consisting of orbits of the vector field. We note that the localization for these sets is in the parameter directions ($\alpha$ and $\kappa$), but not in the flow direction. Connecting orbits (patterned fronts) can thus be interpreted as intersections of these manifolds. If these manifolds intersect \emph{transversely} in a 2-dimensional smooth manifold, then, after lifting the time translation invariance, this results locally in a one-parameter family of patterned fronts.
This leads to the following concept of a \emph{regular} patterned front.
\begin{dfn} \label{defn:regular patterned front}
An element $x^* = ( \alpha^* , \omega^* , \gamma^*, u_{\text{het}}^* ) \in \mathcal{X}$ is called a {\em regular patterned front} if it is a patterned front and the manifolds $\widetilde{W}_{\text{loc}}^{cu}(u_{\text{per}}^*; \alpha^*)$ and $\widetilde{W}_{\text{loc}}^{cs}(0; \alpha^*)$ intersect transversely in $\R^5$ along the heteroclinic orbit $(u^*_{\text{het}};\alpha^*)$.
\end{dfn}

The condition that forces the homoclinic snaking/stacked isolas is the existence of a parametrized closed curve of regular patterned front~\cite{sandstede_notes_2025}.
\begin{hyp} \label{hyp:regular_patterned_front}
	There is a function $\mathcal{S}: \R / \Z \longrightarrow \mathcal{X}$ defined by
	\begin{align}\label{patterned_front}
		\mathcal{S}(s) = \left( \alpha(s) , \omega(s) , \gamma(\cdot, s) , u_{\text{het}}(\cdot,s) \right)
	\end{align}
	of class $C^1$ such that $\mathcal{S}(s)$ is a regular patterned front for each $s$, with $\mathcal{S}_s(s) \neq 0$ for all $s$.
\end{hyp}

In our rigorous computational framework, continuation along the branch of heteroclinic solutions is realized by extending the zero-finding map $F$ for patterned fronts through a pseudo-arclength continuation method.
By choosing judiciously the extra equation for the parameterization of the curve in the pseudo-arclength continuation method, the non-degeneracy condition $\mathcal{S}_s(s) \neq 0$ in Hypothesis~\ref{hyp:regular_patterned_front} is verified a posteriori using interval arithmetic, see Corollary~\ref{cor_cond_smooth}. The property that the patterned fronts we find are regular, i.e.\ transversality, is proven as part of the Newton-Kantorovich argument, see Section~\ref{sec:newtonkantorovich} and Remark~\ref{rem:transversality}. 

As shown in \cite{ Beck2009, Aougab2017, sandstede_notes_2025 }, the dichotomy between snakes and isolas is determined by the global behavior of the solution’s phase as it is tracked along the loop of patterned fronts. For the phase $\theta^*$ in Definition~\ref{dfn:patternedfront} to be well-defined, one needs both to fix consistently (along the loop) the shift invariance of the heteroclinic orbit \emph{and} to use the symmetry  of the periodic orbit (fixing its phase). In this paper we define the phase $\theta^*$ in a manner that fits well with the computer-assisted proof. Namely, after constructing the unstable manifold of the symmetric periodic orbit by using a Fourier-Taylor series based on the Parameterization Method, we read off the phase of the heteroclinic orbit at a fixed distance from the periodic orbit along its unstable manifold. In view of the invariance equation governing the parameterization, this leads both to a natural distance and a natural phase $\theta^*(s)$ for $s \in \R / \Z$, which can be traced consistently along the loop of patterned fronts. We refer to Section~\ref{sec:connectingorbit} and Remark~\ref{rmk:phase_orbit} for the details of the construction. 
Next, let the continuous function $\tilde{\theta}: [0,1] \rightarrow \R$ be the unique lifting of $\theta^*: [0,1] \rightarrow S^1$ satisfying $\tilde{\theta}(0) = \theta^*(0)$.  Clearly $\tilde{\theta}(1) \equiv \tilde{\theta}(0) \mod 2\pi$. In our computer-assisted proof we can trace the value of $\tilde{\theta}$ around the loop.

This concludes the description of the loop of heteroclinic orbits representing patterned fronts, whose existence and properties are established in the present work. These orbits serve as the forcing mechanism for the homoclinic localized states studied in~\cite{Beck2009, Aougab2017, sandstede_notes_2025}. In this paper, we provide only a brief overview of the relevant forcing results and refer the reader to~\cite{Beck2009, Aougab2017}, and in particular to Theorem 1 in \cite{sandstede_notes_2025} (which does not rely on a Hamiltonian structure) for a detailed description. The relevant homoclinic orbits shadow the concatenation of the three orbits $\mathcal{R} u^*_{\text{het}}$, $u^*_{\text{per}}$ and $u^*_{\text{het}}$. The time $2L$ which the homoclinics spend near $u^*_{\text{per}}$ can be made precise, and the forcing result concerns the asymptotics for $L \to \infty$. 
All the relevant homoclinic orbits cross $\fix \mathcal{R}$ either near $u^*_{\text{per}}(0)$ or near $u^*_{\text{per}}(\pi/\omega^*)$, thus distinguishing two types of localized patterned states.
Assuming the existence of a loop of regular patterned fronts, the results in~\cite{ Beck2009, Aougab2017,sandstede_notes_2025 } on localized homoclinic patterns can be summarized in the following forcing theorem, depending on the homotopy type of the phase curve $\tilde{\theta}$.
\begin{thm}[{\bf Forcing Theorem}] \label{thm:forcing_theorem} Assume Hypotheses~\ref{hyp:hyperbolic} and~\ref{hyp:regular_patterned_front} are satisfied, and let $\tilde{\theta}$ be the lifting of the phase defined above.
If $\tilde{\theta}(1) \neq \tilde{\theta}(0)$, then for large $L$ there are localized patterned state in \eqref{Second Order System} whose $(\alpha,L)$ bifurcation diagram is described by two interlaced snaking curves. If $\tilde{\theta}(1) = \tilde{\theta}(0)$, then for large $L$ there are localized patterned states in \eqref{Second Order System} whose $(\alpha,L)$ bifurcation diagram is described by an infinite stack of isolas.
\end{thm}

The goal of this work is to develop a framework that enables the verification of the Forcing Theorem~\ref{thm:forcing_theorem} using computer-assisted methods, including proving the regularity/transversality of the patterned fronts and a rigorous evaluation of $\tilde{\theta}(1)-\tilde{\theta}(0)$, applicable across all problems within the family of systems \eqref{Second Order System}. The loop of heteroclinics forces an asymptotic part of the homoclinic snake with an infinite number of folds. By methods similar to those in this paper, establishing a finite part of the snake can also be addressed directly, although this is outside the current scope.  From a more global perspective, the paper illustrates that rigorous continuation of complex connecting orbits in nonlinear systems, as well as obtaining both qualitative and quantitative information about the orbits, is feasible by using state-of-the-art computer-assisted proof techniques. 


The paper is organized as follows.
In Section~\ref{sec:maps}, we define in \eqref{eq:zero_finding_f} the zero-finding map $f$ for the heteroclinic orbits. Each  orbit connects the periodic orbit $u_{\text{per}}$ to $0$ and is constructed (and proved) by combining parameterizations of the local unstable manifold of $u_{\text{per}}$ and the local stable manifold of $0$. These manifolds are represented using Fourier–Taylor and Taylor–Taylor expansions respectively. The boundary value problem for the orbit segment connecting these local manifolds is solved by means of a Chebyshev series expansion. This zero-finding problem is supplemented with a rigorous pseudo-arclength continuation method that allows us to extend solutions to continuous branches of solutions with varying parameter $\alpha$. Section \ref{sec:bounds} introduces the Newton–Kantorovich theorem and outlines our computer-assisted proof strategy for rigorously validating the existence of the heteroclinic orbit near a numerical approximation. We detail the construction of the necessary Banach spaces and their Banach algebra structures, the finite-dimensional projections, and the computation of operator norms and convolution bounds that will be needed for our validation. Finally, in Section \ref{sec:application}, we discuss the computational challenges encountered in applying this framework, including memory optimization and error control across the parameter interval. We present computer-assisted proofs validating both snaking and isola structures in the Swift–Hohenberg equation, as well as snaking behavior in the Gray–Scott model. These applications demonstrate the robustness and generality of our method for systems of the form \eqref{Second Order System}.

	
\section{Zero-finding problem for the regular patterned fronts} \label{sec:maps}

In this section, we formulate a zero-finding problem (see Equation~\eqref{eq:zero_finding_F}) whose solutions correspond to regular patterned fronts, as defined in Hypothesis~\ref{hyp:regular_patterned_front}. As regular patterned fronts are heteroclinic connections defined over unbounded time domains, compactifying the domain is essential for a well-posed numerical approach. To this end, we recast the heteroclinic connection as a projected boundary value problem (BVP), where the boundary conditions are specified to lie in the local unstable manifold $W_{\rm loc}^u(\Gamma)$ of the periodic orbit $\Gamma$ and the local stable manifold $W_{\rm loc}^s(0)$ of the origin, respectively. In Section~\ref{sec:unstable_manifold_PO}, we use the parameterization method to compute $W_{\rm loc}^u(\Gamma)$, deriving a sequence of intermediate problems to obtain a Fourier-Taylor expansion of the manifold. Section~\ref{sec:stable_manifold_origin} introduces a corresponding zero-finding formulation for the parameterization of the stable manifold $W_{\rm loc}^s(0)$. Finally, we compute an orbit segment connecting these two manifolds. By casting the connection as a BVP, we ensure that the resulting trajectory approximates a segment of the desired heteroclinic orbit. Throughout this section, the parameter~$\alpha$ is assumed to be fixed.

\subsection{Local unstable manifold of the periodic solution} \label{sec:unstable_manifold_PO}

To find the unstable manifold of the periodic orbit, we develop an approach based on the parameterization method to obtain a rigorous enclosure of the local unstable manifold associated with a periodic orbit. The parameterization method is a functional-analytic framework that constructs invariant manifolds by seeking a mapping whose image satisfies an invariance equation. This approach provides a systematic and often rigorous means of computing local stable and unstable manifolds near fixed points or periodic orbits. For a more comprehensive treatment of the parameterization method, we refer the reader to \cite{Cabre2003,Cabre2003a,MR3467671,MR2299977,MR2240743}, and to \cite{Cabre2005,Castelli2015,Castelli2017,MR2551254,MR3118249,MR2851901} for its application in the context of periodic orbits.

We begin by rescaling system~\eqref{First Order System} so that the periodic solution becomes $2\pi$-periodic. Let $\tau$ denote the period of the periodic orbit for a fixed value of the parameter $\alpha$. We perform the time rescaling $t \mapsto \frac{2\pi}{\tau} t$, and define $\omega = \frac{\tau}{2\pi}$. Under this change of variables, system~\eqref{First Order System} becomes
 \begin{equation} \label{rescale first order system}		
 u'_i = \omega \Psi_i(u, \alpha), \quad \text{for } i = 1,2,3,4,
 \end{equation}
where $\omega$ is now treated as an additional unknown. Its value will be determined later when we construct the connecting orbit. 

The \emph{parameterization method} provides a functional framework for computing the local unstable manifold of the periodic solution by expressing it as the image of a smooth map $W: S^1 \times \R \rightarrow \R^4$ that satisfies the conditions
\begin{equation} \label{Para diff equa_first_order}
W(\theta, 0) = \gamma(\theta), \qquad \frac{\partial}{\partial \sigma} W(\theta, 0) = v(\theta),
\end{equation}
and solves the \emph{invariance equation}
\begin{equation} \label{Para diff equa}
\frac{\partial}{\partial \theta} W(\theta, \sigma) + \lambda \sigma \frac{\partial}{\partial \sigma} W(\theta, \sigma) = \omega \Psi(W(\theta, \sigma), \alpha),
\end{equation}
where $\gamma$ denotes the $2\pi$-periodic solution of~\eqref{rescale first order system}, $\lambda$ is the Floquet exponent, $v$ is the eigenvector field spanning the unstable bundle along $\gamma$, and $\omega$ is a frequency scaling parameter. The map $W$ thus encodes both the periodic orbit and its associated unstable dynamics, with the zeroth and first-order terms in $\sigma$ corresponding to $\gamma$ and $v$, respectively. Hence, to compute $ W^u(\Gamma) $, we proceed by first solving for $ \gamma $, $ \lambda $, $ \omega $, and $ v $. In Section~\ref{sec:PO}, we introduce in~\eqref{eq:map_periodic_sol} a zero-finding formulation whose solution yields the periodic orbit, which defines the zeroth-order term in the parameterization. This is followed in Section~\ref{sec:unstable_bundle_PO} by the computation of the eigenpair $(v,\lambda)$, obtained as the solution to the zero-finding problem defined in \eqref{eq:bundle_map}, corresponding to the first-order term. Finally, in Section~\ref{sec:high-order-term_WuGamma}, we derive in~\eqref{eq:high_order_term_WuGamma} a recursive system whose solution provides the higher-order coefficients of the parameterization, thereby completing the construction of a rigorous enclosure of the local unstable manifold $ W^u_{\mathrm{loc}}(\Gamma) $.


\subsubsection{Periodic solution} \label{sec:PO}

We aim to formulate a zero-finding problem whose solution yields the Fourier coefficients of the periodic solution of system~\eqref{rescale first order system}. According to Hypothesis~\ref{hyp:regular_patterned_front}, the periodic orbit $\gamma$ satisfies the symmetry condition described in Definition~\ref{def:symmetrix}, namely $\gamma(0) \in \fix \mathcal{R}$. This implies that the first and third components of $\gamma$ are even functions, while the second and fourth components are odd. Consequently, we represent the even components using cosine Fourier series and the odd components using sine Fourier series. Specifically, for indices $i = 1, 3$ and $j = 2, 4$, the components of $\gamma$ are expanded as
\begin{equation} \label{eq:Fourier series}
\gamma_i(\theta) \bydef (\gamma_i)_0+2 \sum_{k \geq 1 } (\gamma_i)_k \cos(k \theta ), \quad \mbox{and} \quad \gamma_j(\theta) \bydef 2 \sum_{k \geq 1 } (\gamma_j)_k \sin(  k \theta),
\end{equation}
where we abused notation by identifying the function $\gamma_i(\theta)$ with its corresponding sequence of Fourier coefficients $\gamma_i = ((\gamma_i)_k)_{k \ge 0}$. 
Denote $\gamma = (\gamma_i)_{i=1}^4$. Substituting the Fourier expansions from \eqref{eq:Fourier series} into the system \eqref{First Order System} and rearranging terms, we are led to define a map 
\begin{equation} \label{eq:map_periodic_sol}
\Gamma(\gamma ,\omega) = (\Gamma_i(\gamma ,\omega))_{i=1}^4
\end{equation}
whose roots correspond to periodic solutions. The Fourier coefficients of this map are defined component-wise by
\[
(\Gamma_i(\gamma,\omega))_k \bydef k (\gamma_i)_k + (-1)^{i+1} \omega \left( \Psi_i(\gamma, \alpha) \right)_k 
\]
for $i = 1, 2, 3, 4$ and $k \in \mathbb{N}$, where $\Psi_i(\gamma,\alpha)$ denotes the $i$th component of the nonlinearity in \eqref{First Order System}, now evaluated in Fourier space. In this formulation, all products of functions are replaced by discrete convolutions of their respective Fourier coefficients. Specifically, for sequences $a = (a_k)$ and $b = (b_k)$, the {\em cosine} convolution and convolution powers are defined by
\begin{equation} \label{convo_gamma}
	(a* b)_k = \sum_{\substack{k_1+k_2=k \\ k_1,k_2 \in \Z}} a_{|k_1|}  b_{|k_2|} \quad \mbox{and} \quad (a^n)_k = (\underbrace{a * a*\dots * a}_{n ~times})_k.
\end{equation}
If there exists a sequence $(\gamma^*,\omega^*)$ such that $\Gamma(\gamma^*,\omega^*)= 0$, then the corresponding Fourier series yields, via \eqref{eq:Fourier series}, a periodic orbit 
\[
\Gamma \bydef \left\{ \gamma^*(\theta) : \theta \in \left[ 0, 2 \pi \right] \right\}
\] 
of the original system~\eqref{rescale first order system}. To rigorously formulate this problem, we work in a weighted Banach space of Fourier coefficients. We define
\begin{equation} \label{eq:Banach_space_X_gamma}
X_\gamma = \{x = (x_k)_{k\in \N} : x_k \in \R , \| x \|_{X_{\gamma}}  \bydef |x_0| + 2\sum_{k > 0} \left| x_k \right| \nu_\gamma^k < \infty \},
\end{equation}
where $\nu_\gamma > 1$ is a fixed weight chosen a priori. In this setting, finding a periodic solution reduces to solving the zero-finding problem for $\Gamma$ in the product space $X_\gamma^4$. With this formulation in place, we next turn to the computation of the eigenvector bundle associated with the unstable direction.

\subsubsection{Unstable bundle of the periodic solution} \label{sec:unstable_bundle_PO}

To derive the differential equation satisfied by the function $v(\theta)$, which corresponds to the first-order term in the Taylor expansion of the parameterization \eqref{Para diff equa_first_order}, we differentiate the invariance equation \eqref{Para diff equa} with respect to $\sigma$ and evaluate the result at $\sigma = 0$. This yields the linear variational equation
\[
v'(\theta) + \lambda v(\theta) = \omega D_u \Psi(\gamma(\theta),\alpha) v(\theta).
\]
where $D_u\Psi(\gamma(\theta),\alpha)$ denotes the Jacobian of $\Psi$ evaluated along the periodic orbit $\gamma$ at the parameter value $\alpha$. Thus, the eigenvector function $v$ associated with the unstable direction of the $2\pi$-periodic solution of system~\eqref{rescale first order system} satisfies the following linear system
\begin{equation} \label{First Order Bundle}
v'(\theta) + \lambda v(\theta)
=  
\omega
\begin{pmatrix}
 v_2(\theta)  \\ 
 \frac{\partial f_1 }{ \partial u_1} (\gamma_1(\theta),\gamma_3(\theta),\alpha) v_1(\theta) +  \frac{\partial f_1 }{ \partial u_3} (\gamma_1(\theta),\gamma_3(\theta),\alpha) v_3(\theta)  \\ 
 v_4(\theta)  \\
 \frac{\partial f_2}{ \partial u_1} (\gamma_1(\theta),\gamma_3(\theta),\alpha) v_1(\theta) + \frac{\partial f_2}{ \partial u_3} (\gamma_1(\theta),\gamma_3(\theta),\alpha) v_3(\theta)  
\end{pmatrix}.
\end{equation}

For the remainder of this section, we assume that the unstable bundle associated with the periodic solution is orientable. Under this assumption, the corresponding eigenvectors are $2\pi$-periodic functions and can thus be represented by their Fourier series expansion
\begin{equation} \label{Fourier series Bundle}
v_j(\theta) \bydef \sum_{k \in \Z } (v_j)_k e^{\mi k \theta},
\end{equation}
where $\mi$ denotes the imaginary number. 
\begin{rem}
In the case where the bundle is non-orientable, the eigenvectors become $4\pi$-periodic. This is due to the fact that, after a $2\pi$ rotation, the solution returns to the same spatial position along the orbit, but the associated tangent vector reverses direction. In application, we wish to use the same zero-finding formulation for both the orientable and non-orientable cases. To achieve this, we rescale the periodic solution to be $\pi$-periodic, so that the corresponding eigenvectors become $2\pi$-periodic as in the orientable case. Under this rescaling, the coefficients of the periodic solution being solved in \eqref{eq:map_periodic_sol} correspond to an extended solution with non-minimal period $2\pi$, tracing the orbit twice. More details are given in Section \ref{subsec:Isolas}.
\end{rem}
 To define a zero-finding problem for the bundle, we will first rewrite the Cosine/Sine Fourier series of \eqref{eq:Fourier series} into a full two-sided exponential Fourier series by
\begin{equation} \label{periodic solution fourier series}
\gamma_j(\theta) \bydef  \sum_{k \in \Z} (\gamma_j^F)_{k} e^{\mi k \theta}, 
\end{equation}
where the coefficients are given by
\[
(\gamma_j^F)_{k} \bydef (\gamma_j)_{\left|k\right|}, \quad \mbox{ or } \quad 
(\gamma_j^F)_{k} \bydef 
\begin{cases}
0, & \text{for } k = 0, \\
- \mi \sign\left(k \right) (\gamma_j)_{\left|k \right|}, & \text{otherwise},
\end{cases} 
\]
if $j$ is odd or even respectively. Substituting the Fourier series expansions \eqref{Fourier series Bundle} and \eqref{periodic solution fourier series} into the first-order condition \eqref{First Order Bundle} leads to a functional equation whose zeros yield the Fourier coefficients of the unstable bundle. This equation defines a map
\[
V_i(\gamma, v,\lambda,\omega) = ((V_i(\gamma, v,\lambda,\omega))_k)_{k \in \Z}, \quad \text{with } i\in \{1,2,3,4\},
\]
where 
\begin{equation} \label{eq:system_bundle_periodic_sol}
\begin{pmatrix}
(V_1(\gamma, v,\lambda,\omega))_k \\ (V_2(\gamma, v,\lambda,\omega))_k \\ (V_3(\gamma, v,\lambda,\omega))_k \\ (V_4(\gamma, v,\lambda,\omega))_k 
\end{pmatrix}
\bydef 
\begin{pmatrix}
\left( \mi  k + \lambda \right)(v_1)_k - \omega (v_2)_k \\
\left( \mi  k + \lambda \right)(v_2)_k - \omega \left( \frac{\partial f_1 }{ \partial u_1} (\gamma_1,\gamma_3,\alpha) * v_1 \right)_k - \omega \left( \frac{\partial f_1 }{ \partial u_3} (\gamma_1,\gamma_3,\alpha) * v_3 \right)_k \\
\left( \mi  k + \lambda \right)(v_3)_k - \omega (v_4)_k \\
\left( \mi  k + \lambda \right)(v_4)_k - \omega \left( \frac{\partial f_2 }{ \partial u_1} (\gamma_1,\gamma_3,\alpha) * v_1 \right)_k - \omega \left( \frac{\partial f_2 }{ \partial u_3} (\gamma_1,\gamma_3,\alpha) * v_3 \right)_k  
\end{pmatrix}
\end{equation}
and where $*$ now denotes the {\em full} discrete convolution defined by
\begin{equation}\label{convo_v}
	(a* b)_k = \sum_{\substack{k_1+k_2=k \\ k_1,k_2 \in \Z}} a_{k_1}  b_{k_2}.
\end{equation}
We define the Banach space of Fourier coefficients as
\begin{equation} \label{eq:Banach_space_X_v}
X_v = \{ x = (x_k)_{k\in \Z} : x_k \in \C , \| x \|_{X_{v}}  \bydef \sum_{k\in \Z} \left| x_k \right| \nu_v^{|k|} < \infty \}
\end{equation}
for a weight $\nu_v>1$ fixed apriori. Since the system \eqref{First Order Bundle} is linear in $V$, any scalar multiple of a solution is also a solution. As a result, solutions are not isolated, and we must impose an additional condition to normalize the eigenvector. Fix $\bar{\ell} \in \{1,2,3,4\} $ and $ \rho \in \R\setminus\{0\}$, and define
\begin{equation}
V_0(v) \bydef \sum_{k \in \Z} (v_{\bar{\ell}})_k -  \rho. \label{normalize_eigenvector}
\end{equation} 
We then define the full nonlinear map as
\begin{equation} \label{eq:bundle_map}
V(\gamma, v,\lambda,\omega) =  (V_i(\gamma, v,\lambda,\omega))_{i=0}^4
\end{equation} 
whose root provide us with the wanted unstable bundle associated to the periodic orbit.  

\subsubsection{Higher order terms of the parameterization of \boldmath$W_{\rm loc}^u(\Gamma)$\unboldmath} \label{sec:high-order-term_WuGamma}

Having introduced the zero-finding problems for the periodic orbit $\gamma$ and the unstable bundle $(v,\lambda)$, given in \eqref{eq:map_periodic_sol} and \eqref{eq:bundle_map} respectively, we now turn to the construction of the map whose zeros correspond to a parameterization of the local unstable manifold $W_{\rm loc}^u(\Gamma)$. Recall that this parameterization $W(\theta,\sigma)$ satisfies the invariance equation \eqref{Para diff equa}. We represent $W$ as a Fourier-Taylor expansion of the form
\begin{equation}\label{Parameterization map}
	W(\theta,\sigma) = \sum_{n=(n_1,n_2) \in \Z \times \N} w_{n} e^{ i n_1 \theta  }\sigma^{n_2},\quad \mbox{with} \quad w_n = 
	\left(
		(w_1)_{n},
		(w_2)_{n},
		(w_3)_{n},
		(w_4)_{n} \right)
\end{equation}
where the zeroth and first-order Taylor coefficients are fixed by the periodic orbit and the eigenfunction
\[
W(\theta,0) = \gamma^F(\theta) \quad \mbox{ and } \quad \frac{\partial }{\partial \sigma } W(\theta,0) = v(\theta),
\]
which translates component-wise to
\[
(w_j)_{(n_1,0)} = (\gamma_j^F)_{n_1} \quad \mbox{ and } \quad (w_j)_{(n_1,1)} = (v_j)_{n_1}.
\]
Substituting the expansion \eqref{Parameterization map} into the invariance equation \eqref{Para diff equa}, we define the nonlinear map
\begin{equation} \label{eq:high_order_term_WuGamma} 
W(\gamma, v,\lambda,w,\omega) = (W_i(\gamma, v,\lambda,w,\omega))_{i=1}^4
\end{equation} 
whose root provide us with the coefficients of the parameterization of $W_{\rm loc}^u(\Gamma)$, and where each $W_i(\gamma, v,\lambda,w,\omega)$ is defined component-wise as
\[
(W_i(\gamma, v,\lambda,w,\omega))_{(n,m)} = 
	\begin{cases}
		(w_i)_{(n,0)} - (\gamma_i^F)_{n} & \text{for } m=0,\\
		(w_i)_{(n,1)} - (v_i)_{n} & \text{for } m=1,\\
		\left(  \mi  n + \lambda m \right)(w_i)_{(n,m)} - \omega (\Psi_i(w,\alpha))_{(n,m)} & \text{otherwise},
	\end{cases}
\]
and where the nonlinear terms $\Psi_i(w,\alpha)$ are defined by substituting the polynomial nonlinearities of \eqref{First Order System} with Fourier-Taylor convolutions, given by
\begin{equation}\label{convo_w}
	(a*b)_k = \sum_{\substack{k_1+k_2=k \\ k_1,k_2 \in \Z \times \N}} a_{k_1}  b_{k_2} \quad \mbox{and} \quad (a^n)_k = (\underbrace{a * a*\dots * a}_{n ~times})_k.
\end{equation} 
We define the Banach space of Fourier-Taylor coefficients by 
\begin{equation} \label{eq:Banach_space_X_w}
X_w = \{x = (x_k)_{ k \in \Z \times \N} : x_k \in \C , \| x \|_{X_{w}}  \bydef \sum_{ k \in \Z \times \N} \left| x_k \right| \nu_w^{|k_1| + k_2} < \infty \},
\end{equation}
for a fixed weight $\nu_w \ge 1$. At this stage, we note that one equation is still missing in order to solve for the frequency $\omega$. This missing condition will be introduced later, when constructing the connecting orbit between the invariant manifolds. We now proceed to define the map corresponding to the stable manifold of the equilibrium.


\subsection{Local stable manifold of the origin} \label{sec:stable_manifold_origin}
We now derive a zero-finding problem for a parameterization of the local stable manifold $W^s_{\rm loc}(0)$ at the origin using the parameterization method. This construction closely mirrors the approach used for computing the local unstable manifold of the periodic orbit, as described in Section~\ref{sec:unstable_manifold_PO}. Let $B_1(0) = \{ (\sigma_1,\sigma_2) \in \C^2: |\sigma_1|,|\sigma_2| \le 1\} \subset \C^2$. We seek a map
\begin{equation} \label{eq:general_map_stable_manifold}
P: B_1(0) \subset \C^2 \to \C^4:(\sigma_1,\sigma_2) \mapsto P(\sigma_1,\sigma_2)
\end{equation}
whose image is the local stable manifold $W^s_{\rm loc}(0)$. The map $P$ must satisfy the following {\em initial conditions}
\begin{equation} \label{init_con_stable}
P(0,0) = 0 \quad \mbox{ and } \quad \frac{\partial }{\partial \sigma_i } P(0,0) = \xi_i,
\end{equation}
and the {\em invariance equation}
\begin{equation}\label{Para diff equa stable}
	\lambda_1 \sigma_1 \frac{\partial }{\partial \sigma_1 } P(\sigma_1,\sigma_2)  + \lambda_2 \sigma_2  \frac{\partial }{\partial \sigma_2 }  P(\sigma_1,\sigma_2) = \Psi (P(\sigma_1,\sigma_2) ),
\end{equation}
where $(\lambda_i, \xi_i)$ for $i=1,2$ are the two stable eigenpairs of the linearized system at the origin, assumed to exist by Hypothesis~\ref{hyp:hyperbolic}. Since the eigenpairs $(\lambda_i, \xi_i)$ are not known in advance, we first solve for them by introducing a zero-finding problem. We define the map
\begin{equation} \label{eq:map_eigenpair} 
	E(\lambda_1, \xi_1, \lambda_2, \xi_2)  \bydef \begin{pmatrix}
		D_u\Psi(0,\alpha) \lambda_1 - \lambda_1 \xi_1 \\
		 (\xi_1)_{\bar{k}_1} - \rho_1 \\
		D_u\Psi(0,\alpha) \lambda_2 - \lambda_2 \xi_1 \\
		(\xi_2)_{\bar{k}_2} - \rho_2
	\end{pmatrix},
\end{equation}
where $D_u\Psi(0,\alpha)$ denotes the Jacobian matrix of the system \eqref{First Order System} evaluated at the origin and parameter value $\alpha$. To ensure uniqueness of the eigenvectors $\xi_1$, $\xi_2$, we fix one component $\bar{k}_i \in \{1,2,3,4\}$ of each and prescribe its value as $\rho_i$.

To solve the invariance equation \eqref{Para diff equa stable}, we expand $P$ as a Taylor series
\begin{equation} \label{Parameterization map stable}
P(\sigma_1,\sigma_2) = \sum_{n \in \N^2 } p_{n} \sigma_1^{n_1}\sigma_2^{n_2},\quad \mbox{with} \quad p_n = \left(
(p_1)_{n} , (p_2)_{n} , (p_3)_{n} , (p_4)_{n} \right).
\end{equation}
The initial conditions \eqref{init_con_stable} then become
\[
(p_j)_{(0,0)} = 0\quad \mbox{ and } \quad (p_j)_{(1,0)} = (\xi_1)_{j}\quad \mbox{ and } \quad (p_j)_{(0,1)} = (\xi_2)_{j}.
\]
For each $j=1,2,3,4$, denote $p_j = ((p_j)_{n})_{n \in \N}$ and denote $p=(p_1,p_2,p_3,p_4)$. Substituting \eqref{Parameterization map stable} in \eqref{Para diff equa stable}, we define a nonlinear map
\begin{equation} \label{eq:map_stable_manifold}
P(\lambda_1,\xi_1,\lambda_2,\xi_2,p) = (P_i(\lambda_1,\xi_1,\lambda_2,\xi_2,p))_{i=1}^4
\end{equation}
whose zeros correspond to the Taylor coefficients of the parameterization with components 
\[ 
(P_i(\lambda_1,\xi_1,\lambda_2,\xi_2,p))_{(n,m)} \bydef
	\begin{cases}
		(p_i)_{(0,0)}  & \text{for } n = m=0, \\
		(p_i)_{(1,0)} - (\xi_1)_{i} & \text{for } n=1 \text{ and } m=0,\\
		(p_i)_{(0,1)} - (\xi_2)_i & \text{for } n=0 \text{ and } m=1,\\
		\left( \lambda_1 n  + \lambda_2 m \right)(p_i)_{(n,m)} - (\Psi_i (p,\alpha))_{(n,m)} & \text{otherwise}.
	\end{cases}
\]
Here, the nonlinear terms $\Psi_i (p,\alpha)$ are defined by replacing the polynomial nonlinearities in \eqref{First Order System} with Cauchy products (discrete convolutions), given by
\begin{equation}\label{convo_p}
	(a* b)_k = \sum_{\substack{k_1+k_2=k \\ k_1,k_2 \in \N^2}} a_{k_1}  b_{k_2} \quad \mbox{and} \quad (a^n)_k = (\underbrace{a * a*\dots * a}_{n ~times})_k.
\end{equation}
We solve the equation $P(p)=0$ in the Banach space of Taylor coefficients
\begin{equation} \label{eq:Banach_space_X_p}
X_p = \{x = (x_k)_{k \in  \N^2 } : x_k \in \C , \| x \|_{X_{p}}  \bydef \sum_{k \in  \N^2} \left| x_k \right| \nu_w^{k_1 + k_2} < \infty \},
\end{equation}
for a fixed weight $\nu_w \ge 0$. Once a zero of the map \eqref{eq:map_stable_manifold} is found in $X_p$, the corresponding Taylor expansion \eqref{Parameterization map stable} defines the parameterization of the local stable manifold as
\[
W^s_{\rm loc}(0) = P(B_1(0)).
\]
Let us next define the zero-finding problem for the connections between $W^u_{\rm loc}(\Gamma)$ and $W^s_{\rm loc}(0)$.


\subsection{Zero-finding problem for the heteroclinic connection}\label{sec:connectingorbit}

Having established the zero-finding problems corresponding to the parameterizations of the local unstable manifold of the periodic orbit (see Section~\ref{sec:unstable_manifold_PO}) and the local stable manifold of the equilibrium at the origin (see Section~\ref{sec:stable_manifold_origin}), we are now prepared to formulate the zero-finding problem that characterizes a heteroclinic connection between $W^u_{\rm loc}(\Gamma)$ and $W^s_{\rm loc}(0)$. The core idea is to express the heteroclinic connection as a solution to a boundary value problem, represented via a Chebyshev expansion in time. To this end, we first rescale and shift the time domain from the interval $[0,{L}_c]$ to $[-1,1]$. Under this change of variables, the system \eqref{First Order System} becomes
\begin{equation} \label{rescale_L}
u' =  \frac{L_c}{2} \Psi(u,\alpha),
\end{equation}
where $L_c>0$ is the unknown length of the connecting segment and must be determined as part of the solution. Our objective is to find an orbit of the rescaled system \eqref{rescale_L} that connects the unstable and stable manifolds over the finite time segment $[-1,1]$. This is formulated as a boundary value problem: the solution at time $t=-1$ is required to lie on $W^u_{\rm loc}(\Gamma)$, while the solution at $t=1$ must lie on $W^s_{\rm loc}(0)$. To evaluate the initial condition at $t=-1$, we refer to the parameterization defined in \eqref{Parameterization map}, and fix the Taylor parameter  at a given $\sigma=\sigma_0 < 1$. The angular variable $\theta$, which appears in the Definition~\ref{dfn:patternedfront} of a patterned front, is treated as an unknown. Since the entire branch of parameterized heteroclinic connections will later be represented using Chebyshev series, all variables must be expressed in terms of Chebyshev expansions, including the angular variable $\theta$. However, expressing $e^{\mi \theta}$ in this form presents analytical difficulties when $\theta$ itself is represented by a Chebyshev series. To avoid this complication, we introduce two auxiliary real variables, namely 
\[
\theta_1 = \Real ({e^{\mi  \theta}}) \quad \text{and} \quad \theta_2 = \Imag ({e^{\mi \theta}}),
\]
so that ${e^{\mi  \theta}} = \theta_1 +\mi \theta_2$. With this formulation, the pair $(\theta_1, \theta_2)$ must lie on the unit circle, a condition we enforce through the constraint equation $J_2$ given in \eqref{eq:angle_condition_map}. As a consequence, the identity
\begin{equation} \label{eq:identity_unit_circle}
(\theta_1 + \mi \theta_2)^{-n} = (\theta_1 - \mi \theta_2)^n
\end{equation}
holds for all $n \in \mathbb{N}$, allowing us to recast the dependence of $W$ on $\theta$ entirely in terms of $\theta_1$ and $\theta_2$:
\begin{align*}
	W(\theta, \sigma_0) &= \sum_{ \substack{n \in \mathbb{Z} \\ m \in  \mathbb{N}}} w_{n,m} e^{in\theta} \sigma_0^m
	= \sum_{ \substack{n \in \mathbb{Z} \\ m \in  \mathbb{N}}} w_{n,m} (\theta_1 +\mi \theta_2 )^n \sigma_0^m 
	=  \sum_{ \substack{n \in \mathbb{Z} \\ m \in  \mathbb{N}}} w_{n,m} (\theta_1 + \sign(n) \mi \theta_2 )^{|n|}  \sigma_0^m  \\
	& \bydef W_u^\Gamma(\theta_2,\theta_2, \sigma_0).
\end{align*}
To solve \eqref{rescale_L} on the time interval $[-1,1]$, we expand each component of the solution using Chebyshev polynomials $\{T_n(t)\}_{n \ge 0}$, defined recursively by
$T_0(t)=1$, $T_1(t)=t$ and $T_{n+1}(t)=2tT_n(t)-T_{n-1}(t)$, and which satisfy the identity $T_n(t) = \cos(n \arccos(t))$ for all $t\in [-1,1]$, making them effectively cosine polynomials. Each component $u_i(t)$ of the solution is then represented as
\begin{equation} \label{eq:Chebyshev_series_BVP}
u_i(t) = (a_i)_0 + 2 \sum_{ n \geq 1} (a_i)_n T_n(t)
\end{equation}
where for each $i \in \{1,2,3,4\}$, the sequence $a_i \bydef ((a_i)_n)_{n \ge 0}$ denotes the sequence of Chebyshev coefficients of the function $u_i(t)$. Let $a \bydef (a_1,a_2,a_3,a_4)$ denote the collection of all four component series. Following a similar approach as the ones in \cite{MR3148084,MR4292534}, substituting the Chebyshev expansions \eqref{eq:Chebyshev_series_BVP} into the system \eqref{rescale_L} leads to the definition of a nonlinear map 
\begin{equation} \label{eq:map_projected_BVP}
G(w,L_c,\theta_1,\theta_2,a) \bydef \left( G_i(w,L_c,\theta_1,\theta_2,a) \right)_{i=1}^4
\end{equation}
whose zeros determine the Chebyshev coefficients of the solution to the initial value problem with initial value lying on $W^u_{\rm loc}(\Gamma)$. Each component of this map is defined by
\[
(G_i(w,L_c,\theta_1,\theta_2,a))_n \bydef 
\begin{cases}
\displaystyle \left( W_u^\Gamma(\theta_1,\theta_2,\sigma_0) \right)_i - (a_i)_0 - 2 \sum_{n \geq 1 } (a_i)_n (-1)^n, & \text{for } n = 0, \\
\displaystyle 2n(a_i)_n + \frac{L_c}{2} (\Psi_i(a,\alpha))_{n+1} - \frac{{L}_c}{2} (\Psi_i(a,\alpha))_{n-1}, & \text{for } n \geq 1,
\end{cases}
\]
where the nonlinear terms $\Psi_i(a,\alpha)$ are obtained by replacing the polynomial nonlinearities of the system \eqref{First Order System} with cosine-type discrete convolutions, defined as
\[
(a* b)_k = \sum_{\substack{k_1+k_2=k \\ k_1,k_2 \in \Z}} a_{|k_1|}  b_{|k_2|}.
\]
\begin{rem}\label{rmk:phase_orbit}
The size of $W_{\rm loc}^u(\Gamma)$ is fixed by the normalization constant $ \rho $, introduced in \eqref{normalize_eigenvector}. By also fixing the Taylor parameter $\sigma_0$, we define a circular level set of the manifold at a fixed distance from the orbit, parameterized by the angular variable $\theta$, which is represented by the auxiliary variables $\theta_1$ and $\theta_2$. Since the initial condition in \eqref{eq:map_projected_BVP} is constrained to lie on this level set and $\theta_1$, $\theta_2$ are solved for as part of the zero-finding problem, the phase is well-defined locally and can be traced consistently along the loop of patterned fronts.
\end{rem}

This formulation leverages the cosine structure inherent in the Chebyshev polynomials, which behave like Fourier modes. To solve the zero-finding problem $G(a,L_c,\theta_1,\theta_2)=0$, we equip each component of $a=(a_1,\dots,a_4)$ with a suitable Banach space structure. Specifically, we define the Banach space of Chebyshev coefficients as
\begin{equation} \label{eq:Banach_space_X_a}
X_a = \{x = (x_k)_{k \in  \N } : x_k \in \R , \| x \|_{X_{a}}  \bydef |x_0| + 2 \sum_{k > 1} \left| x_k \right| \nu_a^{k} < \infty \},
\end{equation}
where $\nu_a \ge 1$ is a fixed weight parameter controlling the decay rate of the coefficients.

Although a zero of the map $G$, defined in \eqref{eq:map_projected_BVP}, guarantees that the corresponding orbit originates on the local unstable manifold $W^u_{\rm loc}(\Gamma)$, as enforced by the initial condition at $t=-1$, it remains necessary to ensure that this orbit terminates on the local stable manifold $W^s_{\rm loc}(0)$. To that end, we impose a boundary condition at the terminal point $t=1$, which involves the parameterization $P$ of $W^s_{\rm loc}(0)$. However, as introduced in general in \eqref{eq:general_map_stable_manifold} and expressed as a multivariate Taylor expansion in \eqref{Parameterization map stable}, the parameterization $P$ is naturally defined over the complex domain. This reflects the presence of complex-conjugate stable eigenvalues in the linearized system at the origin, which induces complex-conjugate coefficients in the expansion of $P$. To reformulate the problem in terms of real variables, we exploit this symmetry and rewrite the parameterization using the identity
\[
P( \sigma_1 +\mi  \sigma_2 , \sigma_1 -\mi \sigma_2) = \sum_{n \in \N^2} p_{n} (\sigma_1 +\mi  \sigma_2)^{n_1} (\sigma_1 -\mi \sigma_2)^{n_2},
\]
which allows us to represent the stable manifold using real coordinates while preserving the structure imposed by the complex eigenstructure.

To enforce that the computed orbit connects to the local stable manifold at the final time $t=1$, we exploit a classical identity satisfied by Chebyshev polynomials: $T_n(1)=1$ for all $n \ge 0$> Using this we evaluate the Chebyshev expansion of each component $u_i(t)$ at $t=1$, yielding $u_i(1) = (a_i)_0 + 2 \sum_{ n \geq 1} (a_i)_n$. This observation allows us to formulate the endpoint condition requiring the orbit to lie on $W^s_{\rm loc}(0)$ at $t=1$. Specifically, we define the map
\begin{equation} \label{eq:map_endpoint_Ws}
H(\sigma_1,\sigma_2,a,p) = \left( P_i( \sigma_1 +\mi  \sigma_2 , \sigma_1 -\mi \sigma_2)  - (a_i)_0 - 2 \sum_{n \geq 1 } (a_i)_n \right)_{i=1}^4,
\end{equation}
%
%
%
so that the condition $H(\sigma_1,\sigma_2,a,p) =0$ enforces that the orbit connects to $W^s_{\rm loc}(0)$ at $t=1$. 
However, because the segment length $L_c$ is itself an unknown, the boundary condition $H(\sigma_1,\sigma_2,a,p)=0$ alone does not specify a unique solution. In fact, there typically exists a continuum of values for the pair $(\sigma_1, \sigma_2)$ that satisfy the condition. To restore uniqueness, we introduce the scalar constraints 
\begin{equation} \label{eq:angle_condition_map}
J(\sigma_1,\sigma_2,\theta_1,\theta_2) = 
\begin{pmatrix} J_1(\sigma_1,\sigma_2) \\ J_2(\sigma_1,\sigma_2) \end{pmatrix}
\bydef
\begin{pmatrix}	
0.95 - \sigma_1^2-\sigma_2^2 \\ 
1 - \theta_1^2-\theta_2^2 \end{pmatrix}.
\end{equation}
The first constraint $J_1(\sigma_1,\sigma_2)=0$ bounds $(\sigma_1, \sigma_2)$ to remain within the radius of convergence of the stable manifold’s power series expansion. This can be ensured, for instance, by selecting a sufficiently small normalization for the eigenvectors $\xi_1,\xi_2$ (see, e.g., \cite{MR3437754}). The second constraint $J_2(\sigma_1,\sigma_2)=0$ guarantees that the pair $(\theta_1, \theta_2)$ lies on the unit circle, enforcing the identity \eqref{eq:identity_unit_circle}. 

We are now in a position to formulate the zero-finding problem that characterizes the heteroclinic connection. A locally unique zero of this map will provide a constructive proof of existence for a patterned front. To this end, we define the variable
\begin{equation} \label{eq:variables_in_y}
y \bydef \left( \gamma , \lambda, v , w, L_c, \theta_1, \theta_2 , \sigma_1, \sigma_2 , \omega, a  , \lambda_1, \xi_1 , \lambda_2, \xi_2 , p \right)
\end{equation}
whose components are given as follows:
\begin{itemize}
\item $\gamma = (\gamma_1,\dots,\gamma_4)$: Fourier coefficients of the periodic orbit $\Gamma$;
\vspace{-.2cm}
\item $\lambda$: the unstable Floquet exponent associated with $\Gamma$;
\vspace{-.2cm}
\item $v = (v_1,\dots,v_4)$: Fourier coefficients defining the unstable Floquet bundle of $\Gamma$;
\vspace{-.2cm}
\item $w = (w_1,\dots,w_4)$: Fourier-Taylor coefficients of the parameterization of $W_{\rm loc}^u(\Gamma)$;
\vspace{-.2cm}
\item $L_c$: flight time for the orbit segment connecting $W_{\rm loc}^u(\Gamma)$ to $W_{\rm loc}^s(0)$;
\vspace{-.2cm}
\item $(\theta_1,\theta_2)$: real auxiliary variables encoding the angular parameter $\theta$ via ${e^{\mi  \theta}} = \theta_1 +\mi \theta_2$;
\vspace{-.2cm}
\item $(\sigma_1,\sigma_2)$: coordinates in the parameterization domain of the stable manifold $W_{\rm loc}^s(0)$;
\vspace{-.2cm}
\item $\omega$: frequency of the periodic orbit $\Gamma$;
\vspace{-.2cm}
\item $a = (a_1,\dots,a_4)$: Chebyshev coefficients of the orbit segment defined over $[-1,1]$;
\vspace{-.2cm}
\item $(\lambda_1,\lambda_2)$: stable eigenvalues of the Jacobian $D_u \Psi(0,\alpha)$;
\vspace{-.2cm}
\item $(\xi_1,\xi_2)$: the eigenvectors associated to $\lambda_1$ and $\lambda_2)$;
\vspace{-.2cm}
\item $p = (p_1,\dots,p_4)$: Taylor coefficients of the parameterization of $W_{\rm loc}^s(0)$.
\end{itemize}
We now recall the collection of maps involved in the formulation of the zero-finding problem:
\begin{itemize}
\item $ \Gamma = \Gamma(\gamma ,\omega)$,  corresponding to the periodic solution, as defined in~\eqref{eq:map_periodic_sol};
\vspace{-.2cm}
\item $ V = V(\gamma, v,\lambda,\omega)$, corresponding to the unstable Floquet bundle in~\eqref{eq:bundle_map};
\vspace{-.2cm}
\item $ W = W(\gamma, v,\lambda,w,\omega)$, representing the parameterization of $ W_{\rm loc}^u(\Gamma) $, see~\eqref{eq:high_order_term_WuGamma};
\vspace{-.2cm}
\item $ G = G(w,L_c,\theta_1,\theta_2,a)$, governing the orbit segment via Chebyshev expansion, see~\eqref{eq:map_projected_BVP};
\vspace{-.2cm}
\item $ H = H(\sigma_1,\sigma_2,a,p)$, enforcing the endpoint of the orbit to lie on $ W_{\rm loc}^s(0) $, see~\eqref{eq:map_endpoint_Ws};
\vspace{-.2cm}
\item $ J = J(\sigma_1,\sigma_2,\theta_1,\theta_2) $, providing normalization conditions for the stable and angular variables, defined in~\eqref{eq:angle_condition_map};
\vspace{-.2cm}
\item $ E = E(\lambda_1, \xi_1, \lambda_2, \xi_2) $, solving for the stable eigenpairs, given in~\eqref{eq:map_eigenpair};
\vspace{-.2cm}
\item $ P = P(\lambda_1,\xi_1,\lambda_2,\xi_2,p)$, epresenting the parameterization of $ W_{\rm loc}^s(0) $, as defined in~\eqref{eq:map_stable_manifold}.
\end{itemize}
Together, these maps define the overall system  
\begin{equation} \label{eq:zero_finding_f}
f(\alpha,y) \bydef  
\begin{pmatrix}	
\Gamma(\gamma ,\omega) \\
V(\gamma, v,\lambda,\omega) \\
W(\gamma, v,\lambda,w,\omega) \\ 
H(\sigma_1,\sigma_2,a,p) \\
J(\sigma_1,\sigma_2,\theta_1,\theta_2) \\ 
G(w,L_c,\theta_1,\theta_2,a) \\
E(\lambda_1,\xi_1,\lambda_2,\xi_2) \\
P(\lambda_1,\xi_1,\lambda_2,\xi_2,p)
\end{pmatrix},
\end{equation}
whose zeros correspond to heteroclinic connections between the periodic orbit $ \Gamma $ and the equilibrium at the origin, at the fixed parameter value $ \alpha $. Note that the dependence on $\alpha$ is implicit in all maps involving the vector field $ \Psi(\cdot,\alpha) $, including $\Gamma$, $ V $, $ W $, $ G $, $ E $, and $ P $.


\subsection{Pseudo-arclength continuation of branches of heteroclinic connections}

Having defined in equation~\eqref{eq:zero_finding_f} a nonlinear map $f$ whose zeros correspond to the desired heteroclinic connections, we now aim to compute one-dimensional families (branches) of such solutions by means of a continuation method. A direct continuation in the parameter $\alpha$ may be problematic, since the branch of solutions is expected to undergo fold (saddle-node) bifurcations, where the Fréchet derivative $D_y f$ of $f$ may become singular or ill-conditioned. To overcome this difficulty, we adopt a pseudo-arclength continuation strategy, in which the parameter $\alpha$ is treated as an additional unknown and the solution curve is followed in a desingularized coordinate system.
We introduce the extended variable
\[
x \bydef (\alpha,y)
\]
which includes both the system parameter $\alpha$ and the heteroclinic connection variables $y$ defined in~\eqref{eq:zero_finding_f}. Let $\tilde{x} = (\tilde{\alpha}, \tilde{y})$ be a known numerical approximation of a solution to $f(y,\alpha) = 0$, and let $\eta$ be a normalized approximation of the tangent vector to the solution branch at $\tilde{x}$. We define a scalar function
\[
Q(x) \bydef (x - \tilde{x} ) \cdot \eta
\]
which measures the signed arclength displacement from the base point $\tilde{x}$ along the direction $\eta$. With this setup, we define the pseudo-arclength continuation problem as the following extended zero-finding map:
\begin{equation}\label{eq:zero_finding_F}
	F(x) \bydef  \begin{pmatrix}
		Q(x) \\
		f(x)
	\end{pmatrix}.
\end{equation}
By construction, any root $x^* = (\alpha^*, y^*)$ of $F$ satisfies both the original system $f(y^*,\alpha^*) = 0$ and the orthogonality constraint $Q(x^*) = 0$, which ensures that the new solution lies on the hyperplane passing through $\tilde{x}$ in the direction $\eta$. This desingularization allows us to compute solution branches through folds and other singular points in a robust and numerically stable way.

Recall the definition of the Banach spaces $X_\gamma$, $X_v$, $X_w$, $X_p$ and $X_a$ given by \eqref{eq:Banach_space_X_gamma}, \eqref{eq:Banach_space_X_v}, \eqref{eq:Banach_space_X_w}, \eqref{eq:Banach_space_X_p} and \eqref{eq:Banach_space_X_a}, respectively. We aim to solve the desingularized zero-finding problem $F(x)=0$ on the product Banach space
\begin{equation} \label{eq:Banach_Space_X}
X \bydef \mathbb{K}\times X_\gamma^4 \times \mathbb{K}\times X_v^4 \times X_w^4 \times \mathbb{K}^6 \times X_a^4 \times \mathbb{K}^{10} \times X_p^4 
\end{equation}
where $\mathbb{K} \in \{ \R, \C \}$ and each subspace corresponds to the appropriate spaces introduced in Section~\ref{sec:connectingorbit}. For notational simplicity, we collect all components into a single vector
$x = (x_1,\dots, x_{29})$,
where each $x_i$ corresponds to one of the variables in the set $(\alpha,y)$ with the variables in $y$ are listed in \eqref{eq:variables_in_y}. We use the shorthand
\[
X = X_1\times...\times X_{29}
\]
to denote the full product space and define the norm on the product space $X$ as 
\[
\| x \|_X = \max_{i = 1,...,29} \mu_i \| x_i \|_{X_i}
\]
where the $\mu_i > 0$ are numerical weights introduced to improve the sharpness of our estimates. Further details on their computation can be found in Section~\ref{sec:bounds}. Table~\ref{table_variables} provides the correspondence between the indices $i$, the variables $x_i$, and the respective function spaces $X_i$. 
\begin{table}[H]
	\centering
	{\footnotesize
	\begin{tabular}{||c|c|c||c|c|c||c|c|c||ccc}
		\hline
		\rowcolor[HTML]{C0C0C0} 
		Index & Var        & Space              & Index    & Var        & Space              & Index    & Var        & Space             & \multicolumn{1}{c|}{\cellcolor[HTML]{C0C0C0}Index} & \multicolumn{1}{c|}{\cellcolor[HTML]{C0C0C0}Var} & \multicolumn{1}{c||}{\cellcolor[HTML]{C0C0C0}Space} \\ \hline
		$x_1$ & $\alpha$   & $\mathbb{K}$       & $x_9$    & $v_3$      & ${X}_{v}$  & $x_{17}$ & $\theta_2$ & $\mathbb{K} $     & \multicolumn{1}{c|}{$x_{25}$}                      & \multicolumn{1}{c|}{$\xi$}                       & \multicolumn{1}{c||}{$\mathbb{K}^{10} $}            \\ \hline
		$x_2$ & $\gamma_1$ & ${X}_{\gamma}$ & $x_{10}$ & $v_4$      & ${X}_{v}$  & $x_{18}$ & $\sigma_1$ & $\mathbb{K} $     & \multicolumn{1}{c|}{$x_{26}$}                      & \multicolumn{1}{c|}{$p_1$}                       & \multicolumn{1}{c||}{${X}_{p}$}            \\ \hline
		$x_3$ & $\gamma_2$ & ${X}_{\gamma}$ & $x_{11}$ & $w_1$      & ${X}_{w}$ & $x_{19}$ & $\sigma_2$ & $\mathbb{K} $     & \multicolumn{1}{c|}{$x_{27}$}                      & \multicolumn{1}{c|}{$p_2$}                       & \multicolumn{1}{c||}{${X}_{p}$}            \\ \hline
		$x_4$ & $\gamma_3$ & ${X}_{\gamma}$ & $x_{12}$ & $w_2$      & ${X}_{w}$ & $x_{20}$ & $\omega$   & $\mathbb{K} $     & \multicolumn{1}{c|}{$x_{28}$}                      & \multicolumn{1}{c|}{$p_3$}                       & \multicolumn{1}{c||}{${X}_{p}$}            \\ \hline
		$x_5$ & $\gamma_4$ & ${X}_{\gamma}$ & $x_{13}$ & $w_3$      & ${X}_{w}$ & $x_{21}$ & $a_1$      & ${X}_{a}$ & \multicolumn{1}{c|}{$x_{29}$}                      & \multicolumn{1}{c|}{$p_4$}                       & \multicolumn{1}{c||}{${X}_{p}$}            \\ \hline
		$x_6$ & $\lambda$  & $\mathbb{K}$       & $x_{14}$ & $w_4$      & ${X}_{w}$ & $x_{22}$ & $a_2$      & ${X}_{a}$ &                                                    &                                                  &                                                    \\ \cline{1-9}
		$x_7$ & $v_1$      & ${X}_{v}$  & $x_{15}$ & ${L}_c$        & $\mathbb{K} $      & $x_{23}$ & $a_3$      & ${X}_{a}$ &                                                    &                                                  &                                                    \\ \cline{1-9}
		$x_8$ & $v_2$      & ${X}_{v}$  & $x_{16}$ & $\theta_1$ & $\mathbb{K} $      & $x_{24}$ & $a_4$      & ${X}_{a}$ &                                                    &                                                  &                                                    \\ \cline{1-9}
	\end{tabular}
	}
	\caption{{\small Table of variables and associated Banach space.}}\label{table_variables}
\end{table}

Note that in Table~\ref{table_variables}, the component $x_{25}=\xi = (\lambda_1, \xi_1, \lambda_2, \xi_2) \in \mathbb{K}^{10}$ corresponds to the stable eigenvalues and eigenvectors of the linearization $D_u \Psi(0, \alpha)$.

Having defined in equation~\eqref{eq:zero_finding_F} a desingularized zero-finding problem whose solutions yield heteroclinic orbits for a fixed value of the parameter $\alpha$, our goal is now to extend this to a continuous family of such solutions over an interval of $\alpha$ values.

Assume we can construct two continuous maps
\begin{equation} \label{eq:families_solutions_tangent}
\bx : [-1,1] \rightarrow {X} \quad \mbox{and} \quad {\eta}_s: [-1,1] \rightarrow {X}
\end{equation}
which represent approximate solutions and their corresponding tangent directions. For each $s \in [-1, 1]$, we define the affine constraint
\[
Q_s(x,s) \bydef \left( x -\bx(s) \right) \cdot \eta_s(s)
\]
which defines a hyperplane in $X$ intersecting the expected solution curve. We then formulate the extended zero-finding problem:
\begin{equation} \label{eq:ZFP}
	F(x , s) \bydef  \begin{pmatrix}
		Q_s(x , s) \\
		f(x)
	\end{pmatrix}
\end{equation} 
where $x = (\alpha, y)$ and $f(x)$ is as defined in equation~\eqref{eq:zero_finding_f}. By solving $F(x,s) = 0$ for $s \in [-1, 1]$, we obtain a branch of heteroclinic solutions as the parameter $\alpha$ varies. This formulation sets the stage for applying a rigorous pseudo-arclength continuation method, which we formalize next.


\section{Rigorous validation of branches of heteroclinic connections} \label{sec:bounds}

In this section, we present the rigorous validation framework used to certify the existence of a continuous family of heteroclinic orbits. Building on the zero-finding formulation \eqref{eq:ZFP}, our approach applies a Newton–Kantorovich-type argument (introduced in Section~\ref{sec:newtonkantorovich}) to rigorously establish the existence of solutions near numerically computed approximations. The core idea is to construct a Newton-like operator $T$ on a suitable Banach space, whose fixed points correspond to true solutions of the problem. Given an approximation $\bx(s)$ at a parameter value $s \in [-1,1]$, we prove that $T$ admits a unique fixed point nearby. To make the problem computationally tractable, we project the infinite-dimensional setting onto finite-dimensional subspaces, as introduced in Section~\ref{sec:explicit_bounds}, derive operator bounds to control nonlinearities and ensure contraction, and use Chebyshev interpolation to represent the solution branch in the parameter $s$. In Section~\ref{sec:data_representation}, we detail how to use the Fast Fourier Transform (FFT) for efficient representation of these parameterized families $\{ \bx(s): s \in [-1,1]\}$. Finally, in Section~\ref{sec:explicit_bounds} we present all the estimates required to verify the assumptions of the Newton–Kantorovich theorem (Theorem~\ref{thm:radii_polynomial}).

\subsection{Newton-Kantorovich theorem}\label{sec:newtonkantorovich}

We begin this section by recalling the parameterized zero-finding map $F(\cdot, s)$ defined in  \eqref{eq:ZFP}, and introduce the associated fixed-point operator $T : X \times [-1,1] \to X$ given by
\begin{equation}\label{fixed point operator}
	T(x,s) \bydef x - A(s)F(x,s)
\end{equation}
where $A(s)$ is an injective linear operator that serves as an approximate inverse of the Fréchet derivative $DF(\bx(s),s)$ of $F$ at a numerical approximation $\bx(s)$ and that satisfies $A(s)F(\cdot,s) : X \to X$, for all $s \in [-1,1]$. Note that $DF$ denotes the Fréchet derivative with respect to the $x$ variable. 
The idea behind this construction is to reformulate the problem of solving $F(x, s) = 0$ as a fixed-point problem: if $x^*(s) \in X$ satisfies $T(x^*(s), s) = x^*(s)$, then $x^*(s)$ is necessarily a zero of $F(\cdot, s)$, and hence a validated solution of our original problem. The goal is to establish that, for each $s \in [-1,1]$, the operator $T(\cdot, s)$ admits a unique fixed point in a ball of radius $r > 0$ centered at a numerically computed approximation $\bx(s) \in X$. The Newton–Kantorovich-type theorem stated below will provide the conditions under which such a fixed point exists and is unique, ensuring the rigorous validation of the heteroclinic connection at each parameter value $s$.

Before stating the Newton-Kantorovich Theorem, we introduce some notation. Let $B(X,Y)$ denote the space of bounded linear operators from a Banach space $X$ to another Banach space $Y$, and write $B(X)\bydef B(X,X)$. For $x \in X$ and $r>0$, denote the closed ball of radius $r$ centered at $x$ 
by $\overbar{B_r(x)} \bydef \{ y \in X : \|x-y\|_X \le r \} \subset X$.

\begin{thm}[{\bf Newton-Kantorovich Theorem}]\label{thm:radii_polynomial}
Let $X,Y$ be Banach spaces and consider $F:X \times [-1,1] \to Y$ be a $C^k$ Fréchet diﬀerentiable mapping. Suppose that $\bx(s) \in X$,  $A(s) \in B(Y,X)$ is injective and $A^\dagger(s) \in B(X,Y) $ for any $s\in[-1,1]$. Moreover the existence of positive bounds $Y_0, Z_0,Z_1$ and $Z_2$ satisfying for all $s \in [-1,1]$
\begin{align}
	\| A(s) F(\bx(s), s) \|_X &\leq Y_0 \label{bound_Y_0}, \\
	\| I - A(s) A^\dagger(s) \|_{B(X)} &\leq Z_0 \label{bound_Z_0}, \\
	\| A(s)[DF(\bx(s),s) - A^\dagger(s)] \|_{B(X)} &\leq Z_1 \label{bound_Z_1} , \\
	\| A(s) [ DF(c,s) - DF( \bx(s) ,s) ] \|_{B(X)} &\leq Z_2(r)r \label{bound_Z_2},	
\end{align}
for all $c \in \overbar{B_r(\bx(s))}$ and $r>0$. Define the {\em radii polynomial}
	\begin{align} \label{radii_poly}
		p(r) \bydef Z_2(r)r^2 - (1 - Z_0 - Z_1) + Y_0.
	\end{align}
	If there exists $r_0 > 0 $ such that $p(r_0) < 0$, then there exists a function 
	\begin{align*}
		\tilde{x}:[-1,1] \rightarrow \bigcup_{s\in [-1,1]} \overbar{B_{r_0}(\bx (s))}
	\end{align*}
	with $\tilde{x} \in C^k$ and such that 
	\begin{align*}
		F(\tilde{x}(s) , s ) = 0,  \mbox{ for all } s\in [-1,1].
	\end{align*}
	Furthermore, these solutions are unique.
\end{thm}
The proof of the theorem is a direct consequence of the Uniform Contraction Principle (e.g. see \cite{Chow1982}), and we refer the reader to \cite{Breden2015,Berg2010} for full details. 

\begin{rem}[{\bf Transversality Condition}] \label{rem:transversality}
A key advantage of a Newton–Kantorovich-type theorem such as Theorem~\ref{thm:radii_polynomial} is that it provides a uniform non-degeneracy condition for the family of solutions $\{\tilde{x}(s):s\in[-1,1]\}$. Specifically, it ensures that the Fréchet derivative $D_xF(\tilde{x}(s),s)$ is invertible for all $s \in [-1,1]$. This property can be used to verify the transversality of the intersection between $W^u(\Gamma)$, the unstable manifold of $\Gamma$, and $W^s(0)$, the stable manifold of the origin. As a consequence, the heteroclinic connections established in Theorem~\ref{thm:radii_polynomial} correspond to regular patterned fronts in the sense of Definition~\ref{defn:regular patterned front}. Related earlier transversality results that follow directly from Newton–Kantorovich-type arguments can be found in \cite{MR3068557,MR3207723,MR4068579,MR4658475}. 
\end{rem}

Having already formulated a zero-finding problem $F$ to characterize the heteroclinic connections, as well as the Banach space $X$ defined in \eqref{eq:Banach_Space_X} on which $F$ acts, the remaining steps required to apply Theorem~\ref{thm:radii_polynomial} are as follows: define the linear operators $A^\dagger(s)$ and $A(s)$, construct a numerical approximation $\bx(s)$ to the solution, and compute the associated bounds $Y_0, Z_0,Z_1$ and $Z_2$. These elements enable the construction of the radii polynomial \eqref{radii_poly}, from which we determine a value $r_0 > 0$ that rigorously validates the existence and local uniqueness of a solution $\tilde{x}(s)$ in a neighborhood of $\bx(s)$, for all $s \in [-1,1]$. By design, the regular patterned front $\mathcal{S}(s)$ described in \eqref{patterned_front} is a subset of the solution $\tilde{x}(s)$. The only remaining step is to verify a posteriori that the derivative of this front with respect to the continuation parameter $s$ does not vanish.
\begin{thm}\label{thm:smoothbranchaposteriori}
	Suppose the hypotheses of Theorem \ref{thm:radii_polynomial} are satisfied, yielding a continuous branch of solutions $\bigcup_{s \in [-1,1]} \tilde{x}(s) $. If 
	\begin{align}\label{cond_smooth}
		- \left( \frac{d\bx}{ds} (s) \cdot \eta_s(s)     \right) + r_0 \sup_{s \in [-1,1]} \left|  \frac{d\eta_s}{ds} (s) \right| < 0,
	\end{align}
	for all $s\in [-1,1]$, then the branch is a smooth curve and $\frac{d \tilde{x}}{ds}(s) \neq 0$ for $s \in [-1,1]$.
\end{thm}
\begin{proof}
	Since $\tilde{x}(s)$ is the root of the map $F$, this implies that $Q(\tilde{x}(s);s) = (\tilde{x}(s) - \bx(s)) \cdot \eta_s(s) = 0$. Differentiating with respect to $s$ and using (\refeq{cond_smooth}) give us
	\begin{align*}
		- \frac{d \tilde{x}}{ds}(s) \cdot \eta_s(s) &= - \frac{d \bx }{ds}(s) \cdot \eta_s(s) + (\tilde{x}(s) - \bx(s)) \cdot \frac{d \eta_s}{ds}(s) \\
		& \leq  - \frac{d \bx }{ds}(s) \cdot \eta_s(s) +  r_0 \sup_{s \in [-1,1]} \left|  \frac{d\eta_s}{ds} (s) \right| 
		< 0 ,
	\end{align*}
	for all $s\in [-1,1]$. This implies that $\frac{d \tilde{x}}{ds}(s) \neq 0$ for all $s\in [-1,1]$, yielding smoothness.
\end{proof}

Since the constituents of a regular patterned front $\mathcal{S}(s)$ form a subset of the solution $\tilde{x}(s)$, we cannot directly apply Theorem~\ref{thm:smoothbranchaposteriori} to establish the desired result. Instead, we must impose a specific structure on the vector $\eta_s(s)$.
\begin{cor}\label{cor_cond_smooth}
	Suppose the hypothesis of Theorem~\ref{thm:smoothbranchaposteriori} are satisfied, and that the components of $\eta_s(s)$ satisfy $(\eta_s(s))_i \neq 0$ for $i \in \{1,2,3,4,5,20,21,22,23,24\}$ and $(\eta_s(s))_i = 0$ otherwise. Then, $\mathcal{S}_s(s) \neq 0$ for all $s \in [-1,1]$.
\end{cor}
\begin{proof}
Let $\tilde{\mathcal{S}}(s) \bydef \left(
		\tilde{\alpha}(s) , \tilde{\gamma}_1(s) , \tilde{\gamma}_2(s) , \tilde{\gamma}_3(s)  , \tilde{\gamma}_4(s) , \tilde{\omega}(s) , \tilde{a}_1(s) , \tilde{a}_2(s) , \tilde{a}_3(s) , \tilde{a}_4(s)
\right)$. Note that the set of components of $\tilde{\mathcal{S}}(s)$ is a subset of the set of components of the regular patterned front $\mathcal{S}(s)$. Let
$\tilde{\eta}_s(s) \bydef \left(
		(\eta_s(s))_1 , \cdots  ,  (\eta_s(s))_{5} ,  (\eta_s(s))_{20} ,  \dots ,  (\eta_s(s))_{24} \right)$. By construction of the parameterized family of tangent vectors $\eta_s(s)$, we have
	\begin{align*}
		- \frac{d \tilde{x}}{ds}(s) \cdot \eta_s(s) = -  \frac{d\tilde{\mathcal{S}}}{ds} (s) \cdot \tilde{\eta}_s(s) <  0.
	\end{align*}
	Therefore, $\frac{d\tilde{\mathcal{S}}}{ds}(s) \neq 0$ for all $s \in [-1,1]$, implying that the derivative of the regular patterned front $\mathcal{S}(s)$ with respect to $s$ does not vanish.
\end{proof}
By choosing $\eta_s(s)$ to satisfy the conditions of Corollary~\ref{cor_cond_smooth}, and using the fact that both $\eta_s(s)$ and $\bx(s)$ are known finite polynomials, the inequality \eqref{cond_smooth} can be rigorously verified once a suitable radius $r_0$ is determined via the Newton-Kantorovich Theorem~\ref{thm:radii_polynomial}.

\subsection{Projections, data representation, and the operators \boldmath$A^\dagger(s)$\unboldmath~and~\boldmath$A(s)$\unboldmath}\label{sec:data_representation}

In this section, we introduce the finite-dimensional projections used to carry out our computations, we develop an efficient computational framework that leverages the Fast Fourier Transform (FFT) to obtain a parameterized representation of the numerical approximations over the parameter space $s \in [-1,1]$, and we define the parameter-dependent linear operators $A^\dagger(s)$ and $A(s)$.

There are two main challenges in finitely representing the data needed for the application of computer-assisted methods. First, the solution forms a continuous branch that depends on the parameter $s \in [-1,1]$. Second, for any fixed value of $s$, the solution $x \in X$ lies in an infinite-dimensional Banach space, and involves a countably infinite number of components. Our approximation strategy begins by fixing $s$, and constructing an approximate solution lying in some fixed finite-dimensional subspace of $X$. This approximation thus has a finite representation, which is then extended to the entire interval $s \in [-1,1]$ by interpolating the data obtained from a finite set of sampled parameter values.

Throughout the remainder of this work, we let $\bx(s) \in X$ denote a numerical approximation of the solution to the zero-finding problem, satisfying $F(\bx(s), s) \approx 0$ for each fixed $s \in [-1,1]$.

As noted above, we require our numerical approximation to lie in a finite-dimensional subspace of $X$. To this end, we construct the approximation $\bx(s)$ such that each component $(\bx(s))_i$, corresponding to a coefficient sequence in a series representation, contains only finitely many nonzero terms. For instance, recalling the labeled variables in Table~\ref{table_variables}, for any $x\in X$, the components $(x_2,x_3,x_4,x_5)$ represent the cosine/sine Fourier coefficients of the periodic solution, and fixing $N_\gamma \in \N$, we define the finite projection operator $\pi^{N_i}$ acting on $ x_i $ for $i = 2,3,4,5$ element-wise by 
\[
(\pi^{N_i} x_i)_k  = 
\begin{cases}
	(x_i)_k, & \text{for } k \leq N_\gamma, \\
	0, & \text{otherwise}.
\end{cases}
\]
Analogous projection operators are defined for other components of the solution, such as those associated with the bundle, connecting orbit, unstable manifold, and stable manifold. These projections use respective truncation orders $ N_v , N_a , N_{w_F}, N_{w_T}, N_p\in \N$. We then define a global projection operator $\pi^N:{X} \rightarrow {X}$ acting component-wise on $x \in X$, by
\[
(\pi^N x )_i = 
\begin{cases}
	x_i, & \text{if }  i = 1,6,15,16,17,18,19,20,25,\\
	\pi^{N_i} x_i,	& \text{otherwise}.
\end{cases}
\]
To describe the discarded (infinite) tail components, we define the complementary projections
\[
\pi^\infty_i \bydef I - \pi^{N_i} \quad \mbox{ and } \quad \pi^\infty \bydef I - \pi^{N}.
\]
By requiring that $\pi^N \bx(s) = \bx(s) \in X$, we ensure that our numerical approximation contains only a finite number of nonzero terms. This guarantees that the solution can be stored, manipulated, and represented within a finite-dimensional computational framework. To align the notation with the application, we define the truncation operator $\pi_{{N}_i}$, which maps a component $x_i \in X_i$ to a finite-dimensional vector by removing the zero-mode from the projection $\pi^{N_i} x_i$. We also define the global truncation operator $\pi_{N}$ acting component-wise, as
	\[
	\pi_{N} x = 
	\begin{bmatrix}
		\pi_{{N}_1}x_1 & \pi_{{N}_2}x_2  & \cdots &  \pi_{{N}_{29}}x_{29} 
	\end{bmatrix} .
	\]
The next challenge is to represent the entire continuous solution branch over the parameter range $s\in [-1,1]$ in a form suitable for use in Theorem~\ref{thm:radii_polynomial}. Let $N_s \in \mathbb{N}$, and suppose we are given a collection of $N_s+1$ data points $ \{ \bx^{0} , \bx^{1},\dots,\bx^{{N_s}} \} \subset X$, associated with parameter values $  \{ s_0,\dots,s_{N_s} \} \subset [-1,1]$, where $-1 = s_0 < s_1 < \dots  < s_{N_s  } = 1 $. Each data point $\bx^i$ approximates a solution of the equation $F(x,s_i)=0$ at the corresponding parameter value $s=s_i$. We define $\bx(s):[-1,1] \rightarrow X$ as the unique polynomial of degree at most $N_s$ that interpolates the data, that is $\bx(s_i) =\bx^i$ for all $i=0,\dots,N_s$. Such an interpolating polynomial exists and is continuous over the interval. Moreover, if each data point $\bx^i$ satisfies the projection condition $\pi^N \bx^i = \bx^i$, then the resulting polynomial $ \bx(s)$ is finite-dimensional and hence suitable for computation. In this work, we represent $\bx(s)$ using a truncated Chebyshev series of the form
\[
\bar{x}(s) = \mx_0 +2 \sum_{n = 1 }^{N_s} \mx_n T_n({s}) 
\]
where, as before, $T_n(s)$ denotes the $n^{th}$ Chebyshev polynomial of the first kind, and where each coefficient $\mx_i$ belongs to the Banach space $X$. While the choice of basis in the expansion in $s$ is in principle arbitrary, the Chebyshev basis is particularly advantageous due to the rapid decay of its coefficients, which enables efficient and accurate numerical approximation. In particular, the Fast Fourier Transform (FFT) can be used to rigorously compute the Chebyshev coefficients, resulting in reduced computational cost and runtime. A detailed description of the construction of the interpolation grid is given in Section~\ref{sec:opti}. To minimize interpolation error in $\bx(s)$, the data points are distributed according to arc-length along the solution branch. Specifically, the initial point $\bx^0$, the total number of points $N_s+1$, and an approximate value of the total arc-length are prescribed. A pseudo-arclength continuation algorithm is then used to compute the remaining points $\bx^i$.

Finally, we observe that the norm of the interpolated polynomial $\bx(s)$ over the interval $[-1,1]$ can be computed using only finitely many operations. Specifically, using that $|T_n(s)| \le 1$ for all $n$, we have the estimate
\begin{equation} \label{norm_s_to_fin}
	\| \bx(s) \|_X =  \left\| \mx_0 +2 \sum_{n = 1 }^{N_s} \mx_n T_n({s}) \right\|_{X} \leq \| \mx_0 \|_X +2 \sum_{n = 1 }^{N_s} \| \mx_n \|_{X},
\end{equation}
where the right-hand side is independent of the parameter $s$. Moreover, since the Chebyshev coefficients were obtained from data points that are invariant under the projection map $\pi^N$, it follows that the coefficients $\mx_n$ also satisfy $ \pi^N \mx_n = \mx_n $. Consequently, the computation of each norm $\| \mx_n \|_X$ is finite and can be evaluated explicitly.

Having established a finite-dimensional representation of our data and introduced the required finite-dimensional projections, we are now in a position to formally define the linear operators $A^\dagger(s)$ and $A(s)$. As suggested by the structure of the bounds in the Newton-Kantorovich Theorem~\ref{thm:radii_polynomial}, the radii polynomial $p(r)$ defined in \eqref{radii_poly} is minimized when the operator $A^\dagger(s)$ closely approximates the Fréchet derivative $DF(\bx(s),s)$ and when $A(s)$ approximates its inverse, that is
\[
 A^\dagger(s) \approx   DF( \bx(s) , s ) \quad \text{and} \quad A(s) \approx \left( A^\dagger(s) \right)^{-1}. 
\]
To enable practical computation, we construct $A^\dagger(s)$ and $A(s)$ so that they permit a transition between the infinite-dimensional Banach space $X$ and a corresponding finite-dimensional vector space. Since each component of the nonlinear map $F$ is a polynomial in its arguments, there exist integers $d_i$ such that the composition $F_i(\bx(s))$ can be expressed as a Chebyshev series of degree at most $d_iN_s$. Letting $d = \max_{i = 1,\dots,29} d_i$, it follows that each entry of the Fréchet derivative $DF(\bx(s),s)$ is a Chebyshev polynomial of degree at most $(d-1)N_s$. Using this feature to construct $A^\dagger(s)$, we define a collection of linear operators
\[
\left\{A^\dagger_0, \dots ,A^\dagger_{(d-1)N_s} \right\},
\]
which serve as Chebyshev coefficients of the interpolant $A^\dagger(s)$. Each linear operator $A^\dagger_n$ is block-structured as
\begin{equation}\label{eq:A dagger}
	A^\dagger_n = \begin{pmatrix}
		(A^{\dagger}_n)_{1,1} & \cdots & (A^{\dagger}_n)_{1,29}  \\
		\vdots & \ddots & \vdots \\
		(A^{\dagger}_n)_{29,1} & \cdots & (A^{\dagger}_n)_{29,29} 
	\end{pmatrix}
	\quad \text{where each block satisfies } (A^\dagger_n)_{i,j} = \begin{pmatrix}
		(A^\dagger_n)_{i,j}^N & 0 \\
		0 & \delta_{i,j}\Lambda_{i}^n
	\end{pmatrix}.
\end{equation}
%
%
%
Here, $(A^\dagger_n)_{i,j}\approx  \pi_{N_i} D_{x_j}F_i(\bx^n)\pi^{N_j}$ and $\delta_{i,j}$ is the Kronecker delta. Each operator $\Lambda_i^n$ is a diagonal operator on the tail space, capturing the spectral structure of the linearization and defined entry-wise as
\[
{\footnotesize
(\Lambda^n_{i} )_{kk}  = 
\begin{cases}
		k + N_\gamma + 1,	& \text{if } 2 \leq i \leq 5 \text{ and } k \in \N, \\
		\mi (k + \sign(k) (N_v +1 )) + \bx^i_6,	& \text{if } 7 \leq i \leq 10 \text{ and } k \in \Z, \\
		\mi (k_1 + \sign(k_1) (N_{w,1} +1 ) ) + \bx^i_6 (k_2 + (N_{w,2} +1 )),& \text{if } 11 \leq i \leq 14  \text{ and } k = (k_1,k_2) \in \Z \times \N, \\
		2( k + N_a +1 ), 	& \text{if } 21 \leq i \leq 24 \text{ and } k \in \N, \\
		\bx_{25,1} ( k_1 + N_p +1 ) + \bx_{25,6} ( k_2 + N_p +1 ),	& \text{if } 26 \leq i \leq 29 \text{ and } k = (k_1,k_2) \in \N^2, 
\end{cases}
}
\]
and zero elsewhere. In practice, the finite matrices $(A^\dagger_n)_{i,j}^N$ are computed by evaluating the derivatives $\pi_{N_i} D_{x_j}F_i(\bx^n)\pi^{N_j}$ at each grid point $\bx^n$, using the same Chebyshev interpolation grid of  $(d-1)N_s +1$ points $\{ \bx^0, \dots, \bx^{(d-1)N_s} \}$ as in the construction of $\bx(s)$.

Finally, the operator $A^\dagger (s)$ is defined by interpolating the block matrices $ \{A^\dagger_0,\dots,A^\dagger_{(d-1)N_s} \} $ using a Chebyshev expansion, and is denoted as
\begin{equation}
	A^\dagger(s) = \begin{pmatrix}
		A^{\dagger}_{1,1}(s) & \dots & A^{\dagger}_{1,29}(s)  \\
		\vdots & \ddots & \vdots \\
		A^{\dagger}_{29,1}(s) & \dots & A^{\dagger}_{29,29}(s)
	\end{pmatrix} \mbox{ with } A^{\dagger}_{i,j}(s) = \begin{pmatrix}
		A^{N^\dagger}_{i,j}(s) & 0 \\
		0 & \delta_{i,j} \Lambda_i (s)
	\end{pmatrix}.
\end{equation}
We now define the operator $A(s)$, which serves as an approximate inverse to $ A^{\dagger}(s)$. To this end, we first construct a sequence of matrices
\begin{equation} \label{eq:grid_of_approximated_inverses}
\left\{ A_0^N,\dots,A_{N_s}^N \right\}
\end{equation}
where each matrix $A_n^N$ is a numerical approximation of the inverse 
\[
A_n^N \approx \left( \pi_{N}A^{N^\dagger}_n \pi^N \right)^{-1}.
\]
Using the grid of approximated inverses \eqref{eq:grid_of_approximated_inverses}, we define the finite-dimensional part of the operator $A(s)$, denoted $A^N(s)$, as the Chebyshev interpolation of the sequence$\{ A_0^N,\dots,A_{N_s}^N \}$. In addition to the finite part, the tail operator $\Lambda_i(s)$, introduced in the definition of $A^\dagger(s)$, must also be inverted to define $A(s)$ completely. Provided that the truncation orders $N_\gamma, N_v, N_{w,F},N_{w,T}, N_a, N_p $ are sufficiently large, one can verify that the diagonal entries of $\Lambda_i(s)$ remain strictly nonzero for all $s \in [-1,1]$. Hence, each $\Lambda_i(s)$ is invertible, and we denote its inverse by $\Lambda_i^{-1}(s)$. Together, $A^N(s)$ and $\Lambda_i^{-1}(s)$ define the full linear operator $A(s)$ as
\begin{equation}
	A(s) = \begin{pmatrix}
		A^{}_{1,1}(s) & \dots & A^{}_{1,29}(s)  \\
		\vdots & \ddots & \vdots \\
		A^{}_{29,1}(s) & \dots & A^{}_{29,29}(s)
	\end{pmatrix} \mbox{ with } A^{}_{i,j}(s) = \begin{pmatrix}
		A^{N}_{i,j}(s) & 0 \\
		0 & \delta_{i,j} \Lambda_i^{-1} (s)
	\end{pmatrix}.
\end{equation}
This construction yields an operator $A(s)$ such that the composition $A(s) A^{\dagger} (s)\approx I$ holds uniformly for all $s \in [-1,1]$, where $I$ denotes the identity operator on $X$. With the operators $A^\dagger(s)$ and $A(s)$ now defined, we are ready to proceed with the computation of the bounds $Y_0, Z_0, Z_1$ and $Z_2$ as required in Theorem~\ref{thm:radii_polynomial}.

\subsection{Explicit bounds for the radii polynomial} \label{sec:explicit_bounds}

In this section, we present the explicit formulas required to compute the bounds $Y_0, Z_0, Z_1$ and $Z_2$ appearing in Theorem~\ref{thm:radii_polynomial}. These bounds quantify the residual error of the approximate solution, the defect in the approximate inverse, and the Lipschitz properties of the Fréchet derivative. Their rigorous estimation is central to validating the existence and uniqueness of a true solution in a neighborhood of the numerical approximation.

Throughout this section, to simplify the presentation, we omit the explicit dependence of the map $F(x(s), s)$ on the parameter $s \in [-1,1]$ and denote it simply by $F(x(s))$.

\subsubsection{The bound \boldmath$Y_0$\unboldmath}

In this section, we develop an explicit strategy to compute the bound $Y_0$ satisfying \eqref{bound_Y_0}. To begin, we first establish an upper bound on the norm of the inverse operator $\Lambda_i^{-1}(s)$ as follows
\[
{\footnotesize
\| \Lambda_i^{-1}(s) \|_{X_i} \leq \overline{\Lambda}_i \bydef
\begin{cases}
(N_\gamma + 1)^{-1}, & \text{if } 2 \leq i \leq 5, \\
(N_v + 1)^{-1}, & \text{if } 7 \leq i \leq 10, \\
\max\left( (N_{w_F} + 1)^{-1}, \left((N_{w_T} + 1) \inf\limits_{s \in [-1,1]} |\bx_6(s)|\right)^{-1} \right), & \text{if } 11 \leq i \leq 14, \\
(2N_a + 1)^{-1}, & \text{if } 21 \leq i \leq 24, \\
(N_p + 1)^{-1} \max \left( 
\sup\limits_{s \in [-1,1]} |\Re(\bx_{25,1}(s))|^{-1}, 
\sup\limits_{s \in [-1,1]} |\Re(\bx_{25,6}(s))|^{-1} 
\right), & \text{if } 26 \leq i \leq 29, \\
0, & \text{otherwise}.
\end{cases}
}
\]
Here, $\bx_{25,1}(s)$ and $\bx_{25,6}(s)$ correspond to the Chebyshev interpolation of the stable eigenvalues $\bar{\lambda}_1(s)$ and $\bar{\lambda}_2(s)$ of $D_u \Psi(0,\bar{\alpha}(s))$. Therefore, we obtain 
\[
\| A(s)F(\bx(s)) \|_X = \max_{i = 1,\dots,29} \mu_i \left\| \sum_{j = 1}^{29} A_{i,j}(s) F_j(\bx(s))\right\|_{X_i} \leq \max_{i = 1,\dots,29} \mu_i \sum_{j = 1}^{29} \left\|  A_{i,j}(s) F_j(\bx(s)) \right\|_{X_i},
\]
with
\begin{align} 
\nonumber
\left\|  A_{i,j}(s) F_j(\bx(s)) \right\|_{X_i} & \leq  \left\| \pi^{N_i} A_{i,j}(s) F_j(\bx(s)) \right\|_{X_i} + \left\| \pi^\infty_i  A_{i,j}(s) F_j(\bx(s)) \right\|_{X_i} \\
&\leq  \left\| \pi^{N_i} A_{i,j}(s) F_j(\bx(s)) \right\|_{X_i} + \delta_{i,j} \overbar{\Lambda}_i \left\|  \pi^\infty_i F_j(\bx(s)) \right\|_{X_i}.
\label{eq:Y_bound_expansion}
\end{align}

For the first term in the right-hand side of \eqref{eq:Y_bound_expansion}, we recall that both $A_{i,j}(s)$ and $\bx(s)$ are finite Chebyshev series of order $N_s$ and $F_i(\bx(s))$ is a Chebyshev series of order at most $d_iN_s$. It follows that the product $ A_{i,j}(s)  F_i(\bx(s))$ is itself a finite Chebyshev series of order at most $(d_i+1)N_s$. To compute the Chebyshev coefficients of the product $ \pi^{N_i} A_{i,j}(s)  F_i(\bx(s))$, we proceed by partitioning the interval $[-1,1]$ into a grid of $(d_i+1)N_s+1$ points, denoted by $\{s_0,\dots,s_{(d_i+1)N_s}\}$, and evaluate the product at each grid point
\[
\pi^{N_i} A_{i,j}(s_k)  F_i(\bx(s_k)) \quad \mbox{for} \quad k = 0,\dots,(d_i+1)N_i.
\]
From this sampled data, we can compute the Chebyshev coefficients of $ \pi^{N_i} A_{i,j}(s)  F_i(\bx(s))$ using a rigorous FFT algorithm combined with interval arithmetic (e.g. see \cite{MR2269503,MR3833658,MR3167726}). Finally, as established in (\refeq{norm_s_to_fin}), the norm $ \| \pi^{N_i} A_{i,j}(s)  F_i(\bx(s)) \|_{X_i}$ can be computed finitely and rigorously using interval arithmetic techniques.
	 
For the second term in the right-hand side of \eqref{eq:Y_bound_expansion}, as previously mentioned, each component $F_i$ is a polynomial of order at most $d_iN_s$. We apply the same strategy used for the first term by partitioning the interval $[-1,1]$ into the grid of $d_iN_s+1$ points, denoted by $\{s_0,\dots,s_{d_iN_s}\}$. Evaluating the function at these nodes yields the set $\{F_i(\bx(s_0),\dots,F_i(\bx(s_{d_iN_s})\}$. Now, let $x_i,y_i \in X_i$ be such that $ \pi^{N_i} x_i = x_i $ and $ \pi^{N_i} y_i = y_i $. Then the convolution $x_i * y_i$, as defined in one of the equations \eqref{convo_gamma}, \eqref{convo_v}, \eqref{convo_w} or \eqref{convo_p}, has only a finite number of nonzero modes. Since $F_i$ is polynomial, it follows that each $F_i(\bx(s_k))$ also has a finite Chebyshev expansion. Consequently, the computation of $\left\| \pi^\infty_i F_i(\bx(s)) \right\|_{X_i}$ is also finite. Hence, using interval arithmetic, we can compute a rigorous bound $Y_0$ such that \eqref{bound_Y_0} holds.

\subsubsection{The bound \boldmath$Z_0$\unboldmath}

In this section, we provide an explicit computational strategy for estimating the bound $Z_0$ appearing in \eqref{bound_Z_0}, which involves bounding the operator norm $\| I - A(s) A^\dagger(s) \|_{B(X)}$. 

Note first that each entry of the composed operator $A(s) A^\dagger(s)$ is given by
\[
[A(s) A^\dagger(s)]_{i,j} = \sum_{k = 1}^{29} A_{i,k}(s) A^\dagger_{k,j}(s),
\]
for $1 \leq i,j \leq 29$, with
\[
A_{i,k}(s) A^\dagger_{k,j} (s) 
= \begin{pmatrix}
	 	A_{i,k}^N(s) A_{k,j}^{N^\dagger}(s) & 0 \\
	 	0 & \delta_{i,k} \delta_{k,j} \Lambda_{i}^{-1} (s)  \Lambda_{k} (s) 
\end{pmatrix}.
\]
By the construction of the operators, it follows that
\begin{equation} \label{Z_0_eq}
	[I - A(s)A^\dagger(s)]_{i,j} =  \pi^{N_i} I_{i,j} \pi^{N_j} - \pi^{N_i} \left( \sum_{k= 1}^{29}A_{i,k}(s) A_{k,j}^{\dagger}(s) \right)  \pi^{N_j}.
\end{equation}
In other words, the entries $[I - A(s)A^\dagger(s)]_{i,j}$ contain only finitely many nonzero terms, corresponding to $\pi_{N_i}I_{i,j} \pi_{N_j} - A_{i,k}^N(s) A_{k,j}^{N^\dagger}(s)$, which is a Chebyshev polynomial of degree at most $dN_s$. As a result, evaluating the bound $\| I - A(s) A^\dagger(s) \|_{B(X)}$ reduces to computing the norm of a finite-dimensional matrix. However, among all the bounds appearing in Theorem~\ref{thm:radii_polynomial}, the computation of $Z_0$ is by far the most time-consuming. This is primarily because it involves full matrix-matrix products of the form $A(s) A^\dagger(s)$, whereas the other bounds can be reduced to matrix–vector operations. By construction, the bound $Z_0$ is expected to be small. Nevertheless, the Newton-Kantorovich Theorem does not require it to be computed with high precision. In practice, it suffices to verify that $Z_1 + Z_0 < 1$, so obtaining $Z_0$ to machine precision is unnecessary. To address this, we introduce a strategy that yields a sufficiently accurate upper bound for $Z_0$, while significantly reducing both computational time and memory usage. We begin by decomposing the $(d-1)N_s$-degree Chebyshev polynomial $\bar{A}^\dagger(s)$ into a truncated and a tail component as
\[
A^\dagger(s) \bydef A_{N_s}^\dagger(s) + A_\infty^\dagger(s),
\]
where $A_{N_s}^\dagger(s)$ denotes the degree-$N_s$ truncation of the Chebyshev expansion, and $A_\infty^\dagger(s)$ corresponds to the tail. The next step is to derive a bound for the norm
\[
\| I - A(s) A^\dagger(s) \| = \left\| I - A(s)\left(A_{N_s}^\dagger(s) + A_\infty^\dagger(s)\right) \right\| 
	\leq \| I - A(s) A_{N_s}^\dagger(s) \| + \| A(s) \| \cdot \| A_\infty^\dagger(s) \|.
\]
The term $I - A(s) A_{N_s}^\dagger(s)$ yields a Chebyshev polynomial of degree at most $2N_s$, which significantly reduces the number of points required in the interpolation grid for its evaluation. Let $I - A_k (A_{N_s}^\dagger)_k$ denote a representative element of this grid. Each entry of the matrices $A_k$ and $(A_{N_s}^\dagger)_k$ is stored as an interval consisting of a midpoint and a radius. Let $\bar{A}_k$ and $(\bar{A}_{N_s}^\dagger)_k$ denote the midpoints of $A_k$ and $(A_{N_s}^\dagger)_k$, respectively. We define $fl\left(\bar{A}_k (\bar{A}_{N_s}^\dagger)_k\right)$ as the floating-point evaluation of the matrix product. Then, there exists a constant $\varepsilon_k > 0$ such that
\[
\left\| I - \bar{A}_k (\bar{A}_{N_s}^\dagger)_k \right\| \leq \left\| I - fl\left(\bar{A}_k (\bar{A}_{N_s}^\dagger)_k\right) \right\| + \varepsilon_k.
\]
A detailed justification of this estimate is provided in Appendix~\ref{appA:flaoting_point}. This result enables us to perform the matrix multiplication $\bar{A}_k (\bar{A}_{N_s}^\dagger)_k$ using standard floating-point arithmetic while accounting for rounding errors via the constant $\varepsilon_k$. The value of $\varepsilon_k$ depends explicitly on the norms of the matrices involved and the precision used during the computation. In cases where the matrices are poorly conditioned, the product can be computed in higher precision (e.g., using \texttt{BigFloat}), which typically remains more efficient than carrying out the entire computation using interval arithmetic alone. Hence, we obtain a rigorous bound $Z_0$ satisfying \eqref{bound_Z_0}.

\subsubsection{The bound \boldmath$Z_1$\unboldmath}

In this section, we provide an explicit computational strategy for estimating the bound $Z_1$ appearing in \eqref{bound_Z_1}, which involves bounding the operator norm $\|A(s)[DF(\bx(s),s) - A^\dagger(s)]\|_{B(X)}$. 

Given $x_i \in X_i$ and $x_j \in X_j$, we define the generalized discrete convolution component-wise by
\begin{align*}
	x_i * y_j  = 
	\left\{	
	\begin{tabular}{cl}
		$ y_i x_i $ & if $ X_j = \mathbb{K} $, \\
		$ x_i * y_i $ & if $X_i = X_j  \neq  \mathbb{K}$,
	\end{tabular}
	\right.
\end{align*}
 where the discrete convolution $x_i * y_i$ is defined according to equations \eqref{convo_gamma}, \eqref{convo_v}, \eqref{convo_w} and \eqref{convo_p}. To obtain $Z_1$, note first that
\begin{align*}
	\left\| A(s)[ DF (\bx(s)) - A^\dagger(s) ] \right\|_{B(X)} &\leq \max_{i= 1,\dots,29} \mu_i \sum_{j = 1}^{29} \frac{1}{\mu_j} \left\| \left( A(s)[ DF (\bx(s)) - A^\dagger(s)] \right)_{i,j} \right\|_{B(X_j,X_i)},
\end{align*}
where each component can be separated into two bounds, namely
\begin{align} \nonumber
	\left\| \left( A(s)[ DF (\bx(s)) - A^\dagger(s)\right)_{i,j} \right\|_{B(X_j,X_i)} \leq & \left\| \pi^{N_i} (A(s) DF \left(\bx(s)) \right)_{i,j}\pi^\infty_j \right\|_{B(X_j,X_i)} \\
	& \quad + \left\|  A_{i,i}(s) \pi_i^\infty \left( DF(\bx(s)) - A^\dagger (s) \right)_{i,j}\right\|_{B(X_j,X_i)} \\
	\label{eq:expansion_Z1_bound}
	&\le \left\| \pi^{N_i} (A(s) DF(\bx(s)))_{i,j}\pi^\infty_j \right\|_{B(X_j,X_i)} \\
& \quad + 	\sup_{\|h_j\|_{X_j} \leq 1} \overbar{\Lambda}_i \left\|  \pi_i^\infty \left( DF(\bx(s)) - A^\dagger (s) \right)_{i,j} h_j\right\|_{X_i}
\end{align}
%
%
The remainder of this section is devoted to bounding the two terms on the right-hand side of \eqref{eq:expansion_Z1_bound}. 

We begin by bounding the second term in the right-hand side of \eqref{eq:expansion_Z1_bound}. The analysis of each expression $ \pi_i^\infty \left( DF(\bx(s)) - A^\dagger (s) \right)_{i,j} h_j$ follows a largely uniform structure across most indices $(i,j)$, with a few notable exceptions that require a more detailed treatment. We denote the set of these exceptional indices by $\mathcal{I}$, defined as 
\begin{align*}
   \mathcal{I} \bydef \left\{(i,j) \in \N^2 ~ | ~ i = 11,12,13,14 \mbox{ and } j = i-9 \mbox{ or } j = i-4 \mbox{ or } j = i \right\} .
\end{align*} 
Letting
\[
\beta_{i,j} \bydef
	\begin{cases}
		1, & \text{if } X_i = X_j \text{ or } X_i = \mathbb{K} \text{ or } X_j = \mathbb{K}, \\
		0,  & \text{otherwise},
	\end{cases}
\]
and by construction of the map $F$ in \eqref{eq:ZFP} (see also \eqref{eq:zero_finding_f}), if $(i,j) \notin \mathcal{I}$, there exist elements $\bz_{i,j}(s) \in X_i \times [-1,1]$  with finitely many nonzero modes such that the tail of $[DF(\bx(s)) - A^\dagger(s)]_{i,j}$ acting on $h_j \in X_j$ is given by the tail of the discrete convolution $\beta_{i,j} \bz_{i,j}(s)*h_j$. That is,
\begin{align}\label{product_operator}
	\pi_i ^\infty \left( [DF(\bx(s)) - A^\dagger(s) ]_{i,j}h_j \right) = \beta_{i,j}  \pi_i^\infty \left(  \bz_{i,j}(s)*h_j\right).
\end{align}
Thus, we can bound
\begin{align*}
	\sup_{\|h_j\|_{X_j} \leq 1} \overbar{\Lambda}_i \left\|  \pi_i^\infty \left( DF(\bx(s)) - A^\dagger (s) \right)_{i,j} h_j\right\|_{X_i} &=  \sup_{\|h_j\|_{X_j} \leq 1} \overbar{\Lambda}_i \beta_{i,j} \left\|  \pi_i^\infty \left( \bz_{i,j} (s)* h_j \right) \right\|_{X_i}  \\
	&\leq \overbar{\Lambda}_i \beta_{i,j} \left\|  \bz_{i,j} (s) \right\|_{X_i} .
\end{align*}
Hence, for the indices $(i,j) \notin \mathcal{I}$, we define the bound 
\begin{equation} \label{eq:Zinfty_ij_not_in_I}
Z^{\infty}_{i,j} \bydef \sup\limits_{s \in [-1,1]} \overbar{\Lambda}_i \beta_{i,j} \left\|  \bz_{i,j} (s) \right\|_{X_i}
\end{equation}
which can be computed rigorously using a similar strategy as in \eqref{norm_s_to_fin}. While the case $(i,j) \in \mathcal{I}$ is comparable, the product cannot be directly expressed as a discrete convolution. Instead, for $i = 11,12,13,14$ and $j = i-9$, we have 
\[
 \pi_i^\infty \left( DF(\bx(s)) - A^\dagger (s) \right)_{i,i-9} h_{i-9} = 
	 \begin{cases}
	(h_{i-9})_{|n|,0}, & \text{if } |n| > N_j \text{ and } m = 0, \\
	0, & \text{otherwise},
\end{cases}
\]
and thus 
\begin{equation} \label{eq:Zinfty_ij_in_I_1}
		\sup_{\|h_j\|_{X_j} \leq 1} \overbar{\Lambda}_i \left\|  \pi_i^\infty \left( DF(\bx(s)) - A^\dagger (s) \right)_{i,i-9} h_{i-9}\right\|_{X_i}  \leq \overbar{\Lambda}_i \bydef Z^{\infty}_{i,i-9} .
\end{equation}
Similarly, if $j = i-4$ or $j = i$, we have 
\begin{equation} \label{eq:Zinfty_ij_in_I_2}
	\sup_{\|h_j\|_{X_j} \leq 1} \overbar{\Lambda}_i \left\|  \pi_i^\infty \left( DF(\bx(s)) - A^\dagger (s) \right)_{i,j} h_{j}\right\|_{X_i}  \leq \overbar{\Lambda}_i \bydef Z^{\infty}_{i,j} .
\end{equation}

Next, we proceed to bound the first term in the right-hand side of \eqref{eq:expansion_Z1_bound}. Notice that
\[
\left\| \pi^{N_i} (A(s) DF \left(\bx(s)) \right)_{i,j}\pi^\infty_j \right\|_{B(X_j,X_i)} \leq 
\sup_{\|h_j\|_{X_j} \le 1}
\sum_{ k = 1}^{29} \left\| \pi^{N_i}  A_{i,k}(s) \pi^{N_k} D_jF_k (\bx(s))\pi^\infty_j h_j \right\|_{X_i }
\]
and as with the second term, the computation of the expressions $\pi^{N_i}  A_{i,k}(s) \pi^{N_k} D_jF_k (\bx(s))\pi^\infty_j h_j$ is divided into two distinct cases. To this end, we define the finite set of indices $\mathcal{J} \subset \N^2$ corresponding to these cases as
\[
(k,j) \in  \mathcal{J} \Leftrightarrow  k,j \in \{ 1,\dots,29\} \mbox{ and } 
	\begin{cases}
		k = 15,16,17,18 \text{ and } j = k+6 \text{ or } j = k+11, \\
		k = 23,24,25,26 \text{ and } j = k -12 \text{ or } j = k.
	\end{cases}
\]
This set contains the indices corresponding to the partial derivatives for which the projected Jacobian $\pi^{N_k} D_jF_k (\bx(s))\pi^\infty_j$ has infinitely many nonzero constant terms. Intuitively, this corresponds to each entry in a row of $  D_jF_k (\bx(s))$ being constant. First, let consider the cases $(k,j) \notin \mathcal{J}$, then by structure of the problem the operator $\pi^{N_k}  D_jF_k \bx(s)\pi^\infty_j h_j $ is an operator that can be represented by a convolution $\beta_{k,j} \bz_{k,j} (s)$ similar to (\refeq{product_operator}) such that
\begin{align*}
	\pi^{N_k}  D_jF_k (\bx(s))\pi^\infty_j h_j =\beta_{k,j}  \pi^{N_k} \left(\bz_{k,j} (s) * \pi^\infty_j h_j \right).
\end{align*}
For each element of $ \left| \pi^{N_k} \left(\bz_{k,j} (s) * \pi^\infty_j h_j \right)_{n} \right|$, we can compute a bound denoted by $ \left(\psi_{k,j}(s)\right)_n$. For example, if $i=j=2,3,4,5$, then 
\begin{align*}
	\left|  \left( \bz_{k,j}(s) * \left(  \pi^\infty_j h_j \right) \right)_n \right| &= \left| \sum_{\ell \in \Z }  ( \bz_{k,j}(s) )_{|n-\ell|} \left(  \pi^\infty_j h_j \right)_\ell \right| 
	\leq  \sum_{\ell \in \Z }  \left| ( \bz_{k,j}(s) )_{|n-\ell|} \left(  \pi^\infty_j h_j \right)_\ell \right| \\
	& =  \sum_{ \substack{   |n- \ell| \leq  d_k N_k  \\  N_k < |\ell|}  }  \left| ( \bz_{k,j}(s) )_{|n-\ell|} \left(  \pi^\infty_j h_j \right)_\ell \right| 
	 =  \sum_{ \substack{   |n- \ell| \leq  d_k N_k  \\  N_k < |\ell|}  }  \left| ( \bz_{k,j}(s) )_{|n-\ell|} \left(  \pi^\infty_j h_j \right)_\ell \right| \frac{\nu_\gamma^{|\ell|}}{\nu_\gamma^{|\ell|}}  \\
	&\leq \max_{\substack{   |n- \ell| \leq  d_k N_k  \\  N_k < |\ell|} } \frac{\left| ( \bz_{k,j}(s) )_{|n-\ell|} \right|}{\nu_\gamma^{|\ell|}}  \sum_{   N_k < |\ell|  }  \left|  \left(  \pi^\infty_j h_j \right)_\ell \right| \nu_\gamma^{|\ell|} \\
	&=  \max_{\substack{   |n- \ell| \leq  d_k N_k  \\  N_k < |\ell|} } \frac{\left| ( \bz_{k,j}(s) )_{|n-\ell|} \right|}{\nu_\gamma^{|\ell|}}  \| \pi^\infty_j h_j \|_{X_j}  
	\leq \max_{\substack{   |n- \ell| \leq  d_k N_k  \\  N_k < |\ell|} } \frac{\left| ( \bz_{k,j}(s) )_{|n-\ell|} \right|}{\nu_\gamma^{|\ell|}} 
	\bydef \psi_{k,j} (s).
\end{align*}
We can compute such bound for other values of $(k,j) \notin \mathcal{J}$ in a similar way. Let $|A_{i,k}(s)|$ represents the absolute value element-wise of the operator $A_{i,k}(s)$, then 
\begin{equation} \label{eq:ZN_kj_not_in_J}
\left\| \pi^{N_i}  A_{i,k}(s) \pi^{N_k} D_jF_k( \bx(s))\pi^\infty_j h_j \right\|_{X_i } \leq \beta_{k,j} \left\| \pi^{N_i} \left|  A_{i,k}(s)  \right| \pi^{N_k} \psi_{k,j}(s) \right\|_{X_i } \bydef  Z_{i,k,j}^N
\end{equation}
and such computation is finite and done as the other before. The cases where $(k,j) \in \mathcal{J}$ are similar but by their structure, the operator action cannot be represented by a discrete convolution because of the equation corresponding to the initial conditions. For $k = 15,16,17,18$ and $j = k+6$, 
\begin{align} 
\left\| \pi^{N_i}  A_{i,k}(s) \pi^{N_k}  D_{k+6}F_k (\bx(s))\pi^\infty_{k+6} h_{k+6} \right\|_{X_{i} }  &\leq \left\| \pi^{N_i}  A_{i,k}(s) \pi^{N_k} \right\|_{B(X_k,X_i)} \left|2\sum_{ n > N_a } (h_{k+6})_n T_n(-1)\right| \\
&\leq \frac{\left\| \pi^{N_i}  A_{i,k}(s) \pi^{N_k} \right\|_{B(X_k,X_i)}}{\nu^{N_a+1}}
 \bydef  Z_{i,k,k+6}^N.
\label{eq:ZN_kj_in_J_1}
\end{align}
For $j = k+11$, we can verify that $\displaystyle 0.95 \approx \sup_{ s \in [-1,1] }  \bx_{18}^2(s) + \bx_{19}^2 (s) \bydef \bar{\sigma}^2 < 1 $, and hence 
\begin{align*}
	\left| \pi^{N_k}  D_jF_k (\bx(s))\pi^\infty_{k+11} h_{k+11} \right| & = \left|  \sum_{(n,m) \in \N^2} (\pi^\infty_{k+11}h_{k+11})_{n,m} \left( \bx_{18}(s)  + \mi \bx_{19}(s)   \right)^n \left( \bx_{18}(s)  - \mi \bx_{19}(s)   \right)^m \right| \\
	&\le  \sum_{ (n,m) \in \N^2} |(\pi^\infty_{k+11}h_{k+11})_{n,m}| \left( \sqrt{ \bx_{18}^2(s) + \bx_{19}^2 (s) }  \right)^{n + m} \\
	&\le \bar{\sigma}^{N_p+1} \sum_{ (n,m) \in \N^2} |(\pi^\infty_{k+11}h_{k+11})_{n,m}| 
	\le \left( \frac{\bar{\sigma}}{\nu_a} \right) ^{N_p+1}.
\end{align*}
Thus,
\begin{align} 
	\left\| \pi^{N_i}  A_{i,k}(s) \pi^{N_k}  D_{k+11}F_k (\bx(s))\pi^\infty_{k+11} h_{k+11} \right\|_{X_{i} } & \leq  \left\| \pi^{N_i}  A_{i,k}(s) \pi^{N_k} \right\|_{B(X_k,X_i)}\left( \frac{\bar{\sigma}}{\nu_a} \right) ^{N_p+1} \\
	& \bydef Z_{i,k,k+11}^N. \label{eq:ZN_kj_in_J_2}
\end{align}
Finally, we consider the case $k = 23,24,25,26 $. When, $ j = k-12$, most elements are trivial except
\begin{align*}
	\left( \pi^{N_k}  D_{k-12}F_k (\bx(s))\pi^\infty_{k-12} h_{k-12} \right)_0 &= 
	\sum_{(n,m) \in \Z \times \N} (\pi^{\infty}_{k-12} h_{k-12})_{n,m}  (\bx_{16}(s) + \sign(n) \mi \bx_{17}(s))^{|n|}   \sigma_0^{m}.
\end{align*}
We define the bound over $1 \approx \sup\limits_{s\in[-1,1]}  (\bx_{16}^2(s) +  \bx_{17}^2(s))^n \leq \bar{\theta}^2$ and also, making sure that $\bar{\theta} < \nu_w$, we can bound, setting
$\epsilon^{(u,\gamma)} \bydef \min ( ( \frac{\bar{\theta}}{\nu_w} )^{N_{w_F}}, ( \frac{|\sigma_0|}{\nu_w} )^{N_{w_T}} )$.

\begin{align*}
		\left| \pi^{N_k}  D_{k-12}F_k (\bx(s))\pi^\infty_{k-12} h_{k-12} \right|_0 & \leq  \sum_{(n,m) \in \Z \times \N} |(\pi^{\infty}_{k-12} h_{k-12})_{n,m}|  |\sigma_0|^{m} \bar{\theta}^n \\
		&\leq \epsilon^{(u,\gamma)} \| \pi^\infty_{k-12} h_{k-12} \|_{X_{k-12}} \leq  \epsilon^{(u,\gamma)}.
\end{align*}
Let $e_k^\ell \in X_k $ defined element-wise by
\begin{align*}
	\left(e_k^\ell \right)_n =  \left\{
	\begin{tabular}{cl}
		$1$ & if $ n = \ell $ \\
		$0$ & otherwise.
	\end{tabular}
	\right.
\end{align*}
Then, we can bound
\begin{equation} \label{eq:ZN_kj_in_J_3}	
\left\| \pi^{N_i}  A_{i,k}(s) \pi^{N_k}  D_{k-12}F_k (\bx(s))\pi^\infty_{k-12} h_{k-12} \right\|_{X_{i} } \leq  
Z_{i,k,k-12}^N \bydef \left\| \pi^{N_i}  A_{i,k}(s) \pi^{N_k} e_k^0 \right\|_{X_i} \epsilon^{(u,\gamma)}.
\end{equation}
Similarly, for $ j = k$, we can compute
\begin{equation} \label{eq:ZN_kj_in_J_4}
	\left\| \pi^{N_i}  A_{i,k}(s) \pi^{N_k}  D_{k}F_k (\bx(s))\pi^\infty_{k} h_{k} \right\|_{X_{k} } 
	\le Z_{i,k,k}^N \bydef \frac{	\left\| \pi^{N_i}  A_{i,k}(s) \pi^{N_k} e_k^0 \right\|_{X_i}}{\nu_a^{N_a+1}}.
\end{equation}
Recall the bounds $Z^{\infty}_{i,j}$ for $(i,j) \notin \mathcal{I}$ in \eqref{eq:Zinfty_ij_not_in_I} and for $(i,j) \in \mathcal{I}$ in \eqref{eq:Zinfty_ij_in_I_1} and \eqref{eq:Zinfty_ij_in_I_2}, and recall the bounds $Z_{i,k,j}^{N}$ for $(k,j) \notin \mathcal{J}$ in \eqref{eq:ZN_kj_not_in_J} and for $(k,j) \in \mathcal{J}$ in \eqref{eq:ZN_kj_in_J_1}, \eqref{eq:ZN_kj_in_J_2}, \eqref{eq:ZN_kj_in_J_3} and \eqref{eq:ZN_kj_in_J_4}, and let
\begin{align*}
	Z^{i,j}_{1} \bydef Z_{i,j}^{\infty} +  \sum_{k = 0}^{29} Z_{i,k,j}^{N}.
\end{align*}
Using this representation, we can establish the following bound
\begin{equation} \label{perron}
	\left\| A(s)[ DF (\bx(s)) - A^\dagger(s) ] \right\|_{B(X)} \leq Z_1 \bydef \max_{i = 1,\dots,29} \left( \mu_i \sum_{j = 1}^{29} \frac{Z^{i,j}_{1}}{\mu_j} \right).
\end{equation}
Among all the bounds in Theorem~\ref{thm:radii_polynomial}, the bound $Z_1$ is the most challenging to compute. For the theorem to be successfully applied, it is necessary that $Z_1 < 1$. Due to the structure of the problem, certain components $Z^{i,j}_1$ are easier to estimate than others. We now proceed to explain the definition of the weights $\mu_i$.  In practice, one could simply set $\mu_i = 1$ for all $i$ and improve the bound by increasing the number of coefficients used in the Chebyshev expansion. However, this approach significantly increases both computational time and memory usage. An alternative approach involves selecting optimal weights $\mu_i$ using the Perron-Frobenius theorem. While this method can minimize the bound $Z_1$, it will worsen the other bounds in the radii polynomial $p(r)$ defined in \eqref{radii_poly}. In general, if $Y_0$ is already close to machine precision, this trade-off is acceptable. However, in our setting, where we limit the number of Chebyshev coefficients due to memory constraints, this degradation can become significant, especially on more complex segments, ultimately causing the proof to fail. To address this issue, we enforce a lower bound on $Z_1$ and reformulate \eqref{perron} as a constrained optimization problem. This problem can be solved using the nonlinear programming solver JuMP.jl \cite{Lubin2023}. Hence, once a set of weights $\mu_j$ is determined, we can rigorously compute with interval arithmetic the bound $Z_1$ satisfying \eqref{perron}.

\subsubsection{The bound \boldmath$Z_2$\unboldmath}

In this final section, we provide an explicit computational strategy for estimating the bound $Z_2$ appearing in \eqref{bound_Z_2}, which involves bounding the operator norm $\| A(s) [ DF(c) - DF( \bx(s)) ] \|_{B(X)}$ for an arbitrary $c \in  \overline{ B_{r}(\bx(s))}$. To simplify the notation, we define the operator 
\begin{align*}
	B(s) = \begin{pmatrix}
		B_{1,1}(s) & \cdots & B_{1,29}(s) \\
		\vdots & \ddots & \vdots \\
		B_{29,1}(s) & \cdots & B_{29,29}(s) \\
	\end{pmatrix} \quad \mbox{ with } \quad B_{i,j}(s) \bydef D_jF_i(c) -  D_jF_i(\bx(s)).
\end{align*}
Then
$\| A(s)B(s) \|_{B(X)} \leq  \max_{i = 1,\dots,29} \mu_i \sum_{j = 0}^{29}  \frac{1}{\mu_j} \| \left( A(s)B(s)\right)_{i,j} \|_{B(X_j,X_i)}$, 
with
\[
\| \left( A(s)B(s) \right)_{i,j} \|_{B(X_j,X_i)} = \sup_{\|h_j\|_{X_j} \leq 1 } \| \left( A(s)B(s) \right)_{i,j} h_j \|_{X_i}
	 \leq \sup_{\|h_j\|_{X_j} \leq 1 } \sum_{k= 0}^{29} \| A_{i,k}(s)B_{k,j}(s)  h_j \|_{X_i}.
\]
The analysis of each term $\| A_{i,k}(s)B_{k,j}(s)  h_j \|_{X_i}$ follows a similar structure across most indices, with a few notable exceptions that require a more detailed examination. We denote the set of such exceptional index by $ \mathcal{Q}$ define as
\[
	\mathcal{Q} = \left\{ \left. (i,k,j) \in \N^3  ~ \right| ~ i \in \{11,12,13,14,26,27,28,29\} \mbox{ and } i = k = j~ \right\}.
\]
These indices correspond to the cases when the operator $ D_j F_k (c)   -  D_j F_k (\bx(s) )$ is unbounded and we need the tail of $ \Lambda_{k}^{-1}(s)$ to balance it. First, we consider the case $ (i,k,j) \notin \mathcal{Q}$, we bound
 \begin{align*}
 	 \|  A_{i,k}(s)B_{k,j}(s)  h_j \|_{X_i} \leq \| A_{i,k}(s) \|_{B(X_j, X_k) } \| B_{k,j}(s)  h_j \|_{X_k},
 \end{align*}
 where the operator $A(s)$ is bounded by
 \begin{align*}
 	\| A_{i,k}(s) \|_{B(X_k, X_i) } \leq \max\left( \left\| \pi^{N_i} A_{i,k}(s) \pi^{N_k} \right\|_{B(X_k, X_i) }  , \delta_{i,k} \overline{ \Lambda}_{i}^{-1} \right),
 \end{align*}
 and this computation is finite and can be done rigorously.  As we already discuss in the computation of the $Z_1$ bound, our map $F$ has a polynomial structure component-wise. We will use that properties to bound the term $\| B_{k,j}(s)  h_j \|_{X_k}$. As we did with the bound $Z_1$, the nonzero element of the result of the operator $B_{k,j}(s)$ acting on $h_j$ can written as the result of the convolution $\by_{k,j}(s) *h_j$ and thus the norm can be bounded by
 \begin{align*}
 	  \| B_{k,j}(s)  h_j  \|_{X_k} \leq  \|\by_{k,j}(s) *h_j  \|_{X_k}  \leq  \|\by_{k,j}(s) \|_{X_k} \| h_j  \|_{X_j}\leq  \|\by_{k,j}(s) \|_{X_k}  .
 \end{align*}
  There exists a $t \in X$ with $  \| t \|_{X} \leq 1$ satisfying $c = \bx(s) + r t$, thus for any $\ell \in \N$, and using the binomial theorem, we can bound
  \begin{align*}
  \left\|  c_i^\ell - (x_i(s))^\ell \right\|_{X_i} & = \left\|  (x_i(s) + rt_i)^\ell - (x_i(s))^\ell \right\|_{X_i} 
  = \left\| \sum_{n = 0}^\ell \binom{\ell}{n} (x_i(s))^{n}  (rt_i)^{\ell-n} - (x_i(s))^\ell \right\|_{X_i} \\
  &= \left\| \sum_{n = 0}^{\ell-1} \binom{\ell}{n} (x_i(s))^{n}  (rt_i)^{\ell-n}  \right\|_{X_i} 
  \leq \sum_{n = 0}^{\ell-1} \binom{\ell}{n} \left\|x_i(s) \right\|^{n}  \left\|rt_i\right\|^{\ell-n}_{X_i} \\
  & \leq \sum_{n = 0}^{\ell-1} \binom{\ell}{n} \left\|x_i(s) \right\|^{n}  \left( \frac{r}{\mu_i}\right)^{\ell - n} 
   = \sum_{n = 0}^{\ell} \binom{\ell}{n} \left\|x_i(s) \right\|^{n}  \left( \frac{r}{\mu_i}\right)^{\ell - n} - \| x_i(s) \|_{X_i}^\ell \\
  & = \left(  \| x_i(s) \|_{X_i}+ \frac{r}{\mu_i} \right)^\ell - \| x_i(s) \|_{X_i}^\ell,
  \end{align*}
	which is a polynimal of $r$ that is at least linear if not trivial. Using this result, we can rigorously compute a bound over $\| \by_{k,j}(s) \|_{X_k}$. For $(i,k,j) \in \mathcal{Q}$, we split the equation by 
\begin{align*}
	\| A_{i,i}(s) B_{i , i}(s) \|_{B( X_i) }  &= \left \| \pi^{N_i} A_{i,i}(s) \pi^{N_i} B_{i , i}(s) + \pi^{\infty}_i A_{i,i}(s) \pi^{\infty}_i B_{i , i}(s) \right\|_{B( X_i) } \\
	&\leq \left\| \pi^{N_i} A_{i,i}(s) \pi^{N_i} B_{i , i}(s) \right\|_{B( X_i) } + \left\| \pi^{\infty}_i A_{i,i}(s) \pi^{\infty}_i B_{i , i}(s) \right\|_{B( X_i) } \\
	&\leq \left\| \pi^{N_i} A_{i,i}(s) \pi^{N_i} \right\|_{B(X_i) } \left\|   \pi^{N_i} B_{i , i}(s) \right\|_{B( X_i) } + \left\| \pi^{\infty}_i A_{i,i}(s) \pi^{\infty}_i B_{i , i}(s) \right\|_{B(X_i) } .
\end{align*}
The first term is bounded similarly to the case $(i,k,j) \notin \mathcal{Q}$, by bounding the finite nonzero element of $ \pi^{N_i} B_{i , i}(s) h_i $. For the second term, when $ i = 11,12,13,14 $, we have 
\begin{align*}
	\| \pi^{\infty}_i A_{i,i}(s) \pi^{\infty}_i B_{i , i}(s) \|_{B( X_i) } &= \sup_{\|h\|_{X_i} \leq 1} \| \pi^{\infty}_i A_{i,i}(s) \pi^{\infty}_i B_{i , i}(s) h_i \|_{ X_i } \\
	&= \sup_{\|h\|_{X_i} \leq 1}  \sum \frac{1}{| in + \bx_6(s)m |} \left| m (c_6 - \bx_6(s)) (h_{i})_{n,m} \right| \nu_v^{|n|+m} \\
	&\leq \sup_{\|h\|_{X_i} \leq 1}  \sum \frac{1}{| \bx_6(s) |} \frac{r}{\mu_6}  \left|  (h_{i})_n \right| \nu_v^{|n|} 
	 \leq \frac{r}{\mu_6} \left(  \sup_{s \in [-1,1]} \frac{1}{| \bx_6(s) |} \right) .
\end{align*}
Similarly, when $ i = 26,27,28,29 $, we get
\begin{align*}
	\left\| \pi^{\infty}_i A_{i,i}(s) \pi^{\infty}_i B_{i , i}(s) \right\|_{B( X_i) } &\leq \frac{r}{\mu_{25}} \max_{\ell = 1, 6} \left( \sup_{s \in [-1,1 ]}  \frac{1}{|\bx_{25,\ell}(s) |}\right).
\end{align*}
Thus, we have all the components needed to compute the bound $Z_2(r)$.

\section{Applications} \label{sec:application}
In this section, we present several applications of the framework developed in this work. We first address key numerical challenges and outline the strategies implemented to improve computational performance. We then investigate the cubic Swift–Hohenberg equation, proving the existence of both snaking and isolas. Finally, we apply our approach to the Gray–Scott equation, where we rigorously verify the presence of snaking behavior. All computations in this study were conducted on Windows 11, AMD Ryzen 9 7950X, 64GB DDR5 DRAM 6400MT/s using Julia 1.11.1 with RadiiPolynomial.jl v0.8.24 \cite{RadiiPolynomial}. The code used to generate the results presented in the following sections is freely
available from \cite{Duchesne}.

\subsection{Optimizing computational time and memory usage}\label{sec:opti}
Given that the majority of our computations rely on interval arithmetic, it is crucial to optimize these operations to mitigate computational overhead. Let $\bx(s) \in X$ denote a vector-valued function whose components are represented by Chebyshev polynomials of degree $N_s$, as described in Section~\ref{sec:data_representation}. 
For a fixed $s \in [-1,1]$, let $\mathcal{N} \in \N$ such that the projection $\pi_{{N}} \bx(s)$ lies in $\mathbb{K}^{\mathcal{N}}$. Each component in this projection requires $N_s + 1$ data points to store the coefficients of its finite Chebyshev expansion. Consequently, the full representation of the approximate solution $\bx(s)$ can be stored as a matrix of size $\mathcal{N} \times (N_s + 1)$. Converting this data grid into the polynomial representation $\bx(s)$ requires applying a discrete Fourier transform (DFT), which has a computational complexity of $\mathcal{O}(N_s^2)$. To improve efficiency, we aim to use the Fast Fourier Transform (FFT) instead. 
To apply the FFT algorithm with the RadiiPlynomial.jl package, the data grid must be symmetric and its length must be a power of two (see the documentation of \cite{RadiiPolynomial} for more details). To construct such a data grid, we define:
\begin{align*}
	N_{\fft} \bydef 2^{ \lceil \log_2 (2N_s+1) \rceil } \quad \mbox{ and } \quad n_{\fft} \bydef N_{\fft} / 2 +1 .
\end{align*}
Here, $N_{\fft}$ is the total size of the grid used for the FFT to recover the Chebyshev coefficients, and $n_{\fft}$ is the number of unique data points required to construct the symmetric grid. Given $n_{\fft}$ values ${ \bx^1, \ldots, \bx^{n_{\fft}} }$, we define the full FFT input grid as:
\begin{align*}
	\bx_{\fft} =  \{ \bx^{n_{\fft}} , \bx^{n_{\fft}-1 }, \dots,\bx_2,  \bx_1 , \bx_2,  \dots, \bx^{n_{\fft}-1}  \}.
\end{align*}
This symmetric grid of $N_{\fft}$ points is suited for the FFT algorithm, which yields the Chebyshev coefficients with a computational complexity of $\mathcal{O}(N_{\fft} \log N_{\fft})$, significantly improving over the $\mathcal{O}(N_s^2)$ cost of a direct DFT. While this approach increases memory usage, the gain in computational speed makes it advantageous. 

In practice, computing the entire solution branch within a single segment is numerically challenging. Therefore, we will partition the branch into multiple segments. To ensure that each segments are connected and that the branch is periodic we need to impose some restriction on the endpoints. Suppose that a branch is approximated by $K$ segments $\{ \bx^{(k)} (s)  \}_{k = 1}^K$ with approximate $K$ tangent vector segments $\{ \eta_s^{(k)}(s)\}$, each parameterized over $s \in [-1,1]$, then we impose
\begin{equation} \label{joining_condition}
	\bx^{(k)} (1)=
	\begin{cases}
		\bx^{(1)}(-1), & \text{for } k = K,\\
		\bx^{(k+1)}(-1), & \text{otherwise}
	\end{cases}
	\quad \text{and} \quad
		\eta_s^{(k)} (1)=
	\begin{cases}
		\eta_s^{(1)}(-1), & \text{for } k = K,\\
		\eta_s^{(k+1)}(-1), & \text{otherwise}
	\end{cases}
\end{equation}
This condition guarantees that each segment starts exactly where the previous one ends, and, recalling \eqref{eq:ZFP}, that the resulting maps $\{ F^{(k)}(x,s) \}_{k=1}^K$, with first component $Q^{(k)}_s(x,s) \bydef \left( x -\bx^{(k)}(s) \right) \cdot \eta^{(k)}_s(s)$ also match at the segment boundaries. Consequently, the zero-finding problem solved at the end of one segment is consistent with that used to initialize the next. From that construction, we obtain automatically the smoothness at the intersection of the segment (e.g. see \cite{Breden2015}). 
\begin{figure}[H]
	\center
	\includegraphics[scale = 0.25]{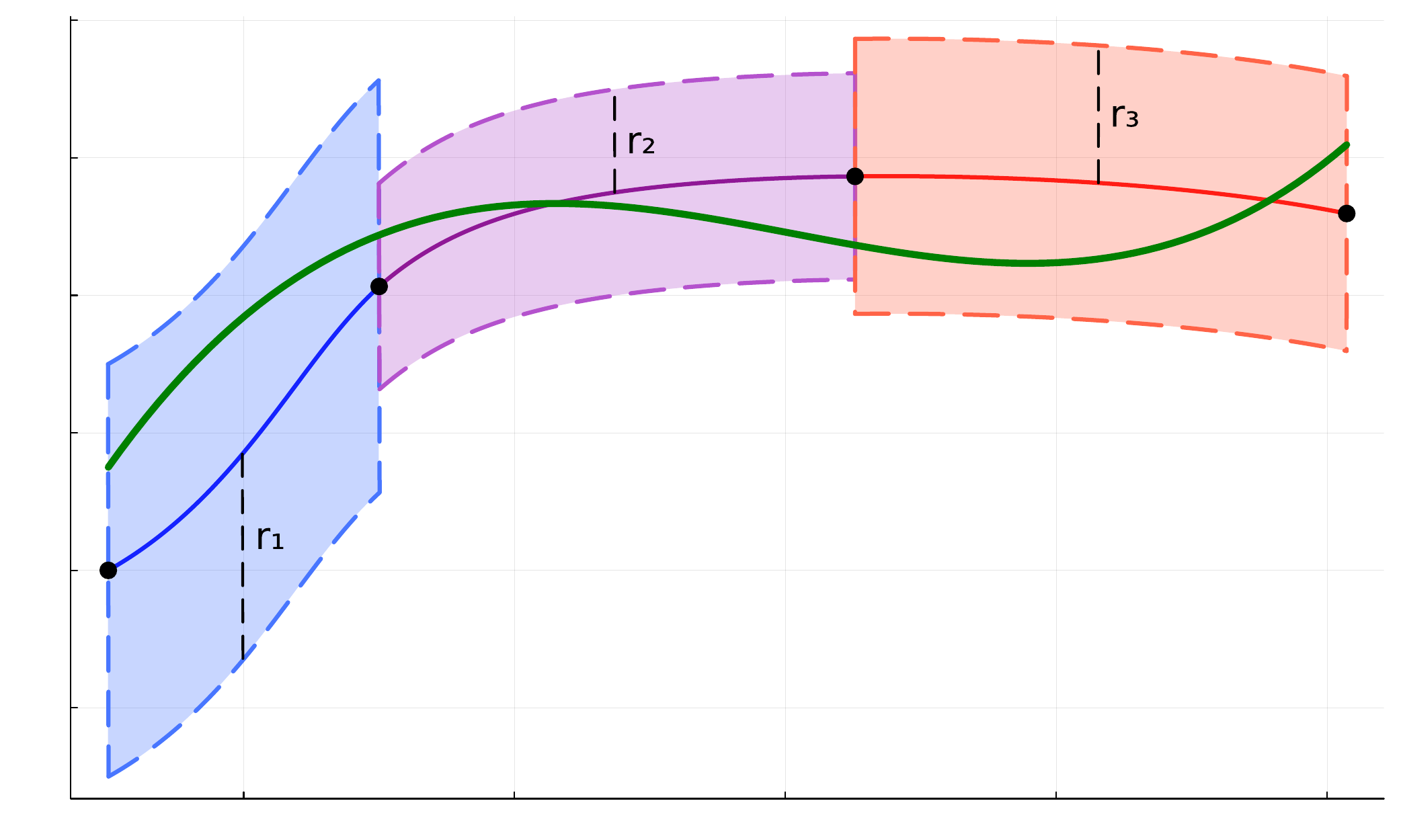} 
	\vspace{-.25cm}
	\caption{{\small When successful, the radii polynomial approach yields a validated neighborhood, a {\em tube} of radius $r_i > 0$, centered around the numerical approximation of each segment (illustrated as the shaded region), within which the true solution curve (shown in green) is guaranteed to exist.}} \label{application_plot_rad}
\end{figure}
Before turning to the specific applications to the Swift–Hohenberg and Gray–Scott equations, we first consolidate the general framework and key assumptions underlying our approach into a formal result, which provides a unifying statement that encapsulates the core analytical and computational components used throughout this work.
\begin{lem}\label{lem:final}
For $s \in [-1,1]$, let $\left\{ F^{(k)}(x, s) \right\}_{k = 1}^K$, where $F^{(k)} : X \times [-1,1] \rightarrow X$ is of the form \eqref{eq:ZFP}, constructed from the system \eqref{Second Order System} and whose first component is defined using a list of parameterized curves $\bx^{(k)}(s) \in X$ and tangent vectors $\eta_s^{(k)}(s) \in X$ satisfying the connecting condition \eqref{joining_condition}. Assume moreover that $\eta_s^{(k)}(s)$ satisfy the structural condition of Corollary~\ref{cor_cond_smooth} and that for each $k$, there exists a linear operator $(A^\dagger)^{(k)}(s) \in B(X)$, an injective operator $A^{(k)}(s) \in B(X)$, and bounds $Y_0^{(k)}$, $Z_0^{(k)}$, $Z_1^{(k)}$, and $Z_2^{(k)}$ such that the hypotheses of the Newton-Kantorovich Theorem~\ref{thm:radii_polynomial} are satisfied. In particular, assume that the associated radii polynomial \eqref{radii_poly} is negative at some radius $r_0^{(k)} > 0$, and that the smoothness condition \eqref{cond_smooth} holds for every segment. Finally assume that Hypothesis~\ref{hyp:hyperbolic} is satisfied. Let $ \tilde{\theta}:[0,1] \rightarrow \R$ denote the unique lifting of the phase. If $\tilde{\theta}(1) \neq \tilde{\theta}(0)$, then for large $L$ there are localized patterned states whose $(\alpha,L)$ bifurcation diagram is described by two interlaced snaking curves. If $\tilde{\theta}(1) = \tilde{\theta}(0)$, then for large $L$ there are localized patterned states whose $(\alpha,L)$ bifurcation diagram is described by an infinite stack of isolas.
\end{lem}
\begin{proof}
By the Newton-Kantorovich Theorem, for each map $F^{(k)}(x(s), s)$, there exists a unique solution $\tilde{x}^{(k)}(s) \in \overline{B_r(\bx^{(k)}(s))}$ satisfying $F^{(k)}\left(\tilde{x}^{(k)}(s), s\right) = 0$ for all $s \in [-1,1]$. Since both the sequence of maps and the sequence of numerical approximations satisfy the connecting condition \eqref{joining_condition}, the full solution curve $\cup_{k = 1}^K ( \cup_{s\in[-1,1]} \tilde{x}^{(k)}(s) )$ forms a smooth closed loop. Along this loop, the parameter $\alpha$ is identified with the component $(x^{(k)}(s))_1$ and therefore remains within a closed interval. 
By construction, for each point on the solution curve, there exists a patterned front and an associated phase, as defined in Definition~\ref{dfn:patternedfront}, which are subsets of the solution at that point. Let $\mathcal{S}(\ell)$ and $\tilde{\theta}(\ell)$, with $\ell \in [0,1]$, denote parameterizations of the patterned front and the lifted phase, respectively, along the loop. The existence of the solution loop established by the Newton-Kantorovich Theorem guarantees that $\mathcal{S}(\ell)$ is a regular patterned front for all $\ell \in [0,1]$ (see Remark~\ref{rem:transversality}). 
Finally, by Corollary~\ref{cor_cond_smooth}, we have $\mathcal{S}_\ell(\ell) \neq 0$ for all $\ell \in [0,1]$, ensuring that Hypothesis~\ref{hyp:regular_patterned_front} is satisfied. Since by assumption, Hypothesis~\ref{hyp:hyperbolic} is satisfied, all the hypotheses of the Forcing Theorem~\ref{thm:forcing_theorem} are met, which yields the desired result.
\end{proof}

\subsection{Snaking and stacked isolas in the Swift-Hohenberg equation}
The first example we consider is the cubic Swift–Hohenberg equation, given by
\begin{align*}
	U_t = - (  1 + \partial_x^2  )^2 U - \alpha U +  2 U^3 - U^3, \quad x \in \R.
\end{align*}
Stationary solutions of this equation can be reformulated as a second-order system by
\begin{equation} \label{Second Order System_SH}
	\left\{
	\begin{aligned}
		u'' &= v \\ 
		v'' &=- (\alpha + 1 ) u  + 2u^2 - u^3 -2 v,
	\end{aligned}
	\right. 
\end{equation}
which clearly fits the structure of equation \eqref{Second Order System}. Snaking and isolas in the Swift–Hohenberg equation have been studied in \cite{Beck2009, Aougab2017}, where the authors leverage a conserved quantity to reduce the problem to a three-dimensional zero level set. Building on this foundational work, the framework developed here establishes global existence using the same forcing theorem, but extends the approach by removing the reliance on a conserved quantity and accommodating both orientable and non-orientable cases in a unified manner.
\subsubsection{Snaking in Swift-Hohenberg}
We begin by considering the case where the unstable manifold associated with the periodic orbit is orientable. 
The first step in the process is to compute a numerical candidate that serves as a starting point for the pseudo-arclength continuation. While various numerical methods may be employed, the structure of the radii polynomial approach makes Newton’s method particularly effective for approximating a zero of the map \eqref{eq:zero_finding_f} at a fixed parameter $\alpha$. The main challenge lies in selecting a suitable initial guess to ensure the convergence of Newton's method. In this work, we begin by computing the local unstable manifold of the steady state at the origin, and then apply a shooting method via numerical integration from the boundary of this manifold to obtain a numerical approximation of an orbit connecting to a periodic solution. This yields an approximate heteroclinic orbit that serves as the initial guess for Newton’s method (see Figure~\ref{application_SH_snake_hetero}). 

\begin{figure}[H]
	\center
	\includegraphics[scale = 0.31]{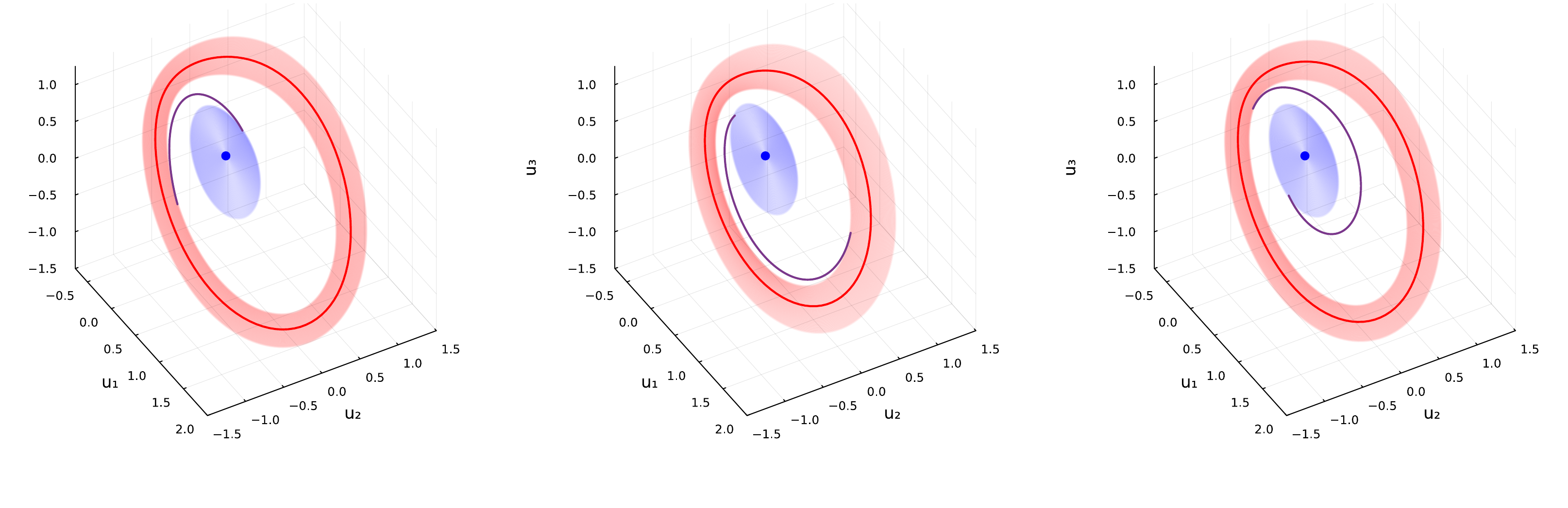} 
	\vspace{-.8cm}
	\caption{{\small Three-dimensional projection of the heteroclinic connection for the Swift–Hohenberg equation \eqref{Second Order System_SH} for $\alpha=0.35$ (left), $\alpha=0.4933$ (center) and $\alpha=0.414$ (right). The figure displays the periodic orbit (solid red) along with its associated local unstable manifold (shaded red), the origin (solid blue) with its associated local stable manifold (cyan), and the heteroclinic orbit connecting the two manifolds (purple).}} \label{application_SH_snake_hetero}
\end{figure}

Once the numerical approximation $\bx(s)$ is defined, we apply the results from Section~\ref{sec:bounds} to compute the bounds required for the Newton-Kantorovich Theorem using interval arithmetic. This allows us to determine a validated radius of existence $r_0$ for each segment. The outcomes of these computations are summarized in Table~\ref{table_SH_snake}. Each segment required approximately 50 minutes to validate, using the weights $\nu_\gamma = \nu_v = \nu_w = \nu_p = 1.4$ and $\nu_a = 1.1$. Each segment approximation $\bx(s)$ are Chebyshev polynomial of order 15. It is worth noting that the validation of individual segments can be performed in parallel, significantly reducing total computational time.
\begin{table}[H]
	\center
	{\scriptsize
	\begin{tabular}{|c|c|c|c|c|c|c|} 
		\hline
		\rowcolor[HTML]{C0C0C0} 
		\# & Arclength & $r_0$ & $Y_0$ & $Z_0$ & $Z_1$ & $Z_2(r)$ \\ \hline
		1 & $ 1.5036 $ & $ 1.2975 \times 10^{ -6 } $ & $ 4.722 \times 10^{ -7 } $ & $ 1.2968 \times 10^{ -4 } $ & $0.6255$ & $8043 + \mathcal{O}(r)$ \\ \hline
		2 & $ 2.5082 $ & $ 1.6996 \times 10^{ -6 } $ & $ 4.1577 \times 10^{ -7 } $ & $ 9.9584 \times 10^{ -2 } $ & $0.6503$ & $3238 + \mathcal{O}(r)$ \\ \hline
		3 & $ 2.5071 $ & $ 1.0954 \times 10^{ -7 } $ & $ 4.1783 \times 10^{ -8 } $ & $ 5.1894 \times 10^{ -2 } $ & $0.5664$ & $1966 + \mathcal{O}(r)$ \\ \hline
		4 & $ 2.0065 $ & $ 4.3896 \times 10^{ -7 } $ & $ 2.0217 \times 10^{ -7 } $ & $ 1.174 \times 10^{ -2 } $ & $0.5267$ & $2229 + \mathcal{O}(r)$ \\ \hline
		5 & $ 2.5086 $ & $ 1.8653 \times 10^{ -6 } $ & $ 8.1174 \times 10^{ -7 } $ & $ 3.903 \times 10^{ -3 } $ & $0.55$ & $5835 + \mathcal{O}(r)$ \\ \hline
		6 & $ 1.6264 $ & $ 2.2221 \times 10^{ -6 } $ & $ 9.2927 \times 10^{ -7 } $ & $ 8.1295 \times 10^{ -6 } $ & $0.5632$ & $8374 + \mathcal{O}(r)$ \\ \hline
	\end{tabular} 
	}
\caption{{\small Data for the computer-assisted proof of snaking in the Swift-Hohenberg equation \eqref{Second Order System_SH}.}} \label{table_SH_snake}
\end{table} 
Our results show that the continuation branch of the heteroclinic orbit is periodic and confined to a closed interval of the parameter $\alpha$. To confirm the presence of snaking behavior, we rigorously verify that condition~\eqref{cond_smooth} holds and that the lifted phase $\tilde{\theta}$, defined implicitly from $\theta_1(s) + \mi \theta_2(s) = e^{\mi \theta(s)}$, is not homotopic to a constant function. These verifications are performed using interval arithmetic.  We refer to Figure~\ref{application_SH_snake} for a graph of the parameter $\alpha(s)$ and $\theta(s)$.
\begin{figure}[H]
\center
\includegraphics[scale = 0.18]{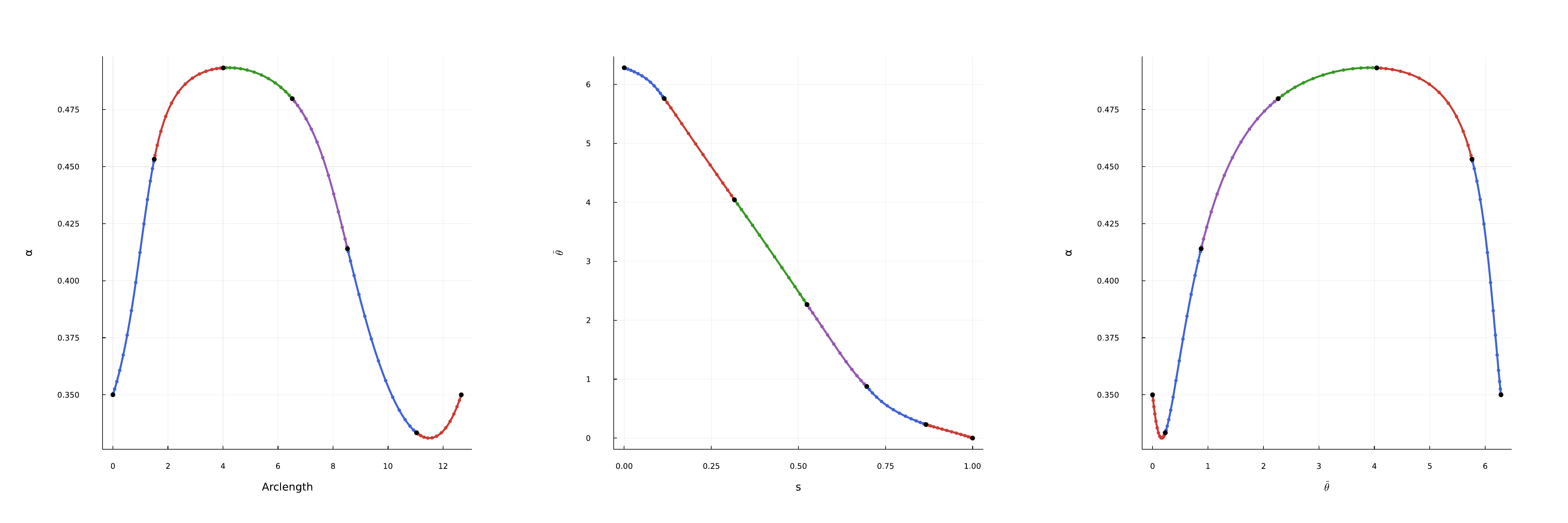} 
\vspace{-.5cm}
\caption{{\small Graphs of the parameter $\alpha$ (left), plotted with respect to its arclength parameterization, and of the angle $\theta$ (center), displayed along the continuation branch. The right panel shows a continuation diagram of the heteroclinic orbit, illustrating the variation of $\alpha$ with respect to $\theta$, as sketched in Figure~\ref{fig:pattern_hetero}. Coloured dots represent the grid points used to approximate the solution $\bx(s)$, while black dots indicate the data points connecting adjacent segments. The parameter $\alpha$ remains bounded above and below and is periodic. On the other hand, the angle $\theta$ is not homotopic to a constant function.}} \label{application_SH_snake}
\end{figure}

We conclude this section by formalizing the result. 

\begin{thm}
For sufficiently large $L$, the cubic Swift–Hohenberg equation \eqref{Second Order System_SH} admits localized patterned states whose $(\alpha, L)$ bifurcation diagram consists of two interlaced snaking curves.
\end{thm}
\begin{proof}
	The hypotheses of Lemma~\ref{lem:final} are rigorously verified and the corresponding implementation is available in the accompanying code repository \cite{Duchesne}.
\end{proof}

\subsubsection{Stacked isolas in Swift-Hohenberg} \label{subsec:Isolas}
Next, we consider the case where the unstable manifold associated with the periodic orbit is non-orientable. 
The process of finding a numerical candidate follows the same general approach as in the orientable case. However, the key distinction is that the bundle is now $4\pi$-periodic, since after a $2\pi$ rotation, the solution returns to the same spatial position on the orbit but with the tangent vector reversed in direction. To address this, we scale the original periodic solution to be $\pi$-periodic. To solve for the coefficients and still be a solution of \eqref{eq:map_periodic_sol}, we extend the solution to the periodic orbit with non-minimal period $2\pi$. For example let $\tau = 2t $, then
\begin{align*}
	\gamma_1(\theta) = \sum_{n \in \N} (\gamma_1^*)_n \cos\left(\frac12 n\theta \right),
\end{align*}
correspond to the cosine Fourier expension of the first equation in \eqref{rescale first order system}, then we have
\begin{align*}
	\gamma_1(\theta) = \sum_{n \in \N} (\gamma_1^*)_n \cos\left(\frac12 n \theta \right) = \sum_{n \in \N} (\gamma_1)_n \cos(n\theta),
\end{align*}
where the new coefficients $(\gamma_1)_n$ are given by
\begin{align*}
	{(\gamma_1) }_n = 
	\begin{cases}
		{(\gamma_1^*) }_{n/2}, & \text{if } n \text{ even}, \\
		0, & \text{if } n \text{ odd}.
	\end{cases}
\end{align*}
As a result, the frequency $\omega$ solved for in this case corresponds to twice the frequency of the original periodic orbit patterned front. The main advantage of this reformulation is that, from this point onward, the continuation and validation procedures remain identical to those used in the orientable case. However, representing the unstable manifold in the non-orientable setting requires twice as many Fourier coefficients, significantly increasing both the computational cost and memory requirements of the proof. From this point, we apply the same technique as in the orientable case to compute a set of numerical candidates (see Figure~\ref{application_SH_isola_hetero}). 
\begin{figure}[H]
	\center
	\includegraphics[scale = 0.3]{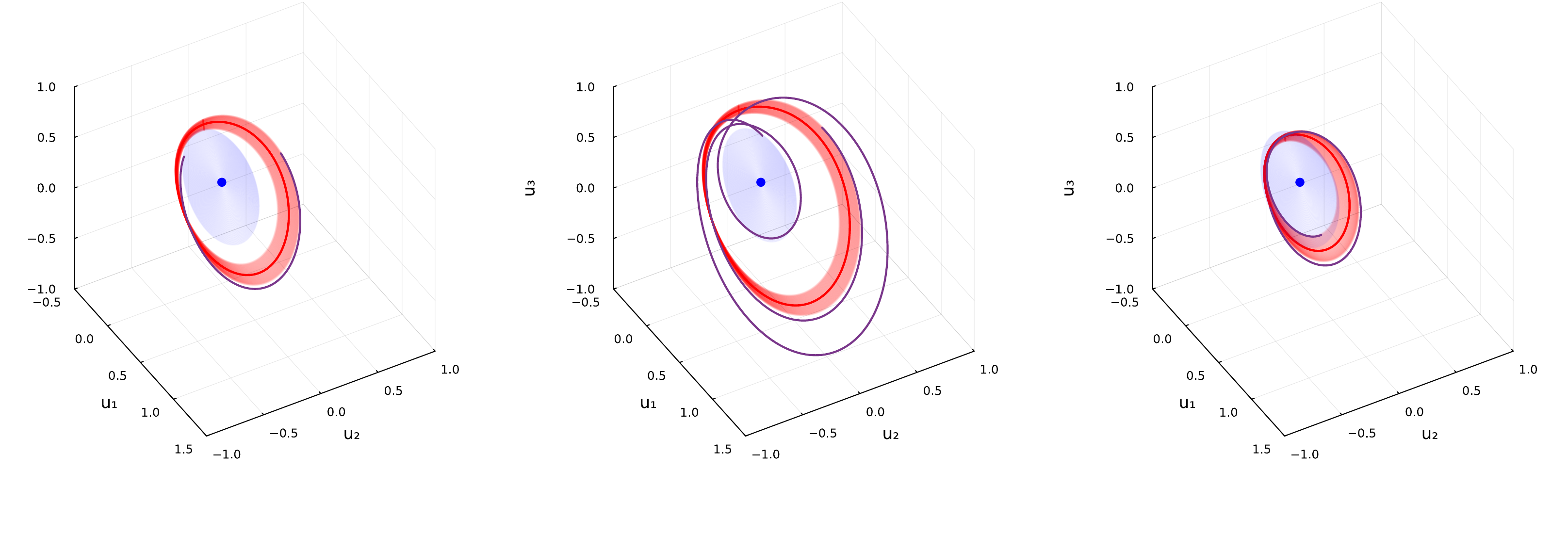} 
	\vspace{-.8cm}
	\caption{{\small Three-dimensional projection of the heteroclinic connection for the Swift–Hohenberg equation \eqref{Second Order System_SH} for $\alpha=0.35$ (left), $\alpha=0.4521$ (center) and $\alpha=0.2468$ (right). The figures display the periodic orbit (solid red) along with its local unstable manifold (shaded red), the origin (solid blue) with its associated local stable manifold (shaded blue), and the heteroclinic orbit connecting the two manifolds (purple).}} \label{application_SH_isola_hetero}
\end{figure}
Similarly as the orientable case, we use the numerical approximation $\bx(s)$ to compute the bounds required for the Newton-Kantorovich Theorem using interval arithmetic. The outcomes of these computations are summarized in Table \ref{table_SH_isolas}. Each segment required approximately 2 hours and 45 minutes to validate, using the weights $\nu_\gamma = \nu_v = \nu_w = \nu_p = 1.4 , 1.5$ and $\nu_a = 1.1$. Each segment approximation $\bx(s)$ are Chebyshev polynomial of order 15. The longer computation time is primarily due to the higher number of coefficients needed to accurately represent the local unstable manifold associated with the periodic orbit.  
\begin{table}[H]
	\center
	{\scriptsize
	\begin{tabular}{|c|c|c|c|c|c|c|}
		\hline
		\rowcolor[HTML]{C0C0C0}
		\# & Arclength & $r_0$ & $Y_0$ & $Z_0$ & $Z_1$ & $Z_2(r)$ \\ \hline
		1 & $ 3.0001 $ & $ 2.8181 \times 10^{ -7 } $ & $ 6.6669 \times 10^{ -8 } $ & $ 2.4679 \times 10^{ -3 } $ & $0.7585$ & $8867 + \mathcal{O}(r)$ \\ \hline
		2 & $ 3.0001 $ & $ 3.3303 \times 10^{ -7 } $ & $ 1.7063 \times 10^{ -7 } $ & $ 1.181 \times 10^{ -2 } $ & $0.4725$ & $9898 + \mathcal{O}(r)$ \\ \hline
		3 & $ 3.0002 $ & $ 4.7953 \times 10^{ -7 } $ & $ 2.6014 \times 10^{ -7 } $ & $ 8.5295 \times 10^{ -3 } $ & $0.4424$ & $13652 + \mathcal{O}(r)$ \\ \hline
		4 & $ 3.0001 $ & $ 7.1363 \times 10^{ -7 } $ & $ 3.9362 \times 10^{ -7 } $ & $ 1.1299 \times 10^{ -2 } $ & $0.4296$ & $10601 + \mathcal{O}(r)$ \\ \hline
		5 & $ 3.0011 $ & $ 1.7193 \times 10^{ -6 } $ & $ 7.8329 \times 10^{ -7 } $ & $ 4.9765 \times 10^{ -2 } $ & $0.451$ & $25408 + \mathcal{O}(r)$ \\ \hline
		6 & $ 1.0004 $ & $ 2.3094 \times 10^{ -6 } $ & $ 1.1951 \times 10^{ -6 } $ & $ 3.7236 \times 10^{ -6 } $ & $0.424$ & $25352 + \mathcal{O}(r)$ \\ \hline
		7 & $ 0.3 $ & $ 1.0022 \times 10^{ -6 } $ & $ 4.3797 \times 10^{ -7 } $ & $ 4.8786 \times 10^{ -6 } $ & $0.5503$ & $12684 + \mathcal{O}(r)$ \\ \hline
		8 & $ 0.2252 $ & $ 1.3924 \times 10^{ -6 } $ & $ 5.8933 \times 10^{ -7 } $ & $ 7.1522 \times 10^{ -6 } $ & $0.5504$ & $18913 + \mathcal{O}(r)$ \\ \hline
		9 & $ 0.1502 $ & $ 2.8318 \times 10^{ -6 } $ & $ 1.0552 \times 10^{ -6 } $ & $ 1.1459 \times 10^{ -5 } $ & $0.5424$ & $30012 + \mathcal{O}(r)$ \\ \hline
		10 & $ 0.5001 $ & $ 2.0231 \times 10^{ -6 } $ & $ 7.343 \times 10^{ -7 } $ & $ 1.0723 \times 10^{ -4 } $ & $0.5664$ & $34868 + \mathcal{O}(r)$ \\ \hline
		11 & $ 2.0004 $ & $ 2.532 \times 10^{ -6 } $ & $ 7.165 \times 10^{ -7 } $ & $ 5.4561 \times 10^{ -5 } $ & $0.641$ & $30001 + \mathcal{O}(r)$ \\ \hline
		12 & $ 3.0012 $ & $ 8.954 \times 10^{ -7 } $ & $ 2.2631 \times 10^{ -7 } $ & $ 2.2446 \times 10^{ -3 } $ & $0.7263$ & $20887 + \mathcal{O}(r)$ \\ \hline
		13 & $ 3.0011 $ & $ 8.6065 \times 10^{ -7 } $ & $ 3.3859 \times 10^{ -7 } $ & $ 1.7253 \times 10^{ -2 } $ & $0.582$ & $8545 + \mathcal{O}(r)$ \\ \hline
		14 & $ 3.0015 $ & $ 2.2096 \times 10^{ -6 } $ & $ 3.3042 \times 10^{ -7 } $ & $ 4.6992 \times 10^{ -2 } $ & $0.783$ & $9286 + \mathcal{O}(r)$ \\ \hline
		15 & $ 3.0018 $ & $ 2.7726 \times 10^{ -6 } $ & $ 9.7184 \times 10^{ -7 } $ & $ 2.7707 \times 10^{ -2 } $ & $0.5981$ & $8525 + \mathcal{O}(r)$ \\ \hline
		16 & $ 1.5 $ & $ 6.8488 \times 10^{ -7 } $ & $ 2.5244 \times 10^{ -7 } $ & $ 1.6994 \times 10^{ -5 } $ & $0.6259$ & $7950 + \mathcal{O}(r)$ \\ \hline
		17 & $ 0.5001 $ & $ 1.0399 \times 10^{ -6 } $ & $ 4.2293 \times 10^{ -7 } $ & $ 6.6276 \times 10^{ -7 } $ & $0.584$ & $8910 + \mathcal{O}(r)$ \\ \hline
		18 & $ 0.3007 $ & $ 2.9065 \times 10^{ -7 } $ & $ 1.3578 \times 10^{ -7 } $ & $ 4.0131 \times 10^{ -6 } $ & $0.5271$ & $19641 + \mathcal{O}(r)$ \\ \hline
		19 & $ 0.8265 $ & $ 2.8034 \times 10^{ -7 } $ & $ 1.2725 \times 10^{ -7 } $ & $ 1.1071 \times 10^{ -5 } $ & $0.5405$ & $19835 + \mathcal{O}(r)$ \\ \hline
		\end{tabular}
	}
	\caption{{\small Data for the computer-assisted proof of stacked isolas in the Swift-Hohenberg equation \eqref{Second Order System_SH}.}}\label{table_SH_isolas}
\end{table}

As shown in Table~\ref{table_SH_isolas}, we notice that some segments were proven with significantly shorter arclength compared to others. These segments correspond to folds in the continuation branch with respect to the parameter $\alpha$ (see the left panel of Figure~\ref{application_SH_isola}). Our results indicate that the continuation branch of the heteroclinic orbit is periodic and remains confined to a closed interval of the parameter $\alpha$. To confirm the presence of isolas in this setting, we a posteriori verify that condition~\eqref{cond_smooth} holds and that the lifted phase $\tilde{\theta}$ is homotopic to a constant function (see the right panel of Figure~\ref{application_SH_isola}) using interval arithmetic.
\begin{figure}[H]
	\center
	\includegraphics[scale = 0.18]{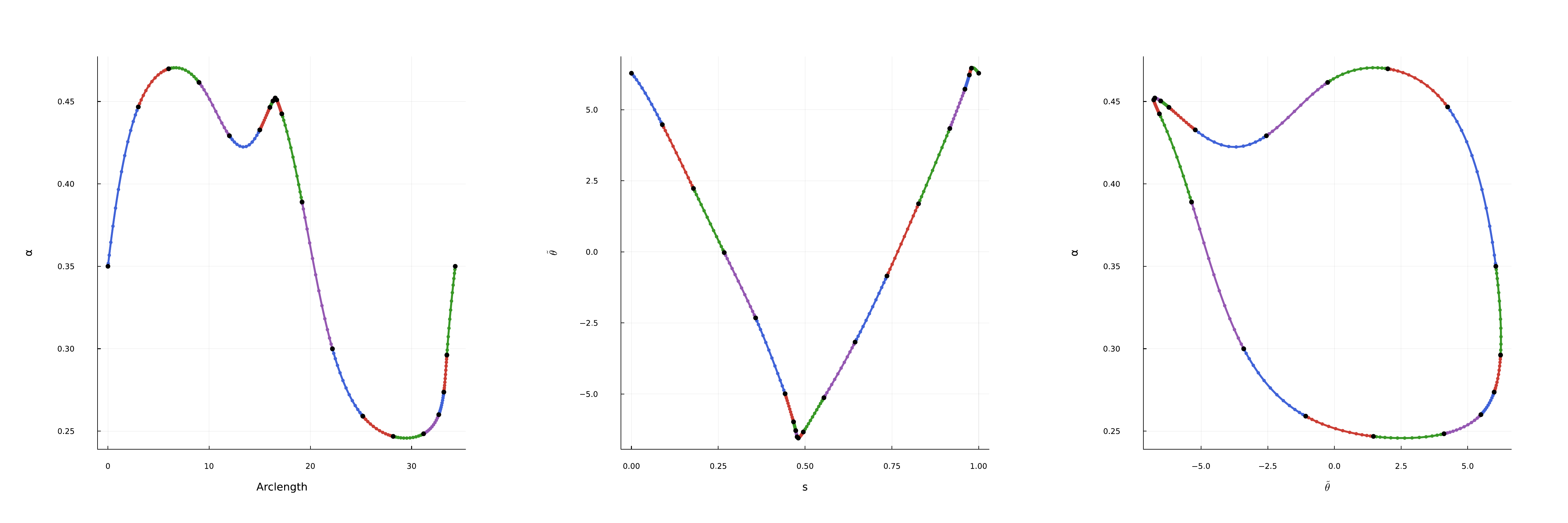} 
		\vspace{-.5cm}
\caption{{\small Graphs of the parameter $\alpha$ (left), plotted with respect to its arclength parameterization, and of the angle $\theta$ (center), displayed along the continuation branch. The right panel shows a continuation diagram of the heteroclinic orbit, illustrating the variation of $\alpha$ with respect to $\theta$, as sketched in Figure~\ref{fig:pattern_hetero}. Coloured dots represent the grid points used to approximate the solution $\bx(s)$, while black dots indicate the data points connecting adjacent segments. The parameter $\alpha$ remains bounded above and below and is periodic. On the other hand, the angle $\theta$ is homotopic to a constant function.}} \label{application_SH_isola}
\end{figure}

We conclude this section by formalizing the result.

	\begin{thm}
		For sufficiently large $L$, the cubic Swift–Hohenberg equation given in \eqref{Second Order System_SH} admits localized patterned states whose $(\alpha, L)$  bifurcation diagram consists of an infinite stack of isolas.
	\end{thm}
	\begin{proof}
		The hypotheses of Lemma~\ref{lem:final} are rigorously verified and the corresponding implementation is available in the accompanying code repository \cite{Duchesne}.
	\end{proof}

\subsection{Snaking in the Gray-Scott equation}
We consider the steady-state of the Gray-Scott equation given by 
\[
	\left\{
	\begin{aligned}
		9u'' &= uv^2 + u - \alpha \\ 
		v'' &=-  uv^2  + \frac12 v,
	\end{aligned}
	\right. 
\]%
for which snaking has been studied in \cite{Gandhi2018}. In this form, the equilibrium for which the snaking behavior occurs is not at the origin, hence, we translate and rescale the system and get
\begin{equation} \label{Second Order System_GS}
	\left\{
	\begin{aligned}
	u'' &=  \frac{1}{9} \left(  \frac12v^2 + \alpha \left( u+ v +uv^2 \right)+ \alpha^2 uv+ \alpha^3 u \right), \\
	v'' &= - \frac12v^2 - \alpha \left( \frac12v + uv^2 \right) -2 \alpha^2 uv - \alpha^3 u ,
	\end{aligned}
	\right. 
\end{equation}
which fits the structure of the general second order system \eqref{Second Order System}. In the remaining of this section, we proof the global snaking for the Gray-Scott equation \eqref{Second Order System_GS}.

We consider the case where the unstable manifold associated with the periodic orbit is orientable. 
As before, we start by finding a numerical approximation of the solution branch (see Figure~\ref{application_GS_snake_hetero}). 
\begin{figure}[H]
	\center
	\includegraphics[scale = 0.29]{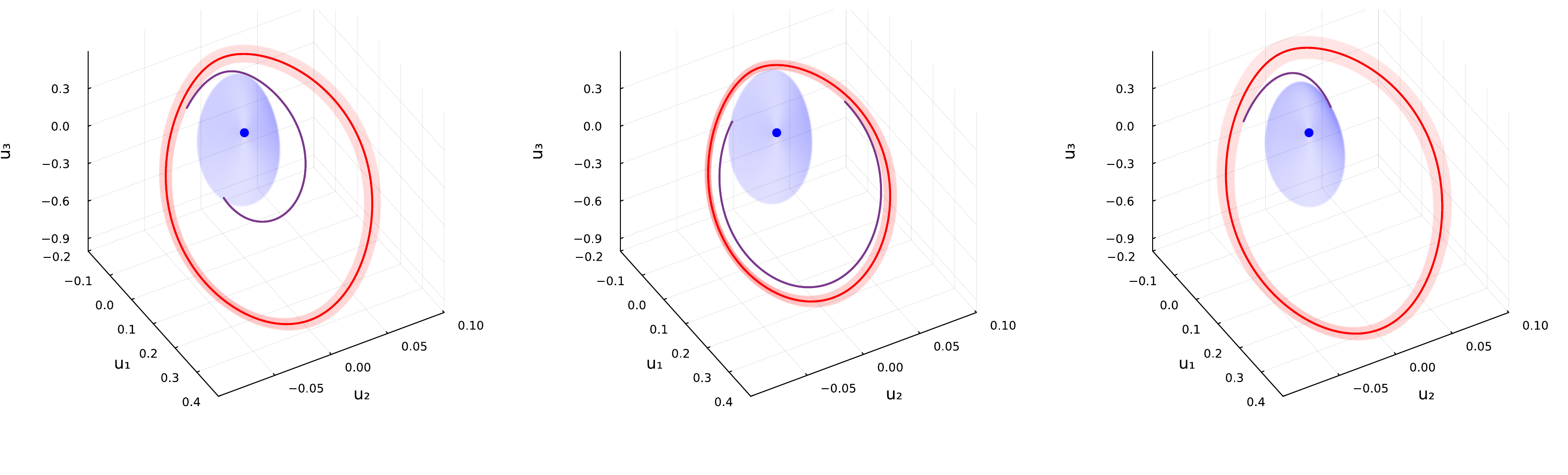} 
	\vspace{-.5cm}
	\caption{{\small Three-dimensional projection of the heteroclinic connection for the Gray-Scott equation \eqref{Second Order System_GS} for $\alpha = 1.1889$ (left), $\alpha = 1.2136$ (center) and $\alpha = 1.1666$ (right). The figure displays the periodic orbit (solid red) along with its associated local unstable manifold (shaded red), the origin (solid blue) with its associated local stable manifold (shaded blue), and the heteroclinic orbit connecting the two manifolds (purple).}} \label{application_GS_snake_hetero}
\end{figure}
Then we use the numerical approximation $\bx(s)$ to compute the bounds required for the Newton-Kantorovich Theorem using interval arithmetic. The outcomes of these computations are summarized in Table~\ref{table_GS}. Each segment required approximately 2 hours to validate, using the weights $\nu_\gamma = \nu_v = \nu_w = \nu_p = 1.4$ and $\nu_a = 1.1$. Each segment approximation $\bx(s)$ are Chebyshev polynomial of order 15. 
\begin{table}[H]
	\center
	{\footnotesize
	\begin{tabular}{|c|c|c|c|c|c|c|}
		\hline
		\rowcolor[HTML]{C0C0C0}
		\# & Arclength & $r_0$ & $Y_0$ & $Z_0$ & $Z_1$ & $Z_2(r)$ \\ \hline
		1 & $ 1.0001 $ & $ 5.8609 \times 10^{ -7 } $ & $ 2.5081 \times 10^{ -7 } $ & $ 2.9712 \times 10^{ -6 } $ & $0.5678$ & $7311 + \mathcal{O}(r)$ \\ \hline
		2 & $ 0.4001 $ & $ 3.2988 \times 10^{ -7 } $ & $ 1.4116 \times 10^{ -7 } $ & $ 8.181 \times 10^{ -6 } $ & $0.5699$ & $6696 + \mathcal{O}(r)$ \\ \hline
		3 & $ 0.2006 $ & $ 5.428 \times 10^{ -7 } $ & $ 2.2949 \times 10^{ -7 } $ & $ 2.7708 \times 10^{ -5 } $ & $0.5732$ & $7321 + \mathcal{O}(r)$ \\ \hline
		4 & $ 0.4008 $ & $ 1.205 \times 10^{ -6 } $ & $ 5.0483 \times 10^{ -7 } $ & $ 1.4978 \times 10^{ -4 } $ & $0.572$ & $7367 + \mathcal{O}(r)$ \\ \hline
		5 & $ 1.4005 $ & $ 1.0503 \times 10^{ -6 } $ & $ 4.622 \times 10^{ -7 } $ & $ 1.5908 \times 10^{ -3 } $ & $0.5495$ & $8421 + \mathcal{O}(r)$ \\ \hline
		6 & $ 1.4002 $ & $ 6.7864 \times 10^{ -8 } $ & $ 2.8663 \times 10^{ -8 } $ & $ 4.8946 \times 10^{ -4 } $ & $0.5766$ & $7931 + \mathcal{O}(r)$ \\ \hline
		7 & $ 1.4002 $ & $ 1.4967 \times 10^{ -6 } $ & $ 5.2672 \times 10^{ -7 } $ & $ 2.2249 \times 10^{ -3 } $ & $0.6205$ & $16951 + \mathcal{O}(r)$ \\ \hline
		8 & $ 2.0006 $ & $ 5.6217 \times 10^{ -7 } $ & $ 1.6054 \times 10^{ -7 } $ & $ 1.9883 \times 10^{ -2 } $ & $0.6889$ & $10084 + \mathcal{O}(r)$ \\ \hline
		9 & $ 2.5014 $ & $ 6.1839 \times 10^{ -7 } $ & $ 9.9782 \times 10^{ -8 } $ & $ 1.5408 \times 10^{ -1 } $ & $0.6783$ & $10182 + \mathcal{O}(r)$ \\ \hline
		10 & $ 2.001 $ & $ 2.3013 \times 10^{ -7 } $ & $ 8.7113 \times 10^{ -8 } $ & $ 2.2801 \times 10^{ -2 } $ & $0.5968$ & $8234 + \mathcal{O}(r)$ \\ \hline
		11 & $ 1.0008 $ & $ 3.9899 \times 10^{ -6 } $ & $ 1.4843 \times 10^{ -6 } $ & $ 1.4151 \times 10^{ -4 } $ & $0.5944$ & $8379 + \mathcal{O}(r)$ \\ \hline
		12 & $ 0.4014 $ & $ 3.958 \times 10^{ -6 } $ & $ 1.5093 \times 10^{ -6 } $ & $ 5.3848 \times 10^{ -5 } $ & $0.5845$ & $8625 + \mathcal{O}(r)$ \\ \hline
		13 & $ 0.4004 $ & $ 2.9759 \times 10^{ -6 } $ & $ 1.1811 \times 10^{ -6 } $ & $ 1.7438 \times 10^{ -6 } $ & $0.577$ & $8775 + \mathcal{O}(r)$ \\ \hline
		14 & $ 0.4 $ & $ 6.0221 \times 10^{ -6 } $ & $ 2.2573 \times 10^{ -6 } $ & $ 2.7087 \times 10^{ -7 } $ & $0.573$ & $8659 + \mathcal{O}(r)$ \\ \hline
		15 & $ 2.0008 $ & $ 7.9219 \times 10^{ -6 } $ & $ 2.4257 \times 10^{ -6 } $ & $ 2.3783 \times 10^{ -3 } $ & $0.5533$ & $17433 + \mathcal{O}(r)$ \\ \hline
		16 & $ 2.0005 $ & $ 2.4714 \times 10^{ -6 } $ & $ 9.824 \times 10^{ -7 } $ & $ 1.1342 \times 10^{ -3 } $ & $0.5697$ & $12820 + \mathcal{O}(r)$ \\ \hline
		17 & $ 2.0004 $ & $ 1.3561 \times 10^{ -6 } $ & $ 5.5229 \times 10^{ -7 } $ & $ 7.436 \times 10^{ -4 } $ & $0.5676$ & $18008 + \mathcal{O}(r)$ \\ \hline
		18 & $ 2.0004 $ & $ 3.9705 \times 10^{ -7 } $ & $ 1.6837 \times 10^{ -7 } $ & $ 5.9075 \times 10^{ -4 } $ & $0.5676$ & $19638 + \mathcal{O}(r)$ \\ \hline
		19 & $ 1.5208 $ & $ 2.7269 \times 10^{ -7 } $ & $ 1.1833 \times 10^{ -7 } $ & $ 3.0672 \times 10^{ -5 } $ & $0.5613$ & $17567 + \mathcal{O}(r)$ \\ \hline
	\end{tabular}
	\caption{{\small Data for the computer-assisted proof of snaking in the Gray-Scott equation \eqref{Second Order System_GS}.}} \label{table_GS}
	}
\end{table}
 Our results indicate that the continuation branch of the heteroclinic orbit is periodic and remains confined to a closed interval of the parameter $\alpha$. To confirm the presence of snaking behavior in this setting, we a posteriori verify that condition~\eqref{cond_smooth} holds and that the angle $\tilde{\theta}$ is not homotopic to a constant function (see the right panel of Figure~\ref{application_SH_isola}) using interval arithmetic.
\begin{figure}[H]
	\center
	\includegraphics[scale = 0.18]{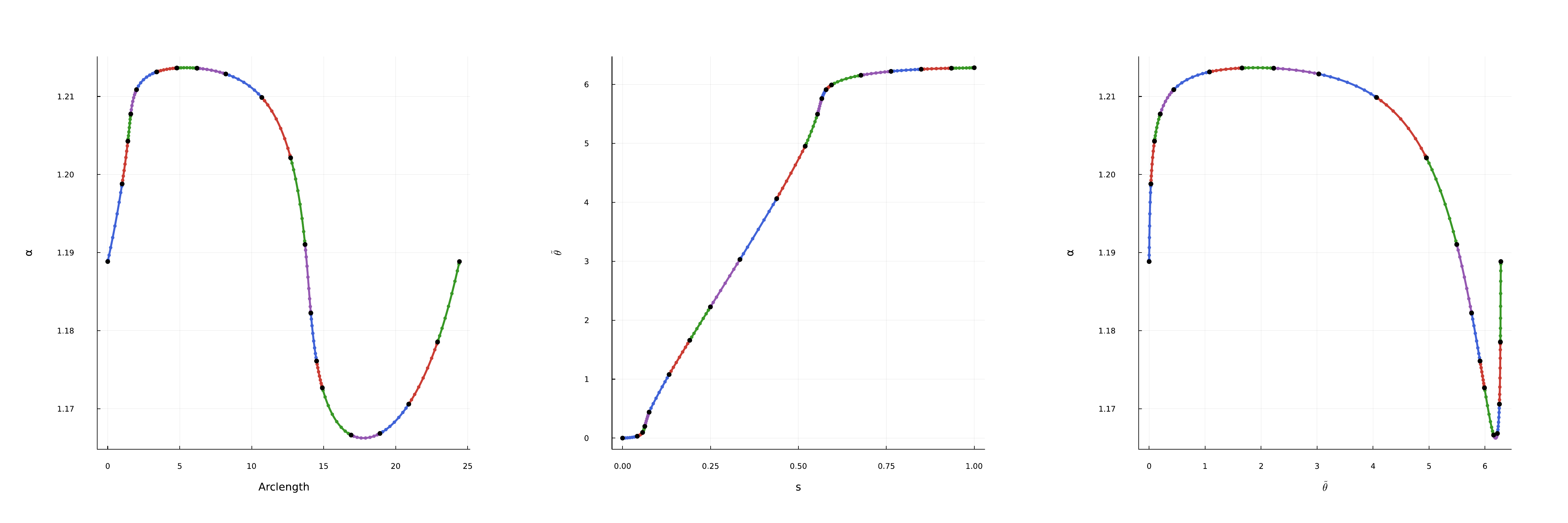} 
	\vspace{-.5cm}
	\caption{{\small Graphs of the parameter $\alpha$ (left), plotted with respect to its arclength parameterization, and of the angle $\theta$ (center), displayed along the continuation branch. The right panel shows a continuation diagram of the heteroclinic orbit, illustrating the variation of $\alpha$ with respect to $\theta$, as sketched in Figure~\ref{fig:pattern_hetero}. Coloured dots represent the grid points used to approximate the solution $\bx(s)$, while black dots indicate the data points connecting adjacent segments. The parameter $\alpha$ remains bounded above and below and is periodic. On the other hand, the angle $\theta$ is not homotopic to a constant function.}} \label{application_GS_snake}
\end{figure}

Our result is summarized in the following Theorem.

\begin{thm}
	For sufficiently large $L$, the Gray-Scott equation given in \eqref{Second Order System_GS} admits localized patterned states whose $(\alpha, L)$ bifurcation diagram consists of two interlaced snaking curves.
\end{thm}
\begin{proof}
	The hypotheses of Lemma~\ref{lem:final} are rigorously verified and the corresponding implementation is available in the accompanying code repository \cite{Duchesne}.
\end{proof}

\appendix
\section{Floating point error}\label{appA:flaoting_point}
We begin by recalling key results about floating-point arithmetic. In the IEEE 754 standard (Float64), for any $x,y \in \mathbb{R}$, we have
	$fl\left( x \pm y \right) = (x \pm y)(1 + \delta_1)$ and 
	$fl\left( x \cdot y \right) = (x \cdot y)(1 + \delta_2)$,
with $|\delta_1|,|\delta_2| \leq \epsilon \approx 2.22 \times 10^{-16}$. This can be extended to summations. For $x_1,\dots,x_n \in \mathbb{R}$, 
\begin{align*}
	fl \left( \sum_{k = 1}^n  x_k  \right) &= fl\left[ fl \left( x_1 + x_2 \right)  +  \sum_{k = 3}^n  x_k  \right] 
	= fl \left[ \left( x_1  +  x_2 \right)(1 + \delta_1)  +  \sum_{k = 3}^n  x_k  \right] \\
	&= fl \left[ \left( \left( x_1  +  x_2 \right)(1 + \delta_1)  +  x_3 \right)(1 + \delta_2)   +  \sum_{k = 4}^n  x_k  \right] \\
	& = \cdots = x_1 \prod_{\ell = 1 }^{n-1} (1 + \delta_\ell) +  \sum_{k = 2}^n  x_k \prod_{\ell = k-1}^{n-1} (1 + \delta_\ell)  
\end{align*}
We now recall similar results for complex arithmetic. For $ x,y\in \mathbb{C}$ with $x = x_1 + i x_2 , y = y_1 + iy_2$, 
\begin{align*}
	fl\left( x \pm y \right) &= fl\left( (x_1 + \mi x_2) \pm  (y_1 + \mi y_2) \right) \\
	&= fl\left( x_1  \pm  y_1 \right) + \mi fl\left(  x_2 \pm    y_2 \right) \\
	&= \left( x_1  \pm  y_1 \right)(1 + \delta_1) + \mi \left(  x_2 \pm    y_2 \right)(1 + \delta_2)
\end{align*}
with $|\delta_1|,|\delta_2| \leq \epsilon $. For multiplication, we obtain
\begin{align*}
	fl\left(xy \right) =& ~fl\left( (x_1 + \mi x_2)( y_1 + \mi y_2) \right) \\
	= &~fl \left( fl(x_1 y_1) - fl(x_2y_2) \right) + \mi fl\left( fl(x_2 y_1) + fl(x_1y_2) \right) \\
	= & ~  \left( (x_1 y_1)(1 + \delta_1 ) - (x_2y_2)(1 + \delta_2) \right) (1 + \delta_5) + \mi \left( (x_2 y_1)(1 + \delta_3) + (x_1y_2)(1 + \delta_4) \right)(1 + \delta_6) \\
	= & ~ xy + (x_1 y_1)( \delta_1  + \delta_5 + \delta_1 \delta_5) - (x_2y_2)( \delta_2 + \delta_5 + \delta_2\delta_5) \\
	&+ \mi  (x_2 y_1)( \delta_3 +  \delta_6 +  \delta_3  \delta_6) + i(x_1y_2)( \delta_4 + \delta_6 + \delta_4 \delta_6) ,
\end{align*}
with $|\delta_i| \leq \epsilon  $ for $\mi \in \{ 1,2,3,4,5,6 \}$. The resulting error can be bounded by
\begin{align*}
	|fl\left( xy \right)  - xy |= & ~ | (x_1 y_1)( \delta_1  + \delta_5 + \delta_1 \delta_5) - (x_2y_2)( \delta_2 + \delta_5 + \delta_2\delta_5) \\
	&+ \mi  (x_2 y_1)( \delta_3 +  \delta_6 +  \delta_3  \delta_6) + \mi (x_1y_2)( \delta_4 + \delta_6 + \delta_4 \delta_6) | \\
	\leq & \left(| x_1 y_1| + |x_2y_2| + |x_2 y_1| + | x_1y_2| \right)( 2\epsilon  + \epsilon^2).
\end{align*}
\begin{prop}
	Let $A \in \C^{m \times n}$, $B \in \C^{n \times p}$ be matrices with complex interval entries, written as
	\begin{align*}
		A = \bar{A} \pm R_1 \pm \mi R_2, \quad B = \bar{B} \pm Q_1 \pm \mi Q_2,
	\end{align*}
	  where $\bar{A}$, $\bar{B}$ are the midpoints,  and $R_1$,$R_2$,$Q_1$ and $Q_2$ are real matrices representing the radii of the intervals. Assume all operations are carried out in Float64 under the IEEE 754 standard. Let 
	\begin{align*}
		\epsilon &= 2.22045 \times 10^{-16 }, \quad
		r = \max\left( \| R_1	\| ,\| R_2	\|  \right), \quad
		\rho = \max\left(  \| Q_1	\| ,\| Q_2	\|  \right) ,\\
		M_\epsilon &= (1 + 2 \epsilon + \epsilon^2)(1+\epsilon)^{n-1} - 1, \quad 
		\bar{A} = 	\bar{A}^\Re + \mi 	\bar{A}^\Im, \quad
		\bar{B} = \bar{B}^\Re + \mi \bar{B}^\Im, \\
		C(\bar{A},\bar{B}) &= \left( \| \bar{A}^\Re \| + \| \bar{A}^\Im \| \right) \left( \| \bar{B}^\Re \| + \| \bar{B}^\Im \| \right) .
	\end{align*}
Then, $\| I - AB \| \leq \left\| I - fl\left( \bar{A}\bar{B} \right)\right\| + M_\epsilon C(A,B)  + 2 \rho \| \bar{A} \| + 2 r \| \bar{B} \| + 4 r \rho$.
\end{prop}
\begin{proof} We start by bounding
	\begin{align*}
		\| I - AB \| &= \left\| I - \left( \bar{A} \pm R_1 \pm \mi R_2 \right) \left( \bar{B} \pm Q_1 \pm \mi Q_2 \right)  \right\| \\
		&= \left\| I -  \bar{A}\bar{B} - \bar{A}\left( \pm Q_1 \pm \mi Q_2 \right) - \left( \pm R_1 \pm \mi R_2 \right) \bar{B} - \left( \pm R_1 \pm \mi R_2 \right) \left( \pm Q_1 \pm \mi Q_2 \right) \right\| \\
		&\leq \left\| I -  \bar{A}\bar{B} \right\| + 2 \rho \| \bar{A} \| + 2 r \| \bar{B} \| + 4 r \rho,
	\end{align*}
	which give us the last 3 terms of the right-hand side of the wanted result. Then, we expend
	\begin{align*}
		\left\| I -  \bar{A}\bar{B} \right\|	&= \left\| I -  fl(\bar{A}\bar{B}) + fl(\bar{A}\bar{B}) -  \bar{A}\bar{B} \right\| \leq \left\| I -  fl(\bar{A}\bar{B}) \right\|  + \left\| fl(\bar{A}\bar{B}) -  \bar{A}\bar{B}  \right\|.
	\end{align*}
	To complete the proof, we need to show that $ \left\| fl(\bar{A}\bar{B}) -  \bar{A}\bar{B}  \right\| \leq 4M_\epsilon C(\bar{A},\bar{B}) $. 
	Component-wise we have $	[ \bar{A} \bar{B}]_{ij} = \sum_{k = 1}^n   \bar{A}_{ik} \bar{B}_{kj} $ for $1 \leq i \leq m $ and $i \leq j \leq p$. Let $\bar{A}_{ik} \bydef \bar{A}^\Re_{ik} + \mi\bar{A}^\Im_{ik}$ and $\bar{B}_{kj} \bydef \bar{B}^\Re_{kj} + \mi\bar{B}^\Im_{kj}$ and using the previous result about floating point product for complex number, 
	\begin{align*}
		fl \left( \bar{A}_{ik} \bar{B}_{kj}  \right)  =&  ~ (\bar{A}^\Re_{ik} \bar{B}^\Re_{kj})(1+ \delta_1^k  + \delta_5^k + \delta_1^k \delta_5^k) - (\bar{A}^\Im_{ik}\bar{B}^\Im_{kj})(1+ \delta_2^k + \delta_5^k + \delta_2^k\delta_5^k) \\
		&+ \mi  (\bar{A}^\Im_{ik} \bar{B}^\Re_{kj})(1+ \delta_3^k +  \delta_6^k +  \delta_3^k  \delta_6^k) + \mi(\bar{A}^\Re_{ik}\bar{B}^\Im_{kj})( 1+\delta_4^k + \delta_6^k + \delta_4^k \delta_6^k)
	\end{align*}
	with $|\delta_\ell^k | \leq \epsilon  $ for $\ell \in \{ 1,2,3,4,5,6 \}$. Using the result about floating point summation on the first term give us 
	\begin{align*}
		fl \left( \sum_{k = 1}^n (\bar{A}^\Re_{ik} \bar{B}^\Re_{kj})(1+ \delta_1^k + \delta_5^k + \delta_1^k \delta_5^k) \right) = & ~ \bar{A}^\Re_{i1} \bar{B}^\Re_{1j} (1+ \delta_1^k + \delta_5^k + \delta_1^k \delta_5^k) \prod_{\ell = 1 }^{n-1} (1 + \delta^\Re_\ell) \\
		& +   \sum_{k = 2}^n \bar{A}^\Re_{ik} \bar{B}^\Re_{kj} (1+ p\delta_1^k + \delta_5^k + \delta_1^k \delta_5^k) \prod_{\ell = k-1}^{n-1} (1 + \delta^\Re_\ell)
	\end{align*}
	with $| \delta_\ell^\Re | < \epsilon $ for $\ell \in \{ 1,2,\dots,n-1\}$. We can bound the error of this term by
	\begin{align*}
		&\left| fl \left( \sum_{k = 1}^n (\bar{A}^\Re_{ik} \bar{B}^\Re_{kj})(1+ \delta_1^k + \delta_5^k + \delta_1^k \delta_5^k) \right) - \sum_{k = 1}^n \bar{A}^\Re_{ik} \bar{B}^\Re_{kj} \right| \\
		& =\left| \bar{A}^\Re_{i1} \bar{B}^\Re_{1j} (1+ \delta_1^k + \delta_5^k + \delta_1^k \delta_5^k) \prod_{\ell = 1 }^{n-1} (1 + \delta^\Re_\ell) \right. \\ 
		&\quad + \left.  \sum_{k = 2}^n \bar{A}^\Re_{ik} \bar{B}^\Re_{kj} (1+ p\delta_1^k + \delta_5^k + \delta_1^k \delta_5^k) \prod_{\ell = k-1}^{n-1} (1 + \delta^\Re_\ell) - \sum_{k = 1}^n \bar{A}^\Re_{ik} \bar{B}^\Re_{kj} \right| \\
		& =  \left| \bar{A}^\Re_{i1} \bar{B}^\Re_{1j}  \left( \prod_{\ell = 1 }^{n-1} (1 + \delta^\Re_\ell) -1 \right) + \bar{A}^\Re_{i1} \bar{B}^\Re_{1j} (1+ p\delta_1^k + \delta_5^k + \delta_1^k \delta_5^k)  \prod_{\ell = 1 }^{n-1} (1 + \delta^\Re_\ell) \right. \\ 
		&\quad +   \left. \sum_{k = 2}^n  \bar{A}^\Re_{ik} \bar{B}^\Re_{1k}  \left( \prod_{\ell = k-1}^{n-1} (1 + \delta^\Re_\ell) -1 \right) + \bar{A}^\Re_{ik} \bar{B}^\Re_{kj}  (1+ p\delta_1^k + \delta_5^k + \delta_1^k \delta_5^k)  \prod_{\ell = k-1}^{n-1} (1 + \delta^\Re_\ell) \right|\\
		& \leq  \left| \bar{A}^\Re_{i1} \bar{B}^\Re_{1j} \right| \left( (1+ \epsilon)^{n-1} -1 \right) + \left|\bar{A}^\Re_{i1} \bar{B}^\Re_{1j} \right| (2\epsilon + \epsilon^2)  (1 + \epsilon)^{n-1} \\ 
		&\quad +   \sum_{k = 2}^n \left| \bar{A}^\Re_{ik} \bar{B}^\Re_{1k} \right| \left( (1+ \epsilon)^{n-1} -1 \right) + \left|\bar{A}^\Re_{ik} \bar{B}^\Re_{kj} \right| (2\epsilon + \epsilon^2)  (1 + \epsilon)^{n-1} \\
		& = \sum_{k = 1}^n \left| \bar{A}^\Re_{ik} \bar{B}^\Re_{1k} \right| \left( (1+ \epsilon)^{n-1} -1 \right) + \left|\bar{A}^\Re_{ik} \bar{B}^\Re_{kj} \right| (2\epsilon + \epsilon^2)  (1 + \epsilon)^{n-1} \\
		& = [(1 + 2 \epsilon + \epsilon^2)(1+\epsilon)^{n-1} - 1]\sum_{k = 1}^n \left| \bar{A}^\Re_{ik} \bar{B}^\Re_{kj} \right| 
		 = M_\epsilon \sum_{k = 1}^n \left| \bar{A}^\Re_{ik} \bar{B}^\Re_{kj} \right|
		 = M_\epsilon  [|\bar{A}^\Re| |\bar{B}^\Re|]_{i,j} 
	\end{align*}
Using the same idea for the 3 other terms, we get
\[
\left|	fl \left( [ \bar{A} \bar{B} ] _{ij}  \right) - [ \bar{A} \bar{B} ] _{ij} \right|  \leq M_\epsilon \left( [|\bar{A}^\Re| |\bar{B}^\Re|]_{i,j} + [|\bar{A}^\Im| |\bar{B}^\Im|]_{i,j}  +[|\bar{A}^\Re| |\bar{B}^\Im|]_{i,j}  +[|\bar{A}^\Im| |\bar{B}^\Re|]_{i,j}    \right)
\]
Thus,
\[
		\| fl[\bar{A} \bar{B}]  - \bar{A} \bar{B}\|
		 \leq M_\epsilon \left( \| \bar{A}^\Re \| \| \bar{B}^\Im  \|  +\| \bar{A}^\Im  \| \| \bar{B}^\Re \|+\| \bar{A}^\Im  \| \| \bar{B}^\Re  \|+\| \bar{A}^\Re  \| \| \bar{B}^\Im  \| \right)  \leq  M_\epsilon C(A,B). \qedhere
\]
\end{proof}

	\bibliographystyle{unsrt}
	\bibliography{Bibliography}

\begin{thebibliography}{10}

\bibitem{MR2665448}
Daniele Avitabile, David J.~B. Lloyd, John Burke, Edgar Knobloch, and Bj\"orn
  Sandstede.
\newblock To snake or not to snake in the planar {S}wift-{H}ohenberg equation.
\newblock {\em SIAM J. Appl. Dyn. Syst.}, 9(3):704--733, 2010.

\bibitem{Burke2006}
John Burke and Edgar Knobloch.
\newblock Localized states in the generalized {S}wift-{H}ohenberg equation.
\newblock {\em Phys. Rev. E (3)}, 73(5):056211, 15, 2006.

\bibitem{Coullet2000}
P.~Coullet, C.~Riera, and C.~Tresser.
\newblock Stable static localized structures in one dimension.
\newblock {\em Physical Review Letters}, 84(14):3069--3072, April 2000.

\bibitem{Pomeau1986}
Y.~Pomeau.
\newblock Front motion, metastability and subcritical bifurcations in
  hydrodynamics.
\newblock {\em Physica D: Nonlinear Phenomena}, 23(1--3):3--11, December 1986.

\bibitem{Woods1999}
P.D Woods and A.R Champneys.
\newblock Heteroclinic tangles and homoclinic snaking in the unfolding of a
  degenerate reversible hamiltonian--hopf bifurcation.
\newblock {\em Physica D: Nonlinear Phenomena}, 129(3--4):147--170, May 1999.

\bibitem{Aougab2017}
Tarik Aougab, Margaret Beck, Paul Carter, Surabhi Desai, Bj{\"o}rn Sandstede,
  Melissa Stadt, and Aric Wheeler.
\newblock Isolas versus snaking of localized rolls.
\newblock {\em Journal of Dynamics and Differential Equations},
  31(3):1199--1222, October 2017.

\bibitem{Beck2009}
Margaret Beck, J{\"u}rgen Knobloch, David J.~B. Lloyd, Bj{\"o}rn Sandstede, and
  Thomas Wagenknecht.
\newblock Snakes, ladders, and isolas of localized patterns.
\newblock {\em SIAM Journal on Mathematical Analysis}, 41(3):936--972, January
  2009.

\bibitem{Gandhi2018}
Punit Gandhi, Yuval~R. Zelnik, and Edgar Knobloch.
\newblock Spatially localized structures in the gray--scott model.
\newblock {\em Philosophical Transactions of the Royal Society A: Mathematical,
  Physical and Engineering Sciences}, 376(2135):20170375, November 2018.

\bibitem{AlSaadi2021}
Fahad Al~Saadi and Alan Champneys.
\newblock Unified framework for localized patterns in reaction--diffusion
  systems; the gray--scott and gierer--meinhardt cases.
\newblock {\em Philosophical Transactions of the Royal Society A: Mathematical,
  Physical and Engineering Sciences}, 379(2213), November 2021.

\bibitem{bergbredenlessardmireles}
Jan~Bouwe van~den Berg, Maxime Breden, Jean-Philippe Lessard, and J.~D.
  Mireles~James.
\newblock Computer-assisted proofs in nonlinear dynamics: a spectral approach
  based on fourier analysis, 2025.
\newblock In preparation.

\bibitem{Cabre2003}
Xavier Cabre, Ernest Fontich, and Rafael de~la Llave.
\newblock The parameterization method for invariant manifolds {I}: Manifolds
  associated to non-resonant subspaces.
\newblock {\em Indiana University Mathematics Journal}, 52(2):283--328, 2003.

\bibitem{Cabre2003a}
Xavier Cabre, Ernest Fontich, and Rafael de~la Llave.
\newblock The parameterization method for invariant manifolds {II}: regularity
  with respect to parameters.
\newblock {\em Indiana University Mathematics Journal}, 52(2):329--360, 2003.

\bibitem{Cabre2005}
Xavier Cabr{\'e}, Ernest Fontich, and Rafael de~la Llave.
\newblock The parameterization method for invariant manifolds {III}: overview
  and applications.
\newblock {\em Journal of Differential Equations}, 218(2):444--515, November
  2005.

\bibitem{Castelli2015}
Roberto Castelli, Jean-Philippe Lessard, and J.~D. Mireles~James.
\newblock Parameterization of invariant manifolds for periodic orbits {I}:
  Efficient numerics via the floquet normal form.
\newblock {\em SIAM Journal on Applied Dynamical Systems}, 14(1):132--167,
  January 2015.

\bibitem{Castelli2017}
Roberto Castelli, Jean-Philippe Lessard, and Jason D.~Mireles James.
\newblock Parameterization of invariant manifolds for periodic orbits {(II)}: A
  posteriori analysis and computer assisted error bounds.
\newblock {\em Journal of Dynamics and Differential Equations},
  30(4):1525--1581, August 2017.

\bibitem{Berg2018}
Jan~Bouwe van~den Berg, Maxime Breden, Jean-Philippe Lessard, and Maxime
  Murray.
\newblock Continuation of homoclinic orbits in the suspension bridge equation:
  A computer-assisted proof.
\newblock {\em Journal of Differential Equations}, 264(5):3086--3130, March
  2018.

\bibitem{Berg2010}
Jan~Bouwe van~den Berg, Jean-Philippe Lessard, and Konstantin Mischaikow.
\newblock Global smooth solution curves using rigorous branch following.
\newblock {\em Mathematics of Computation}, 79(271):1565--1584, March 2010.

\bibitem{Breden2023}
Maxime Breden.
\newblock A posteriori validation of generalized polynomial chaos expansions.
\newblock {\em SIAM Journal on Applied Dynamical Systems}, 22(2):765--801, June
  2023.

\bibitem{Yamamoto1980}
Tetsuro Yamamoto.
\newblock Error bounds for computed eigenvalues and eigenvectors.
\newblock {\em Numerische Mathematik}, 34(2):189--199, June 1980.

\bibitem{Yamamoto1982}
Tetsuro Yamamoto.
\newblock Error bounds for computed eigenvalues and eigenvectors. {II}.
\newblock {\em Numerische Mathematik}, 40(2):201--206, June 1982.

\bibitem{Rump2001}
Siegfried~M. Rump.
\newblock Computational error bounds for multiple or nearly multiple
  eigenvalues.
\newblock {\em Linear Algebra and its Applications}, 324(1--3):209--226,
  February 2001.

\bibitem{Hladik2010}
Milan Hlad{\'\i}k, David Daney, and Elias Tsigaridas.
\newblock Bounds on real eigenvalues and singular values of interval matrices.
\newblock {\em SIAM Journal on Matrix Analysis and Applications},
  31(4):2116--2129, January 2010.

\bibitem{Castelli2011}
Roberto Castelli and Jean-Philippe Lessard.
\newblock A method to rigorously enclose eigenpairs of complex interval
  matrices.
\newblock In {\em Applications of mathematics 2013}, pages 21--31. Acad. Sci.
  Czech Repub. Inst. Math., Prague, 2013.

\bibitem{sandstede_notes_2025}
Bjorn Sandstede.
\newblock Branches of localized patterned states.
\newblock arXiv:2507.12277, 2025.

\bibitem{MR3467671}
\`Alex Haro, Marta Canadell, Jordi-Llu\'is Figueras, Alejandro Luque, and
  Josep-Maria Mondelo.
\newblock {\em The parameterization method for invariant manifolds}, volume 195
  of {\em Applied Mathematical Sciences}.
\newblock Springer, [Cham], 2016.

\bibitem{MR2299977}
A.~Haro and R.~de~la Llave.
\newblock A parameterization method for the computation of invariant tori and
  their whiskers in quasi-periodic maps: explorations and mechanisms for the
  breakdown of hyperbolicity.
\newblock {\em SIAM J. Appl. Dyn. Syst.}, 6(1):142--207, 2007.

\bibitem{MR2240743}
\`A. Haro and R.~de~la Llave.
\newblock A parameterization method for the computation of invariant tori and
  their whiskers in quasi-periodic maps: numerical algorithms.
\newblock {\em Discrete Contin. Dyn. Syst. Ser. B}, 6(6):1261--1300, 2006.

\bibitem{MR2551254}
Antoni Guillamon and Gemma Huguet.
\newblock A computational and geometric approach to phase resetting curves and
  surfaces.
\newblock {\em SIAM J. Appl. Dyn. Syst.}, 8(3):1005--1042, 2009.

\bibitem{MR3118249}
Gemma Huguet and Rafael de~la Llave.
\newblock Computation of limit cycles and their isochrons: fast algorithms and
  their convergence.
\newblock {\em SIAM J. Appl. Dyn. Syst.}, 12(4):1763--1802, 2013.

\bibitem{MR2851901}
Gemma Huguet, Rafael de~la Llave, and Yannick Sire.
\newblock Computation of whiskered invariant tori and their associated
  manifolds: new fast algorithms.
\newblock {\em Discrete Contin. Dyn. Syst.}, 32(4):1309--1353, 2012.

\bibitem{MR3148084}
Jean-Philippe Lessard and Christian Reinhardt.
\newblock Rigorous numerics for nonlinear differential equations using
  {C}hebyshev series.
\newblock {\em SIAM J. Numer. Anal.}, 52(1):1--22, 2014.

\bibitem{MR4292534}
Jan~Bouwe van~den Berg and Ray Sheombarsing.
\newblock Rigorous numerics for odes using {C}hebyshev series and domain
  decomposition.
\newblock {\em J. Comput. Dyn.}, 8(3):353--401, 2021.

\bibitem{MR3437754}
Maxime Breden, Jean-Philippe Lessard, and Jason~D. Mireles~James.
\newblock Computation of maximal local (un)stable manifold patches by the
  parameterization method.
\newblock {\em Indag. Math. (N.S.)}, 27(1):340--367, 2016.

\bibitem{Chow1982}
Shui-Nee Chow and Jack~K. Hale.
\newblock {\em Methods of Bifurcation Theory}.
\newblock Springer New York, 1982.

\bibitem{Breden2015}
Maxime Breden, Jean-Philippe Lessard, and Matthieu Vanicat.
\newblock Global bifurcation diagrams of steady states of systems of {PDE}s via
  rigorous numerics: a 3-component reaction-diffusion system.
\newblock {\em Acta Appl. Math.}, 128:113--152, 2013.

\bibitem{MR3068557}
J.~D. Mireles~James and Konstantin Mischaikow.
\newblock Rigorous a posteriori computation of (un)stable manifolds and
  connecting orbits for analytic maps.
\newblock {\em SIAM J. Appl. Dyn. Syst.}, 12(2):957--1006, 2013.

\bibitem{MR3207723}
Jean-Philippe Lessard, Jason~D. Mireles~James, and Christian Reinhardt.
\newblock Computer assisted proof of transverse saddle-to-saddle connecting
  orbits for first order vector fields.
\newblock {\em J. Dynam. Differential Equations}, 26(2):267--313, 2014.

\bibitem{MR4068579}
Jan~Bouwe van~den Berg and Ray Sheombarsing.
\newblock Validated computations for connecting orbits in polynomial vector
  fields.
\newblock {\em Indag. Math. (N.S.)}, 31(2):310--373, 2020.

\bibitem{MR4658475}
Maxime Murray and J.~D. Mireles~James.
\newblock Computer assisted proof of homoclinic chaos in the spatial
  equilateral restricted four-body problem.
\newblock {\em J. Differential Equations}, 378:559--609, 2024.

\bibitem{MR2269503}
Yasuaki Hiraoka and Toshiyuki Ogawa.
\newblock An efficient estimate based on {FFT} in topological verification
  method.
\newblock {\em J. Comput. Appl. Math.}, 199(2):238--244, 2007.

\bibitem{MR3833658}
Jean-Philippe Lessard.
\newblock Computing discrete convolutions with verified accuracy via {B}anach
  algebras and the {FFT}.
\newblock {\em Appl. Math.}, 63(3):219--235, 2018.

\bibitem{MR3167726}
Jacek Cyranka.
\newblock Efficient and generic algorithm for rigorous integration forward in
  time of d{PDE}s: {P}art {I}.
\newblock {\em J. Sci. Comput.}, 59(1):28--52, 2014.

\bibitem{Lubin2023}
Miles Lubin, Oscar Dowson, Joaquim Dias~Garcia, Joey Huchette, Beno\^it Legat,
  and Juan~Pablo Vielma.
\newblock Ju{MP} 1.0: recent improvements to a modeling language for
  mathematical optimization.
\newblock {\em Math. Program. Comput.}, 15(3):581--589, 2023.

\bibitem{RadiiPolynomial}
Olivier H{\'e}not.
\newblock Radiipolynomial.jl.
\newblock https://github.com/OlivierHnt/RadiiPolynomial.jl, 2021.

\bibitem{Duchesne}
Jan~Bouwe van~den Berg, Gabriel~William Duchesne, and Jean-Philippe Lessard.
\newblock Codes to accompany the work ``{F}rom heteroclinic loops to homoclinic
  snaking in reversible systems: rigorous forcing through computer-assisted
  proofs''.
\newblock
  \url{https://github.com/GWDuchesne/CAP_Homoclinic_Snaking_and_Isolas}.

\end{thebibliography}

\end{document}